\documentclass[12pt]{amsart}
\usepackage{graphicx,epsfig,psfrag}
\usepackage{amssymb,amsfonts,amsmath,amsthm,dsfont,wasysym,pifont,stmaryrd}
\usepackage{epstopdf,yfonts}
\usepackage{hyperref}
\usepackage{psfrag}
\usepackage{wasysym,auto-pst-pdf}
\usepackage{tikz}
\usetikzlibrary{arrows,decorations.markings}

\usepackage[stable]{footmisc} 

\usepackage{pstricks, pst-node}\input xy
\usepackage[all]{xy}

\theoremstyle{plain}      

\newtheorem{theorem}{Theorem}[section]

\newtheorem{prop}[theorem]{Proposition}

\newtheorem{claim}[theorem]{Claim}
\newtheorem{cor}[theorem]{Corollary}

\newtheorem{lemma}[theorem]{Lemma}

\theoremstyle{definition} 

\newtheorem{definition}[theorem]{Definition}
 \newtheorem{remark}[theorem]{Remark}
 \newtheorem{example}[theorem]{Example}

\def\bqn{\begin{equation*}}
\def\eqn{\end{equation*}}
\def\bq{\begin{equation}}
\def\eq{\end{equation}}
\def\be{\begin{enumerate}}
\def\ee{\end{enumerate}}
\def\ba{\begin{aligned}}
\def\ea{\end{aligned}}
\def\ban{\begin{aligned*}}
\def\ean{\end{aligned*}}

\renewcommand{\1}{\mathds{1}}
\renewcommand{\(}{\left(}
\renewcommand{\)}{\right)}
\renewcommand{\[}{\left[}
\renewcommand{\]}{\right]}

\renewcommand{\~}[1]{\overline{#1}}
\renewcommand{\*}[1]{{#1}^*}
\renewcommand{\geq}{\geqslant}
\renewcommand{\leq}{\leqslant}

\renewcommand{\>}{\right\rangle}
\newcommand{\8}{\infty}

\renewcommand{\a}{\alpha}
\newcommand{\Aut}{\text{Aut}}

\renewcommand{\b}{\beta}
\newcommand{\B}{\mathcal{B}}

\newcommand{\creg}{\mathcal{C}_{reg}}
\renewcommand{\Cap}[2]{\underset{#1}{\overset{#2}{\cap} }}
\newcommand\cb{\mathrm{C}_\mathrm{b}}
\renewcommand{\Cup}[2]{\underset{#1}{\overset{#2}{\cup} }}

\newcommand{\e}{\epsilon}

\newcommand{\f}{\varphi}

\newcommand{\frakH}{\mathfrak H}

\newcommand{\g}{\gamma}
\newcommand{\G}{\Gamma}

\newcommand{\Gn}{\mathcal G}
\newcommand{\h}{\mathfrak{h}}
\newcommand{\hb}{\mathrm{H}_\mathrm{b}}
\newcommand{\hcb}{\mathrm{H}_\mathrm{cb}}

\renewcommand{\H}{\mathcal{H}}

\newcommand{\I}{\mathcal{I}}

\renewcommand{\int}{\varint}

\newcommand{\Lim}[1]{\underset{#1}{\lim}}

\newcommand{\linfty}{\mathrm{L}^\infty}

\newcommand{\linftyaw}{\mathrm{L}^\infty_{\mathrm {alt,\ast}}}

\renewcommand{\max}[1]{\underset{#1}{\mathrm{max}}}
\newcommand{\N}{\mathbb{N}}

\newcommand{\Pn}{\mbox{$\mathcal{P}$}}
\newcommand{\pri}{\operatorname{pr}_i}

\newcommand{\Q}{\mathbb{Q}}

\newcommand{\R}{\mathbb{R}}
\renewcommand{\Re}{\mathbb{R}}

\newcommand{\s}{\sigma}
\newcommand{\SL}{\textnormal{SL}}
\newcommand{\stab}{\mathrm{Stab}}

\newcommand{\supp}{\mathrm{supp}}
\newcommand{\T}{\mathcal {T}}

\newcommand{\tr}{\textsl{tr}}

\newcommand{\Z}{\mathbb{Z}}

\DeclareMathOperator{\Ker}{Ker}

\DeclareMathOperator{\Stab}{Stab}

\newcommand{\ZZ}{\mathbf{Z}}

\makeatletter
 \let\@wraptoccontribs\wraptoccontribs
\makeatother

\title[The Median Class and Superrigidity]{The Median Class and Superrigidity of Actions \\ 
  on CAT(0) Cube Complexes}
\author{Indira Chatterji, Talia Fern\'os, Alessandra Iozzi}
\contrib[with an appendix by]{Pierre-Emmanuel Caprace}
\address{Universit\'e de Nice, Laboratoire de Math\'ematiques J.A. Dieudonn\'e, 06108 Nice Cedex 02, France}
\email{\url{indira.chatterji@math.cnrs.fr}}
\address{Department of Mathematics and Statistics, University of North Carolina at Greensboro,  317 College Avenue, Greensboro, NC 27412, USA}
\email{\url{t_fernos@uncg.edu}}
\address{Department Mathematik, ETHZ, R\"amistrasse 101, CH-8092 Z\"urich, Switzerland}
\email{\url{iozzi@math.ethz.ch}}
\thanks{I.~C. was partially supported by ANR QuantiT JC08-318197; 
A.~I. was partial supported by the Swiss National Science Foundation projects 
2000021-127016/2 and 200020-144373 and by the Simons Foundation.}

\begin{document}

\begin{abstract}
We define a bounded cohomology class, called the {\em median class}, 
in the second bounded cohomology -- with appropriate coefficients --
of the automorphism group of a finite dimensional CAT(0) cube complex $X$.  
The median class of $X$ behaves naturally with respect to taking products 
and appropriate subcomplexes and defines in turn the {\em median class of an action} 
by automorphisms of $X$.

We show that the median class of a non-elementary action by automorphisms
does not vanish and we show to which extent it does vanish if the action is elementary. 
We obtain as a corollary a superrigidity result
and show for example that any irreducible lattice in the product of at least two locally compact connected groups
acts on a finite dimensional CAT(0) cube complex $X$ with a finite orbit in the Roller compactification of $X$.
In the case of a product of Lie groups, the appendix by Caprace allows us to deduce that the fixed point is in fact inside the complex $X$.

In the course of the proof we construct a $\Gamma$-equivariant
measurable map from a Poisson boundary of $\Gamma$ 
with values in the non-terminating ultrafilters on the Roller boundary of $X$.

\end{abstract}

\maketitle

\today




\section{Introduction}\label{sec:intro}
The goal of this paper is to define a cohomological invariant of some
non-positively curved metric spaces $X$ for a non-elementary action of
a group $\Gamma\to\Aut(X)$ and to use this invariant to establish
rigidity phenomena.

The paradigm is that bounded cohomology with non-trivial coefficients
is the appropriate framework to study negative curvature. The first
instance of this fact is the Gromov--Sela cocycle on the real
hyperbolic $n$-space $X$ (in fact, on any simply connected space with
pinched negative curvature) with values into the $\mathrm{L}^2$ differential
one-forms on $X$ (see \cite{Sela} and \cite[$7.{\mathrm E}_1$]{Gromov_asymptotic}).

The same philosophy has been promoted by Monod
\cite{Monod_splitting}, Monod--Shalom \cite{Monod_Shalom_1,
  Monod_Shalom_2} and Mineyev--Monod--Shalom
\cite{Mineyev_Monod_Shalom}.  They prove that a non-elementary
isometric action on a negatively curved space (belonging to a very
rich class) yields the non-vanishing of second bounded cohomology
with appropriately defined coefficients of a geometric nature.  Such
negatively curved spaces include proper CAT(-1) spaces, Gromov
hyperbolic graphs of bounded valency, Gromov hyperbolic proper
cocompact geodesic metric spaces, or simplicial trees.

On the other hand, if $G$ is a simple Lie group with rank at least two
and $\H$ is any unitary representation with no invariant vectors, then
$\hcb^2(G,\H)=0$, \cite{Burger_Monod_JEMS, Burger_Monod_GAFA}, thus
showing that in non-positive curvature the situation cannot be
expected to be completely analogous.

In this paper we move away from the negative curvature case and look
at actions on CAT(0) cube complexes.

CAT(0) cube complexes are simply connected combinatorial objects
introduced by Gromov in \cite{Gromov_hyperbolic}.  They have been used
is several important contexts, such as Moussong's 
characterization of word hyperbolic Coxeter groups
in terms of their natural presentation, \cite{Moussong}.  
A prominent use of CAT(0) cube complexes was made
by Sageev in his thesis, \cite{Sageev_95}: generalizing
Stalling's theorem on the equivalence between splittings of groups and
actions on trees, \cite{Scott, Serre_amalgames, Serre_trees},
he proved an equivalence between
the existence of an action of a group $\Gamma$ on a CAT(0) cube
complex and the existence of a subgroup $\Lambda<\Gamma$ such that the
pair $(\Gamma,\Lambda)$ has more than one end.  More
recently, Agol's proof of the last standing conjecture in 3-manifolds,
the virtual Haken conjecture, uses (special) cube complexes in
a crucial way, thus indisputably asserting their relevance in the
mathematical scenery.

The first example of a CAT(0) cube complex $X$ is a simplicial
tree; the midpoint of a vertex is the analogue of a {\em hyperplane}
for a general CAT(0) cube complex.  Hyperplanes separate $X$ into two
connected components, called {\em halfspaces}, the collection of
which is denoted by $\frakH(X)$.  If the vertex set of $X$ is locally
countable then $\frakH(X)$ is countable as well.
 
A CAT(0) cube complex is in particular a {\em median space}, that is, given
any three vertices, there is a unique vertex, the {\em median},
that is on the combinatorial geodesics joining any two of the three points.  
For $n\geq2$ let $\frakH(X)^n$ denote the set of $n$-tuples of halfspaces in $X$. 
If $1\leq p<\infty$, we define a one-parameter family of $\Aut(X)$-invariant cocycles
$$c_{(n,R)}:{X}\times {X}\times {X}\to \ell^p(\frakH(X)^n)$$
as the sum of the
characteristic functions of some appropriate finite subsets of nested
halfspaces (called \emph{\"uber-parallel}, see Definitions~\ref{sss} and~\ref{uparallel}) ``around'' the median of three points and at distance\footnote{This is not a distance but just a pseudo-distance on the set of hyperplanes, and will be discussed more in Section~\ref{bridge}.} less than $R$.
Choosing a basepoint $v_0\in X$ and evaluating $c$ on an $\Aut(X)$-orbit,
we get what we call a {\em median cocycle} on 
$\Aut(X)\times\Aut(X)\times\Aut(X)$. 
We then prove that, for every $n\geq2$ and $R\geq 0$, the cocycle so defined is bounded and hence
defines a bounded cohomology class ${\tt m}_{(n,R)}(X)$ in degree two, 
that we call a {\em median class\footnote{For any $n\geq2$ there is a median class, 
but in the following we will not necessarily make a distinction of the various median 
classes for different $n$.} of $X$}.
(See \eqref{cocycle}, Proposition~\ref{prop:bounded1} and Lemma~\ref{lem:sets} 
for the precise definition and the proof of the above statements.)

If $\rho:\Gamma\to\Aut(X)$ is an action of a group $\Gamma$ by automorphisms on $X$,
the {\em median class of the $\Gamma$-action} is the pull-back 
\bqn
\rho^*({\tt m}_{(n,R)}(X))\in\hb^2(\Gamma, \ell^p(\frakH(X)^n))\,.
\eqn

%
\begin{theorem}\label{thm_intro:main}
Let $X$ be a finite dimensional CAT(0) cube complex with a $\Gamma$-action. 
If the $\G$-action is non-elementary, then there is an $R_\G\geq 1$ so that 
the median class of the $\G$-action $\rho^*({\tt m}_{(n,R)}(X))$
does not vanish for all $n\geq2$ and all $R\geq R_\G$.
\end{theorem}
We call an action $\Gamma\to\Aut(X)$ {\em non-elementary} 
if there is no finite orbit in $X \sqcup \partial_\sphericalangle X$, 
where $ \partial_\sphericalangle X$ denotes the visual boundary of $X$
as a CAT(0) space.

Let us say a word about what it means for a $\G$-action to be non-elementary 
in the context of CAT(0) cube complexes. 
First of all, the assumption implies in particular that,
by passing to a subgroup of finite index,  there are
no $\G$-fixed points in $\partial_\sphericalangle X$. 
Under this hypothesis, using the work of Caprace--Sageev 
\cite[Proposition~3.5]{Caprace_Sageev}, 
one can pass to a nonempty convex subset of $X$, 
called the {\em $\G$-essential core}, (see \S~\ref{subsec:essential}) 
which will have rather nice dynamic properties. 
Furthermore, the exclusion of a finite orbit on $\partial_\sphericalangle X$ 
excludes the existence of a Euclidean factor in the essential core 
(see Corollary~\ref{no Euclidean factors}).

A key object in this paper is the {\em Roller boundary} $\partial X$
of a CAT(0) cube complex, defined in \S~\ref{subsec:2.1}. It arises
naturally from considering the hyperplane (and hence halfspace)
structure of $X$.  The vertex set of $X$, together with its Roller boundary,
can be thought of as a closed subset of a Bernoulli space (with
$\frakH(X)$ as the indexing set) and is hence compact and totally
disconnected. Although in the case of a tree the Roller boundary and
the visual boundary coincide, we remark that in general, there is no
natural map between them\footnote{There is a map from the CAT(0)
boundary to a quotient of the Roller boundary \cite{Guralnik}, but
we will not use it in this paper.}. 
In Proposition~\ref{prop:visual-to-roller} we prove nevertheless a result relating,
to the extent to which it is possible, finite orbits in the Roller boundary
to finite orbits in the CAT(0) boundary.  
The dichotomy that we obtain is analogous to the one in the case
of a group $\Gamma$ acting on a symmetric space $\mathcal X$ of non-compact type;
in this case, if $\Gamma$ fixes a point at infinity in the CAT(0) boundary of $\mathcal X$,
the image of $\Gamma$ is contained in a parabolic subgroup,
via which it acts on the symmetric space of non-compact type associated to the 
semisimple component of the parabolic.  The latter action may very well
be non-elementary.

The dichotomy in Proposition~\ref{prop:visual-to-roller} 
leads to the following converse of Theorem~\ref{thm_intro:main}:

\begin{theorem}\label{thm:converse}  Let $X$ be a finite dimensional CAT(0) cube complex
with an elementary $\Gamma$-action.  Then:
\be
\item either there is a finite orbit in the Roller compactification $\~X=X\cup\partial X$ 
of $X$ and hence the median class $\rho^*({\tt m}_{(n,R)}(X))$ of the $\Gamma$-action on $X$ vanishes
for all $n\geq2$ and all $R\geq 1$;
\item or there exists a finite index subgroup $\Gamma'<\Gamma$ and 
a $\Gamma'$-invariant subcomplex $X'\subset\partial X$
(of lower dimension) on which the $\Gamma'$-action is non-elementary.
In this case any median class $\rho^*({\tt m}_{(n,R)}(X))$, $n\geq2$ and $R\geq 1$, 
of the $\Gamma$-action on $X$ restricts 
to a median class of the $\Gamma'$-action on $X'$. In particular, $R_\G=R_{\G'}$.
\ee
\end{theorem}
We say that an action of a group $\G$ on $X$ is \emph{Roller elementary} if it has a finite orbit on the Roller compactification (that is, $X$ union its Roller boundary). Combining the above theorem with Theorem~\ref{thm_intro:main} we get the following formulation.
\begin{theorem}Let $X$ be a finite dimensional CAT(0) cube complex. A $\G$-action on $X$ is Roller elementary if and only if the median class $\rho^*({\tt m}_{(n,R)}(X))$ vanishes for some (equivalently, any) $n\geq2$ and all $R\geq 1$.
\end{theorem}
One of the nice features of the Roller boundary is its robustness when
considering products.  Because of this, the median cocycle 
can be defined for
each irreducible factor of the essential core of $X$.  We refer the
reader to Proposition~\ref{prop:decomposition} for a description
of the cocycle in the case in which the CAT(0) cube complex is not
irreducible and hence of the naturality of the behavior of the median class with respect 
to products.  This, together with Theorem~\ref{thm_intro:main} yields 
immediately the following:

\begin{cor}\label{cor:dim} Let $X$ be a finite dimensional CAT(0) cube complex with a
non-elementary action $\Gamma\to\Aut(X)$. 
Then for all $n\geq2$ and $1\leq p<\infty$
\begin{equation*}
\dim\hb^2(\Gamma,\ell^p(\frakH(X)^n))\geq m\,,
\end{equation*}
where $m\geq1$ is the number of irreducible factors in the 
essential core of the $\Gamma$-action on $X$.
\end{cor}
This result might not be sharp, in the sense that 
$\hb^2(\Gamma,\ell^p(\frakH(X)^n))$ could be 
in some cases infinite dimensional.

On a similar vein Bestvina--Bromberg--Fujiwara have proven 
the non-vanishing of the second bounded cohomology 
with general uniformly convex Banach spaces as coefficients
and for weakly properly discontinuous actions on CAT(0) spaces
in the presence of a rank one isometry (that is an isometry whose axis does
not bound a half flat), \cite{Bestvina_Bromberg_Fujiwara}.  

Our results are different in that in Theorem~\ref{thm_intro:main}, 
we are neither assuming that the action of $\Gamma\to\Aut(X)$
is proper or weakly properly discontinuous, 
nor that the CAT(0) cube complex is proper or has
a cocompact group of automorphisms.

Moreover in the CAT(0) cube complex is a product, then
there are no rank one isometries. Caprace--Sageev proved
\cite{Caprace_Sageev} that there is always a decomposition of a CAT(0)
cube complex analogous to the decomposition of symmetric spaces into
``irreducible" (or ``rank one") factors.  Our result is not sensitive
to this decomposition and hence also applies to products.

But, more than anything else, we want to emphasize that
the existence of a well behaved and concrete bounded cohomological class
goes well beyond the mere knowledge that the bounded cohomology group
does not vanish and is the starting point of a wealth of rigidity results
(see for example \cite{Milnor, Goldman_thesis, Goldman_82, Toledo_89, Hernandez, Matsumoto, Iozzi, Monod_Shalom_2, Mineyev_Monod_Shalom, Burger_Iozzi_supq, Burger_Iozzi_tr, Burger_Iozzi_Wienhard_kahler, Burger_Iozzi_Wienhard_tight, Burger_Iozzi_Wienhard_toledo, Burger_Iozzi_Labourie_Wienhard, Bucher_Burger_Iozzi_mostow, Bucher_Burger_Iozzi_slnc}).

Furthermore our coefficients reflect geometric properties of the CAT(0) 
cube complex, and this is essential to draw
conclusions about the action.  An example of this is the following
superrigidity result:

\begin{theorem}[Superrigidity]\label{thm:appl_intro}
Let $Y$ be an irreducible finite dimensional CAT(0) cube complex
and $\Gamma<G_1\times\dots\times G_\ell=:G$ an irreducible lattice in the 
product of $\ell\geq2$ locally compact groups.
Let $\Gamma\to\Aut(Y)$  be an essential and non-elementary action on $Y$.
Then the action of $\Gamma$ on $Y$ extends continuously to an action of $G$, 
by factoring via one of the factors.
\end{theorem}

Here the group $\Aut(Y)$ is a topological group 
endowed with the topology of the pointwise convergence on vertices.
This theorem is proven in \S~\ref{sec:appl}, to which we refer the reader 
also for an analogous result 
that does not require $Y$ to be irreducible and 
the action to be essential.

We remark that requiring that the action is essential
is necessary if one wants an irreducible CAT(0) cube complex,
as there is no guarantee that the essential core will be
irreducible even when $X$ is.

A result similar to Theorem~\ref{thm:appl_intro} 
was proven by Monod \cite[Theorems~6 and 7]{Monod_splitting} 
(see also \cite[Corollary~1.9]{Caprace_Lytchak})
in the case of an infinite dimensional CAT(0) space, with 
conditions both on the action and on the lattice $\Gamma$.
%
For example, if $\Gamma$ is not uniform, 
then in order to apply Monod's version of
Theorem~\ref{thm:appl_intro}, $\Gamma$ has to be square-integrable and
weakly cocompact.  Although these conditions are verified for a large
class of groups (such as for example Kazhdan Kac--Moody lattices and
lattices in connected semisimple Lie groups), they are in general
rather intractable.  To give a sense of this, let us only remark that
already finite generation (needed for example for square integrability)
is not known for a lattice $\Gamma<\Aut(\T_1)\times\Aut(\T_2)$, not
even by imposing strong conditions on the closure of the projections
on $\Gamma$ in $\Aut(\T_i)$ to insure irreducibility.
Furthermore the more specific nature of
a CAT(0) cube complex versus a CAT(0) space allows us to extend 
the action to the whole complex.

As an illustration we have the following corollary: 

\begin{cor} \label{cor:FiniteOrbit}
Let $\Gamma$ be an irreducible lattice in the product 
$G:=G_1\times\dots\times G_\ell$ of $\ell\geq2$  locally compact groups 
with a finite number of connected components.   
Then any $\Gamma$-action on a finite dimensional CAT(0) cube complex is elementary
and has a finite orbit in the Roller compactification $\overline X = X \cup \partial X$.
\end{cor}

\begin{proof}
Indeed, since $G$ has finitely many connected components and 
for any CAT(0) cube complex $X$ the group $\Aut(X)$
is totally disconnected, a continuous
map from $G$ to $\Aut(X)$ must have finite image.  In view of Theorem~\ref{thm:appl_intro} (in fact, more precisely of Corollary~\ref{thm:appl}),
 this implies that every $\Gamma$-action on a finite-dimensional CAT(0) cube complex is elementary. 
 Notice that every finite index subgroup $\Gamma' \leq \Gamma$ is itself a lattice in $G$. 
 Moreover, the closure of the projection of $\Gamma'$ to each factor $G_i$ is a closed subgroup of finite index in $G_i$. 
 It is thus open, and therefore contains the connected component of the identity of $G_i$. 
 This shows that $\Gamma'$ is itself an irreducible lattice in the product of  $\ell\geq2$  locally compact groups 
with a finite number of connected components.  
By Theorem~\ref{thm:converse}, this implies that every $\Gamma'$-action, 
and thus also every $\Gamma$-action, has a finite orbit in the Roller compactification of $X$. 
\end{proof}


Combining Corollary~\ref{cor:FiniteOrbit} with a description of the structure of a point stabiliser in the Roller boundary, established by Pierre-Emmanuel Caprace in Appendix~\ref{sec:AppB}, one obtains the following Fixed Point property for lattices in semisimple groups, which was pointed out to us by him.

\begin{cor}\label{cor:Lie}
Let $\Gamma$ be an irreducible lattice in a semisimple Lie group of rank at least~$2$. Then every $\Gamma$-action on a finite-dimensional CAT(0) cube complex $X$ has a fixed point.
\end{cor}

\begin{proof}
If the semisimple Lie group has only one simple factor, then $\Gamma$ has property (T) and the desired conclusion is well known, see \cite{Niblo_Reeves}. Otherwise, we apply  Corollary~\ref{cor:Rigidity} from Appendix~\ref{sec:AppB} below: 
Condition (a) holds as a consequence of Margulis' Normal Subgroup Theorem, while  Condition (b) holds by Corollary~\ref{cor:FiniteOrbit}. 
\end{proof}

It is conjectured that the conclusion of Corollary~\ref{cor:Lie} holds without the hypothesis that $X$ is  finite-dimensional, see \cite{Cor}. In fact, Yves de Cornulier shows in  \cite{Cor} that this is indeed the case provided the ambient semisimple Lie group has at least one simple factor of rank at least~$2$.


\medskip 
On a different tone, recall that the concept of
measure equivalence was introduced by Gromov as a measure theoretical
counterpart of quasi-isometries.  The vanishing or non-vanishing of
bounded cohomology is not invariant under quasi-isometries (see
\cite[Corollary 1.7]{Burger_Monod_JEMS}); on the other hand,
Monod--Shalom proved that vanishing of bounded cohomology with
coefficients in the regular representation is invariant under measure
equivalence \cite{Monod_Shalom_1} and hence introduced a class of
groups $\creg:=\{\G:\,\hb^2(\G,\ell^2(\G))\neq0\}$.  They also proved
for example that if $\Gamma\in\creg$ and $\Gamma\times\Gamma$ is
measure equivalent to $\Lambda$, then $\Gamma\times\Gamma$ and
$\Lambda$ are commensurable.  We can
add to the groups in this list:
\begin{cor}\label{cor:creg_intro}  Let $\Gamma$ be a group acting
on a finite dimensional irreducible CAT(0) cube complex. 
If the action is metrically proper, non-elementary and essential,
then $\hb^2(\Gamma,\ell^p(\Gamma))\neq0$, for $1\leq p<\infty$, 
and hence in particular $\Gamma\in\creg$.
\end{cor}

We remark that the same result does not hold if $X$ is not irreducible.  In fact, it can be easily seen 
using \cite[Theorem~16]{Burger_Monod_GAFA}, that if $\Gamma<G_1\times G_2$
is an irreducible lattice in the product of locally compact groups,
then $\hb^2(\Gamma,\ell^p(\Gamma))=0$, provided $G_1$ and $G_2$ are not compact.
An example of such a group is any irreducible lattice $\Gamma$ in 
$\SL(2,\Q_p)\times\SL(2,\Q_q)$, while it is easy to see that  
it acts non-elementarily and essentially
on the product of two regular trees $\mathcal T_{p+1}\times\mathcal T_{q+1}$.

A result similar to Corollary~\ref{cor:creg_intro} has been proven by Hamenst\"adt
in the case of a group $\Gamma$ acting properly on a proper CAT(0) space,
also under the assumption that there exists a rank one isometry and that the group 
$\Gamma$ is closed in the isometry group of $X$,  \cite{Hamenstaedt_iso}.
Similarly, Hull and Osin proved that every group $\Gamma$ 
with a sufficiently nice hyperbolic subgroup has infinite dimensional
$\hb^2(\Gamma,\ell^p(\Gamma))$, for $1\leq p<\infty$, \cite{Hull_Osin}.
Examples of groups satisfying such condition encompass, among others,
groups $\Gamma$ acting properly on a proper CAT(0) space with a rank one isometry
and groups $\Gamma$ acting on a hyperbolic space also with a rank one isometry
and containing a loxodromic element satisfying the Bestvina--Fujiwara ``weakly
properly discontinuous'' condition.
We emphasize that our CAT(0) cube complex are allowed to be locally countable
and that, again, irreducibility is equivalent to the
existence of a rank one isometry, \cite{Caprace_Sageev}.
%
%
\medskip

The proof of Theorem~\ref{thm_intro:main} uses the functorial approach to bounded cohomology
developed in \cite{Burger_Monod_GAFA, Monod_book, Burger_Iozzi_app};
the main point here is to be able to realize bounded cohomology 
via essentially bounded alternating cocycles on a {\em strong  $\Gamma$-boundary}.  
Recall from \cite{Kaimanovich} that a strong $\Gamma$-boundary
is a Lebesgue space $(B,\vartheta)$ endowed with a measure class preserving
$\Gamma$-action that is in addition amenable and doubly ergodic ``with coefficients''
(see \cite{Kaimanovich} for the precise definition or \S~\ref{subsec:towards}).  
An example of a strong $\Gamma$-boundary is the Poisson
boundary of any spread out non-degenerate symmetric probability measure on $\Gamma$,
\cite{Kaimanovich}.
The advantage of the approach using a strong $\Gamma$-boundary 
is that the second bounded cohomology is not a quotient anymore
(hence allowing to
determine easily when a cocycle defines a non-trivial class); the disadvantage
is that the pull-back via a representation has to be realized by a boundary map
(with consequent technical difficulties, \cite{Burger_Iozzi_app}).
The amenability of a strong $\Gamma$-boundary implies immediately the existence of a boundary map
into probability measures on the Roller compactification of $X$, but going from 
probability measures to Dirac masses is often the sore point of many rigidity questions.
In the case of a proper CAT(0) cube complex and a cocompact group of isometries in 
$\Aut(X)$, Nevo--Sageev identified the closure of the set of non-terminating ultrafilters
(see \S~\ref{subsec:2.1} for the definition) as a metric model for a Poisson boundary of $\Gamma$,
\cite{Nevo_Sageev}.  
In this case, the boundary map could have been taken simply to be the identity.  
In general we have the following:

\begin{theorem}\label{thm:boundary}
Let $\G \to \Aut(X)$ be a non-elementary group action
on a finite dimensional CAT(0) cube complex $X$.
If $(B,\vartheta)$ is a strong $\Gamma$-boundary, 
there exists a $\Gamma$-equivariant measurable map $\f : B \to \partial X$.
\end{theorem}

In fact, one can obtain something a bit more precise, namely that the
boundary map takes values in the non-terminating ultrafilters of the 
$\Gamma$-essential core of $X$ (see Theorem~\ref{thm:bdrymap} and
Corollary~\ref{cor:bdrymap}).
To prove Theorem~\ref{thm:boundary}, we develop some methods that take inspiration 
from \cite[Proposition~3.3]{Monod_Shalom_2} in the case of a simplicial tree
but are considerably more involved in the case of a CAT(0) cube complex due to the lack
of hyperbolicity. 

The first step in the identification of a Poisson boundary in
\cite{Nevo_Sageev} is the proof that the set of non-terminating
ultrafilters is not empty, under the assumption that the action is
essential and the CAT(0) cube complex is cocompact.  The same
assertion with the cocompactness of $X$ replaced by the non-existence
of $\Aut(X)$-fixed points in the CAT(0) boundary follows from our
proof that the boundary map takes values into the set of the
non-terminating ultrafilters.

\begin{cor}\label{cor:non-terminating ultrafilters} 
Let $Y$ be a finite dimensional CAT(0) cube complex such that $\Aut(Y)$ acts essentially
and without fixed points in $\partial_\sphericalangle Y$.  Then the set of non-terminating 
ultrafilters in $\partial Y$ is not empty.
\end{cor}

The structure of the paper is as follows.  In \S~\ref{sec:prelim} we
recall the appropriate definitions and fix the terminology of CAT(0)
cube complexes; we establish moreover some basic results needed in the
paper, by pushing a bit further than what was available in the
literature; the knowledgeable reader should have no problem 
parsing through the subsections.  In \S~\ref{sec:cocycle} we
construct the cocycle on the Roller compactification of the CAT(0)
cube complex $X$ and show that it is bounded. We conclude the section with an outlook
on the proof of the non-vanishing of a median class.
The boundary map and Theorem~\ref{thm:boundary}
are discussed in \S~\ref{sec:bdrymap}.
We prove Theorem~\ref{thm_intro:main} and Theorem~\ref{thm:converse}  
in \S~\ref{ProofMainThm}, Corollary~\ref{cor:dim} is a consequence of
Theorem~\ref{thm_intro:main} and Proposition~\ref{prop:decomposition},
while Theorem~\ref{thm:appl_intro} and Corollary~\ref{cor:creg_intro} are
proven in \S~\ref{sec:appl}.

\medskip
\noindent
{\em Acknowledgments:}  
We thank the Forschungsinstitute f\"ur Mathematik at ETH, Z\"urich, the Institut
Mittag-Leffler in Djursholm, Sweden, the Institut Henri Poincar\'e
in Paris and the Mathematical Sciences Research Institute in Berkeley, CA
for their hospitality.  Part of this research was conducted while 
Indira Chatterji was at Jawaharlal Nehru University as a visiting NBHM professor, 
and she thanks them for their hospitality.  Our thanks go also to Marc Burger,
Ruth Charney and Nicolas Monod 
for useful discussions during various phases of the preparation of
this paper. We are very thankful to Michah Sageev
for having suggested a particular case of Lemma~\ref{lem:embedding} 
that lead to the proof of Theorem~\ref{thm:boundary}
and to Pierre-Emmanuel Caprace for pointing out Corollary~\ref{cor:Lie} to us.
We are grateful to Max Forester and Jing Tao for pointing out a mistake in an earlier version. 
Precisely, in an earlier version we used the concept of \emph{tightly nested $n$-tuples of halfspaces}, 
but the cocycle obtained that way failed to be bounded in some cases, as for instance in Example~\ref{blow}.

\section{Preliminaries and Basic Results}\label{sec:prelim}

\subsection{Generalities on CAT(0) Cube Complexes, Hyperplanes, Duality and Boundaries}\label{subsec:2.1}
A {\em cube complex} $X$ is a metric polyhedral complex with cells
isomorphic to $[0,1]^n$ and isometries $\f_j:[0,1]^j\to X$ as gluing
maps.  The cube complex is CAT(0) if it is non-positively curved with
the induced Euclidean metric and has \emph{finite dimension $D$} if
the $m$-dimensional skeleton $X^m$ of $X$ is empty for $m>D$ and
nonempty for $m=D$.  We always assume our cube complexes to be finite
dimensional.  A cube complex $X$ is CAT(0) if and only if it is both
simply connected and the link of every vertex is a flag complex:
recall that a flag complex is a simplicial complex such that any
$(n+1)$-vertices that are pairwise connected by an edge actually span
an $n$-simplex, \cite[Theorem II.5.20]{Bridson_Haefliger}.  
A {\em  combinatorial isometry} between two CAT(0) cube complexes is a
homeomorphism $f:X\to Y$ such that the composition 
$f\circ\f_j:[0,1]^j\to Y$ is an isometry into a cube of $Y$.  Note that any
combinatorial isometry preserves also the CAT(0) metric.  We denote by
$\Aut(X)$ the group of combinatorial isometries from $X$ to itself.

Given a finite dimensional cube complex $X$, we can define an
equivalence relation on edges, generated by the condition that two
edges are equivalent if they are opposite sides of the same square
(i.e. a $2$-cube).  A {\em midcube} of an $n$-cube $\sigma$ with
respect to the above equivalence relation is the convex hull of the
set of midpoints of elements in the equivalence relation.
A {\em hyperplane} is the union of the midcubes that intersect the
edges in an equivalence class.  So a hyperplane is a closed convex
subspace and it defines uniquely two {\em halfspaces}, that is the
two complementary connected components.  On the countable collection
$\frakH(X)$ -- or simply $\frakH$, when no confusion arises -- of halfspaces
on $X$ defined by the hyperplanes, one can define a
fixed-point-free involution
\begin{equation}\label{eq:involution}
\begin{aligned}
*:\frakH&\longrightarrow\qquad\frakH\\
h\,&\mapsto \quad\*h:=X\smallsetminus h\,,
\end{aligned}
\end{equation}
so that a hyperplane is the geometric realization of a pair $\{h,\*h\}$.
In the following we identify  the hyperplane $\hat h$ with the pair of halfspaces
$\{h,\*h\}$ that it defines.  We denote by $\hat\frakH(X)$ 
the set of hyperplanes.

We say that two halfspaces $h, k$ are {\em transverse}, 
and we write $h\pitchfork k$, if all the intersections
\begin{equation}\label{eq:transverse}
h\cap k\,,\qquad h\cap \*k\,,\qquad \*h\cap k\,,\qquad \*h\cap \*k
\end{equation}
are not empty.  
Two halfspaces $ h,  k$ are {\em parallel}, 
and we write $ h\parallel  k$,  if they are not transverse,
equivalently if (exactly) one of the following relations
\begin{equation}\label{eq:nested-halfspaces}
h\subset\*k\,,\qquad h\subset k\,,\qquad \*h\subset\*k\,,\qquad \*h\subset k
\end{equation}
holds; two parallel halfspaces $h$ and $k$ are said to be
\emph{facing} if $h \subset k^*$.  We say that two hyperplanes 
$\hat h, \hat k$ are {\em transverse} (respectively {\em parallel}) if some
(and hence any) choice of corresponding halfspaces $h$ and $k$ is
transverse (respectively parallel).  Finally we say that two points
$u$ and $v$ are {\em separated by} a halfspace $h$ (or a hyperplane
$\hat h=\{h,\*h\}$) if $u\in h$ and $v\in\*h$ (or vice-versa).

Two halfspaces $h,k$ are said to be {\em nested} 
if either $h\subset k$ or $k\subset h$. A subset of hyperplanes is {\em transverse} (respectively {\em parallel}) 
if all of its elements are pairwise transverse (respectively parallel).

Recall that a family of pairwise transverse hyperplanes must have a common 
intersection (\cite{Sageev_95} or \cite{Rolli_notes}).  
We can think of the dimension  
of a CAT(0) cube complex as the largest cardinality of 
a family of pairwise  transverse hyperplanes, because such a maximal 
intersection defines a cube of maximal dimension.

Given a subset $\alpha\subset\frakH$ of halfspaces, 
we denote by $\*\alpha$ the set $\{\*h:\,h\in\alpha\}$.
\begin{definition}\label{def:partial choice-consistency}
We say that a subset $\alpha\subset\frakH$ of halfspaces satisfies:
%
\begin{enumerate}
\item[(i)] the {\em partial choice} condition, if  $\alpha\cap\*\alpha=\varnothing$,
that is if, whenever $h\in\alpha$, then $\*h\notin\alpha$;
\item[(ii)] the {\em choice} condition if $\alpha\cap\*\alpha=\varnothing$ and $\alpha\sqcup\*\alpha=\frakH$;
\item[(iii)] the {\em consistency} condition if, whenever $h\in\alpha$ and $h\subset k$,
then $k\in\alpha$.
\end{enumerate}
\end{definition}

Then an {\em ultrafilter}\footnote{We point out that the notion of 
ultrafilter used in the theory of
CAT(0) cube complexes is slightly off from the classical one in set
theory and topology (see for example \cite{Comfort_Negrepontis}).  In
fact, in the context of CAT(0) cube complexes, subsets of ultrafilters
are never ultrafilters and thus, in particular, the intersection of two
ultrafilters is never an ultrafilter.} on $\frakH$ is a subset of $\frakH$
that satisfies the choice and consistency properties.
In other words, an ultrafilter on $\frakH$ is a choice of a halfspace
for each hyperplane in $X$ with the condition that as soon as a halfspace
is in the ultrafilter, any halfspace containing it must also be
in the ultrafilter.  
We call {\em partially defined ultrafilter} a subset $\alpha\subset\frakH$
that satisfies the partial choice and consistency properties.

We say that an ultrafilter satisfies the {\em Descending Chain
  Condition (DCC)} if every descending chain of halfspaces
terminates.  Such ultrafilters are called {\em principal} and are in
one-to-one correspondence with the vertices of the CAT(0) cube complex
$X$, \cite{Guralnik}.  By abuse of notation, we do not usually make a
distinction between $X$, its vertex set, or the collection of
principal ultrafilters.

The consideration of $X$ as a collection of ultrafilters leads in a
natural way to an inclusion of $X$ into the Bernoulli space
$2^\mathfrak H$, where $v \mapsto \{h\in\frakH: v\in h\}$. This
justifies a further (standard) abuse: thinking of $X \subset 2^\frakH$,
by duality we get that $h\in v$ if and only if $v\in h$
and we can hence write
$v=\bigcap_{h\in v}h$.  Let
$\~X$ be the closure of $X$ in $2^\mathfrak H$.  One can check that
the elements of $\~X$, thought of as subsets of $\mathfrak H$, are
ultrafilters.

The correspondence that associates to an ultrafilter a
  vertex in $\~X$ can be pushed further to give a duality between
  finite dimensional CAT(0) cube complexes and those {\em pocsets} that
  satisfy both the {\em finite interval condition} and the 
  {\em finite width condition}.  Recall that a {\em pocset} $\Sigma$ is a
  partially ordered sets with an order reversing involution.    
  The pocset satisfies the 
  {\em finite interval condition} if for every pair $\alpha,\beta\in\Sigma$ 
  with $\alpha\subset\beta$ there are only finitely many $\gamma\in\Sigma$
  such that $\alpha\subset\gamma\subset\beta$; moreover if satisfies
  the {\em finite width condition} if there is an upper bound on the
  size of a collection of incomparable elements.  Given a
  pocset $\Sigma$, one can consider the space of ultrafilters on $\Sigma$.
  The CAT(0) cube complex $X(\Sigma)$ corresponding to the pocset $\Sigma$ 
  has the principal ultrafilters as vertices, 
  edges joining ultrafilters that differ only in the
  assignment on one element in $\Sigma$ and cubes attached to the 
  $1$-skeleton whenever it is possible.

  The set of halfspaces $\frakH(X)$ in a CAT(0) cube complex is a
  pocset with the above properties and the CAT(0) cube complex
  obtained with the above construction from the set of principal
  ultrafilters on $\frakH(X)$ is exactly $X$.

The boundary $\partial X:=\~X\smallsetminus X$ is called the 
{\em Roller boundary}, and consists of all ultrafilters that are not principal,
\cite{Roller}.  The compact set $\~X$ is called the {\em Roller compactification}.

An ultrafilter $v\in\partial X$ is said to be {\em non-terminating}
if every finite descending chain can be extended, 
i.e. if given a finite collection $\{h_0, \dots, h_N\}\subset v$ 
such that  $h_0 \supset \cdots \supset h_N$ 
there is an  $h_{N+1} \in v$ such that $h_{N+1} \subset h_N \subset \cdots \subset h_0$.


While the Roller boundary $\partial X$ is not empty if the CAT(0) cube
complex is unbounded, it is unclear as to when the set of
non-terminating ultrafilters is not empty.  However one can impose
reasonable conditions which do guarantee that these exist (see
\cite{Nevo_Sageev} and \S~\ref{sec:bdrymap}).  In the case of a tree,
the entire Roller boundary consists of non-terminating ultrafilters;
in the case of a $\Z^D$, there are only $2^D$-many
non-terminating ultrafilters, while there are examples, such as the
wedge of two strips \cite[Remark~3.2]{Nevo_Sageev}, in which 
the set of non-terminating ultrafilter is empty.

A CAT(0) cube complex has of course also a visual boundary
$\partial_\sphericalangle X$ with respect to its CAT(0)
metric, \cite{Bridson_Haefliger}. We recall that $\partial_\sphericalangle X$
is the set of endpoints of geodesic rays in $X$, 
where we identify two geodesic rays if they stay at bounded distance
from each other. 

\subsection{Intervals and Median Structure}\label{subsec:intervals-median}
If $u,v\in X$, their {\em combinatorial distance} $d(u,v)$ is the number of hyperplanes by which 
the two corresponding ultrafilters differ.
We will call a sequence of points $u=x_0,\dots,x_n=v$ a \emph{combinatorial geodesic} if $d(x_i,x_{i+1})=1$ for $i=0,\dots,n-1$ and
$$d(x_i,x_j)+d(x_j,x_k)=d(x_i,x_k)$$
for all $1\leq i\leq j\leq k\leq n$.  Hence the combinatorial distance corresponds to the 
graph metric on the 1-skeleton of $X$.  The oriented {\em interval} of halfspaces
\bqn
[u,v] = \{ h \in \frakH: h \in  {v}\smallsetminus  {u}\}
\eqn
is the (finite) set of halfspaces containing $v$ and not $u$
and the counting measure on $\frakH$ is consistent with the combinatorial metric $d$ on $X$ 
in that $\big| [u,v] \big| = d(u,v)$.  
Notice that on $[u,v]$ there is a partial order given by the inclusion.
It is immediate to check that
\bqn
[u,v]=[v,u]^\ast\,,
\eqn 
where we recall that $[v,u]^\ast=\{h\in\frakH:\,\*h\in[v,u]\}$. 

Combinatorial geodesics and oriented intervals are related as follows.
\begin{lemma}\label{geod}For each $u, v \in X$, combinatorial geodesics between $u$ and $v$ are in one to one correspondance with enumerations of the elements of $[u,v]$ that are compatible with the (reverse) partial order given by inclusions.\end{lemma}
\begin{proof}Let $u=x_0,\dots,x_n=v$ be a combinatorial geodesic. One obtains an enumeration $h_1,\dots,h_n$ of elements of  $[u,v]$ by setting $h_i$ to be the halfspace corresponding to the unique oriented edge between $x_i$ and $x_{i+1}$. Let us now show that this order is consistent with the (reverse) partial order. Indeed, suppose that $i<j$. If $h_i\pitchfork h_j$ then they are incomparable and there is nothing to check. Otherwise, observe that $x_i \in h_i\setminus h_j$ as $x_i$ is obtained by starting at $u$ and crossing each of the elements in $\{h_1, \dots, h_{i}\}$ (which does not contain $h_j$). Therefore, $h_j \subset h_i$.

Conversely, let $u,v\in X$ with $d(u,v)=n$ and assume we are given an enumeration $h_1, \dots, h_n$ consistent with the inclusion,
where $h_j\in[u,v]$ . By consistency, if $h\neq h_j$ were a halfspace between $u$ and $h_1$, 
then $h\in[u,v]$, contradicting that $d(u,v)=n$.
Therefore there is a unique oriented edge starting at $u$ corresponding to $h_1$. Let $x_1$ be the terminal vertex. Inductively, this defines a sequence $x_1, \dots, x_n$ where $x_n =v$. Since $|[u,v]| = d(u,v)$ it follows that this describes a combinatorial geodesic.\end{proof}
%
We consider also the {\em vertex-interval} 
\bqn
\I(u,v) := \{w\in X :  {w}\cap( {u}\cap  {v}) =  {u}\cap {v}\}\,,
\eqn
that is the set of all vertices that are crossed by some combinatorial geodesic 
between $u$ and $v$. 

The following fact seems to be folklore and is essential for our result. We refer to
\cite[Theorem 1.16]{Brodzki_Campbell_Guenthner_Niblo_Wright}  for a complete proof.  

\begin{lemma}[Intervals embedding]\label{embedintervals} Let $u,v\in\~{X}$.  
Then the vertex intervals $\I(u,v)$ isometrically embed into $\~{\Z^D}$ 
(with the standard cubulation) where $D$ is the dimension of $X$.
\end{lemma}

The image $\I_{u,v}$  in $\Z^D$ of the above embedding is 
exactly the CAT(0) cube complex associated 
to the halfspaces $\frakH(u,v):=[u,v]\cup[v,u]$.

\begin{remark}\label{rem:opposite=>Euclidean}  In general if $u\in X$ is an ultrafilter, 
its opposite $\*u$ is not an ultrafilter.  It is easy to see that if $\*u$ is
an ultrafilter, then $\frakH(X)=[u,\*u]\cup[\*u,u]$
and hence $X$ is an interval.
\end{remark}

Also recall that the vertex set of a CAT(0) cube complex 
with the edge metric is a {\em median space} \cite{Roller},
namely for every triple of vertices $u,v, w \in X$, 
the intersection  $\I(u,v) \cap \I(v,w) \cap \I(w,u)$ is  exactly a singleton.  
This unique point is called the {\em median} of $u$, $v$, 
and $w$ and we denote it by $m(u,v,w)$. It is a standard fact that 
\bq\label{eq:median}
 {m(u,v,w)} = ( {u}\cap  {v}) \cup ( {v}\cap  {w})\cup ( {w}\cap  {u})\,.
\eq

\subsection{Isometric Embeddings}\label{subsec:isom-emb}
If $\frakH'\subset\frakH(X)$ is an involution invariant subset of halfspaces, 
then $\frakH'$ is a pocset in its own right and hence one can considers the
associated CAT(0) cube complex $X(\frakH')$.  A priori, the complex $X(\frakH')$ 
that one obtains with this construction cannot be embedded as a subcomplex of $X$, 
but there is always a combinatorial quotient map 
$\pi_{\frakH'}:\~X\to \~X(\frakH')$, defined by $\alpha\mapsto\alpha\cap\frakH'$,
that restricts to $\pi_{\frakH'}:X\to X(\frakH')$. 
If the subset $\frakH'\subset\frakH$ is invariant for the action of a group $\Gamma\to\Aut(X)$
of combinatorial automorphisms, then $\~X(\frakH')$ 
inherits a $\Gamma$-action with respect to which the map $\pi_{\frakH'}$ is $\Gamma$-equivariant.

There are however conditions under which $\~X(\frakH')$ can be embedded in $\~X$. 

\begin{definition}\label{defi:lifting} Let $\frakH'\subset \frakH(X)$ be an involution invariant subset of halfspaces.   A {\em lifting decomposition} of $\frakH'$
is a choice of a subset $W\subset\frakH(X)\setminus \frakH'$ satisfying the partial choice and 
consistency conditions (see Definition~\ref{def:partial choice-consistency}), and so that $\frakH(X)=\frakH'\sqcup(W\sqcup \*W)$.
\end{definition}

We note that a lifting decomposition need not exist. 
A collection $\frakH'\subset \frakH$ is said to be {\em tight} if it satisfies the following: 
for every $h,k \in \frakH'$ if $h\subset \ell \subset k$ then $\ell \in \frakH'$. 
We remark that the existence of a lifting decomposition $W$  of $\frakH' \subset \frakH$ implies that $\frakH'$ is tight. 
Indeed, suppose that $h,k \in\frakH'$ and $h\subset \ell \subset k$. 
If $\ell \notin \frakH'$ then $\ell\in W\sqcup W^*$. Since $\frakH'$ is involution invariant, we may assume that $\ell \in W$. 
But this means that $k\in W$, which contradicts the fact that $W\cap \frakH' = \varnothing$. 
This shows that the condition that $\frakH'$ is tight is necessary for the existence of a lifting decomposition for it. 

\begin{lemma}\label{lem:embedding}  Let $\frakH'\subset \frakH(X)$ be a involution invariant tight subset of halfspaces.
Assume that $\frakH'$ admits a lifting decomposition   
$\frakH=\frakH'\sqcup(W\sqcup\*W)$.
Then there is an isometric embedding
$i:\~X(\frakH')\hookrightarrow\~X(\frakH)$, defined by 
$i(\alpha):=\alpha\sqcup W$, whose image is 
$i\big(\~X(\frakH')\big)=\bigcap_{h\in W}h$.

As particular cases, if $\frakH'=\varnothing$, then $i\big(\~X(\frakH_W)\big)$ is a point;
or if $W$ contains an infinite descending chain, then $i(\~X(\frakH'))\subset\partial X$.
\end{lemma}

\begin{proof} We first show that if $\alpha$ is an ultrafilter on $\frakH_W$, 
then $\alpha\sqcup W$ is an ultrafilter on $\frakH$. 
By construction, $\alpha\sqcup W$ satisfies the choice condition.  
We need to verify the consistency condition, that is that 
if $h\in \alpha\sqcup W$, $k\in\frakH$ with $h\subset k$ then $k\in\alpha\sqcup W$.

If $h\in W$ and $k\in\frakH$ is such that $h\subset k$, then $k\in W$,
since $W\subset \frakH$ satisfies the consistency condition.
If $h\in\alpha$ and $k\in\frakH$ is such that $h\subset k$, then either 
$k\in\alpha\sqcup\alpha^*$ and hence $k\in\alpha$ because $\alpha$ is
an ultrafilter on $\alpha\sqcup\alpha^*$, or $k\in W\sqcup W^*$.
But $k$ cannot be in $W^*$.  In fact, if $k\in W^*$, then $k^*\in W$; since $k^*\subset h^*$ and
$W$ satisfies the consistency condition, then $h^*\in W$, contradicting that
$\alpha\cap(W\sqcup W^*)=\varnothing$.

Now, assume that $\alpha, \beta \in \~X(\frakH')$.  It is easy to check that 
\begin{equation*}
(\alpha\sqcup W)\Delta (\beta\sqcup W)=\alpha\Delta \beta,
\end{equation*}
since $\frakH' \cap W = \varnothing$.
This shows that the embedding is an isometry and hence extends to the cube structure on $\~X(\frakH')$.

By definition $\alpha\sqcup W=\bigcap_{h\in\alpha\sqcup W}h\subset\bigcap_{h\in W}h$,
so that $\frakH \subseteq(\a\sqcup \a^*)\sqcup(W\sqcup W^*)$
Moreover $\bigcap_{h\in W}h$ consists of all partially defined ultrafilters on $W\sqcup \*W$:
to complete an element $x\in\bigcap_{h\in W}h$ to an ultrafilter on $\frakH$ is exactly
equivalent to choosing $\alpha\in\~X(\frakH')$.
\end{proof}

\begin{remark}\label{rem:lifting}  As alluded to above, 
the existence of a lifting decomposition is a very restrictive condition.
However, the existence of a {\em strongly convex} set, that is a set $B\subset X$ such that for any $x,y\in B$
also $\mathcal{I}(x,y)\subset B$, implies the existence of a lifting decomposition.  In fact if 
\bqn
\frakH':=\{h\in\frakH(X):\,h \text{ crosses }B\}\text{ and }W:=\{h\in \frakH(X):\,B\subset h\}\,,
\eqn
then $\frakH=\frakH'\sqcup(W\sqcup W^*)$ is a lifting decomposition of $\frakH'$
with which $\~X(\frakH')$ gets isometrically embedded in $\~X$ with image $B$.
\end{remark}

\begin{definition}\label{defy:projection} A map $\varrho:\~X\to\~X$ is a {\em projection}
if there exists a lifting decomposition $\frakH(X)=\frakH_W\sqcup(W\sqcup\*W)$ 
and  if $\varrho=i\circ\pi_{\frakH_W}$, 
where $i:\~X(\frakH_W)\hookrightarrow\~X$ is the isometric embedding in 
Lemma~\ref{lem:embedding}.
\end{definition}

It is easy to verify that the composition of two projections is still a projection.

If $\Gamma\to\Aut(X)$ is an action and $\frakH_W$ is $\Gamma$-invariant, 
then $\~X(\frakH_W)$ inherits a $\Gamma$-action.  If in addition
$\varrho$ is a projection and the embedding is $\Gamma$-equivariant, 
then the image of the projection is a $\Gamma$-invariant subcomplex in $\~X$.
This happens exactly when the choice of subset $W\subset\frakH(X)$
of the lifting decomposition is $\Gamma$-invariant, 
so that $i\big(\~X(\frakH_W)\big)$ is a $\Gamma$-invariant
subcomplex in $\~X(\frakH)$.


\subsection{Decomposition into Products}\label{subsec:decomposition} 
The product of CAT(0) cube complexes is a CAT(0) cube complex in a natural way.
If $X=Y\times Z$, there is the following decomposition of the hyperplanes
\bq\label{eq:decomposition of hyperplanes}
\hat\frakH(X)=\big\{\hat h_Y\times Z:\,\hat h_Y\in\hat\frakH(Y)\big\}
	\sqcup\big\{Y\times\hat h_Z:\,\hat h_Z\in\hat\frakH(Z)\big\}
        \cong\hat\frakH(Y)\sqcup\hat\frakH(Z)\,,
\eq
and $(\hat h_Y\times Z)\pitchfork(Y\times\hat h_Z)$ for any $h_Y\in\frakH(Y)$ 
and $h_Z\in\frakH(Z)$.

Conversely, any such partition of the hyperplanes into mutually transverse subset corresponds
to a decomposition of the CAT(0) cube complex into a product.  
In fact, by \cite[Proposition~2.6]{Caprace_Sageev} 
any CAT(0) cube complex  decomposes as a product $X=X_1\times\dots\times X_m$ 
of irreducible factors, $m\geq1$,
which are unique up to permutations and are often referred to 
as the {\em rank one} factors of $X$.

The induced CAT(0) metric (respectively, the combinatorial metric) on the product is 
the $\ell^2$-product (respectively, $\ell^1$-product) of the factor metrics.
We record the following standard fact:

\begin{lemma}\label{lem:boundary of factors}  Let $X=X_1\times\dots\times X_k$
be the product of CAT(0) spaces $X_j$, $j=1,\dots,k$
and let $G:=G_1\times \dots\times G_k$, where $G_j\leq\Aut(X_j)$ is a subgroup of the isometries of the $j$-th factor $X_j$.
Then any $G_j$-fixed point in $\partial_\sphericalangle X_j$
defines a $G$-fixed point in $\partial_\sphericalangle X$.
\end{lemma}

\begin{proof}
Let us denote by $\delta_j$ and $\delta$ the CAT(0) metrics 
respectively on $X_j$ and on $X$.
Assume that, up to permuting the indices, there is
a $G_1$-fixed point $\xi_1\in\partial_\sphericalangle X_1$.  Let
$\ell_1:[0,\8)\to X_1$ be a geodesic in $X_1$ representing $\xi_1$,
i.e. $\xi_1=\ell_1(\8)$.  Since $\xi_1$ is $G_1$-invariant, then
$\sup_{t\in[0,\8)}\delta_{1}(\gamma \ell_1(t),\ell_1(t))<\8$.  If $x_j\in
X_j$ for $2\leq j\leq m$ is any point, then $\ell:[0,\8)\to X$ defined by
$\ell(t):=(\ell_1(t),x_2,\dots,x_m)$ is a geodesic in $X$.  Then for any
$\gamma\in G$ we have 
\bqn 
\sup_{t\in[0,\8)}\delta(\gamma\ell(t),\ell(t))^2 
:=\sup_{t\in[0,\8)}\[\delta_{1}(\gamma\ell_1(t),\ell_1(t))^2+\sum_{j=1}^m \delta_{j}(\gamma x_j,x_j)^2\]
<\8\,,
\eqn 
hence $\ell(\infty)$ is $G$-invariant.
\end{proof}

In addition, there is a corresponding decomposition of the Roller boundary, 
\begin{equation*}
\partial X=\bigcup_{j=1}^m\~{X_1}\times\dots\~{X_{j-1}}\times\partial X_j\times
  \~{X_{j+1}}\times\dots\times\~{X_m}\,,
\end{equation*}
and $\Aut(X)$ contains $\Aut(X_1)\times\dots\times\Aut(X_m)$ 
as a finite index subgroup
($\Aut(X)$ is allowed to permute isomorphic factors).
If $\Gamma\to\Aut(X)$ is a group acting by automorphisms, 
then there is a subgroup $\Gamma_0<\Gamma$
of finite index ($\leq m!$) that acts on $X_j$ 
via the projection $\Gamma_0\to\Aut(X_j)$.

\subsection{The Essential Core}\label{subsec:essential}
A halfspace $h\in \frakH$ is said to be $\G$-\emph{essential} if for some (equivalently all) $x\in X$ the $\G$-orbit of $x$ inside $h$, 
that is $\G\cdot x\cap h$, is not at bounded distance from the hyperplane $\hat h$. 
A hyperplane $\hat h\in\hat\frakH$ is called $\Gamma$\emph{-essential} (or essential for short) 
if each of the corresponding halfspaces is $\G$-essential, and half-$\G$-essential (or half-essential)
if only one of the two corresponding halfspaces if $\Gamma$-essential.
The $\Gamma$\emph{-essential core} (or essential core) $Y$ of the $\Gamma$-action on $X$ is 
a CAT(0) cube complex corresponding to the $\Gamma$-essential (or essential) hyperplanes. 
The $\Gamma$-action on the $\Gamma$-essential core $Y$ is essential and 
any non-empty $\Gamma$-invariant convex subcomplex of $Y$ is equal to $Y$.
Following the notation of \cite{Caprace_Sageev}, we denote by 
$\operatorname{Ess}(X,\Gamma)$
the set of $\Gamma$-essential hyperplanes in $X$, so that we can write
\begin{equation*}
\hat\frakH(X)=\operatorname{Ess}(X,\Gamma)\sqcup\operatorname{nEss}(X,\Gamma)\,,
\end{equation*}
where the set of non-essential hyperplanes $\operatorname{nEss}(X,\Gamma)$ includes
both the half-essential and the trivial ones.  What is important is that both 
$\operatorname{Ess}(X,\Gamma)$ and $\operatorname{nEss}(X,\Gamma)$
are $\Gamma$-invariant subsets of $\hat\frakH(X)$, and hence the decomposition
\begin{equation*}
\hat\frakH(X)=\operatorname{Ess}(X,\Gamma)\sqcup\operatorname{nEss}(X,\Gamma)=
		         \operatorname{Ess}(Y,\Gamma)\sqcup\operatorname{nEss}(X,\Gamma)
\end{equation*}
are $\Gamma$-invariant.

While in general the essential core of an action can be empty, 
it is proven in \cite[Proposition~3.5]{Caprace_Sageev} that 
if there are no $\Gamma$-fixed points 
in the visual boundary $\partial_\sphericalangle X$ of $X$ and no
$\Gamma$-fixed points in $X$, then the essential core $Y$ is a non-empty 
$\Gamma$-invariant convex subcomplex $Y\subset X$.
As a consequence, one has both that 
$\partial_\sphericalangle Y\subset\partial_\sphericalangle X$
and $\partial Y\subset\partial X$.
However, even if $X$ is irreducible, its essential core $Y$ need not be.  
Let $Y=Y_1\times\dots\times Y_m$
be the decomposition into irreducible factors.  
Using the decomposition of hyperplanes for products discussed above we obtain
\begin{equation}\label{eq:decomposition of the hyperplanes}
\hat\frakH(X)=\operatorname{Ess}(Y,\Gamma)\sqcup\operatorname{nEss}(X,\Gamma)
 			  =\hat\frakH(Y_1)\sqcup\dots\sqcup\hat\frakH(Y_m)
                          \sqcup\operatorname{nEss}(X,\Gamma)\,,
\end{equation}
where we used for simplicity the notation $\hat\frakH(Y_j)$ to
indicate $\operatorname{Ess}(Y_j,\Gamma)$ (since by hypothesis they
coincide because the induced action on $Y_j$ is $\Gamma$-essential).

Let $\frakH(X)^n$ denote the set of $n$-tuples of halfspaces in $X$. 
Since if $n$-tuple $s\in \frakH(X)^n$ is essential, then any other halfspace
containing the halfspaces in $s$ is essential as well, the decomposition in
\eqref{eq:decomposition of the hyperplanes} induces a decomposition
\begin{equation}\label{decomp}
\frakH(X)^n=\frakH(Y_1)^n\sqcup\dots\sqcup\frakH(Y_m)^n
\sqcup\frakH_{nEss}(X)^n\,,
\end{equation}

where $\frakH_{nEss}(X)^n$ consists of $n$-tuples such that at
least one halfspace is non-essential.

%
\subsection{Skewering, Flipping: \"Uber-Separated and \"Uber-Parallel Pairs of Halfspaces}\label{subsec:skew-flip}
Flipping and double-skewering are important tools introduced by
Caprace--Sageev in \cite{Caprace_Sageev}.
\begin{definition}[\cite{Caprace_Sageev}]
 We say that $\g\in \Aut(X)$ {\em flips} a halfspace $ h\in\frakH(X)$ if
$\g \*h\subset h$.  Moreover we say that $\g$ {\em skewers} $\hat h$ 
if $\g h\subset h$ (or $h\subset  \g h$).
\end{definition}
Under reasonable hypotheses such combinatorial automorphisms can always be found.  
More precisely, if $X$ is a finite dimensional 
CAT(0) cube complex and $\Gamma\to\Aut(X)$ acts essentially on $X$ 
without fixing any point in the visual boundary $\partial_\sphericalangle X$, 
then for every halfspace $h\in\frakH(X)$, there exists $\g\in\G$ that flips $h$,
\cite[Flipping Lemma, \S~1.2]{Caprace_Sageev}.  As a simple consequence, 
we have also that, given any two halfspaces $k\subset h$, there exists
$\g\in\G$ such that $\g h\subset k\subset h$, 
\cite[Double Skewering Lemma, \S~1.2]{Caprace_Sageev}.

The following notion was first introduced by Behrstock--Charney, \cite{Behrstock_Charney}:

\begin{definition}[\cite{Behrstock_Charney}]\label{defi:strongly separated} 
We say that two parallel hyperplanes are {\em strongly separated}
if there is no hyperplane that is transverse to both.
\end{definition}

By the usual abuse of terminology we say that two halfspaces are
strongly separated if the corresponding hyperplanes are.

The existence of strongly separated hyperplanes is definitively a rank
one phenomenon.  In fact, it is easy to see that if $X$ is reducible,
then there are no strongly separated hyperplanes.  For non-elementary 
CAT(0) cube complexes, the fact that the
existence of strongly separated hyperplanes is actually equivalent to
the irreducibility of the CAT(0) cube complex was proven in
\cite{Caprace_Sageev}, although the case of a right-angled Artin group
can already be found in \cite{Behrstock_Charney}.

We will need a finer notion of strong separation, which is less standard but will be key to our work.
\begin{definition}\label{sss} 
Two strongly separated halfspaces $h_1$ and $h_2$ are said to be an {\em \"uber-separated pair} 
if any two halfspaces $k_1, k_2$ with the property that $h_i \pitchfork k_i$ for $i =1, 2$
are parallel.
We say that two strongly separated hyperplanes are
{\em \"uber-separated} if their halfspaces are.
 \end{definition}
\begin{remark}\label{here} If $h\subset k\subset \ell$ are pairwise strongly separated halfspaces, 
$h$ and $\ell$ is an \"uber-separated pair.\end{remark}
\begin{lemma}\label{lem:ss} Let $Y$ be a finite dimensional irreducible CAT(0) cube complex 
and $\Gamma\to\Aut(Y)$ a group acting essentially and non-elementarily. 
Given any hyperplane $\hat h$, there exists $\gamma\in \Gamma$ such that 
$\hat h$ and $\gamma\hat h$ is an \"uber-separated pair and $h\subset \gamma h$ 
(or $\gamma h\subset h$).
\end{lemma}
\begin{proof} By \cite[Proposition~5.1]{Caprace_Sageev} for any halfspace
  $h$ there is a pair of halfspaces $h_1,h_2$ such that
  $h_1\subset h\subset h_2$ and $\hat h_1$ and $\hat h_2$ are strongly
  separated.  We apply now the Double Skewering lemma in
  \cite[\S~1.2]{Caprace_Sageev} to the pair $h_1\subset h_2$ to obtain
  that $h_1\subset h_2\subset \gamma_0 h_1$, for some $\gamma_0\in\Gamma$.  
  By construction, and since
  $\Gamma$ acts by automorphisms of $Y$, we have the chain of
  inclusions
\begin{equation*}
h_1\subset h\subset h_2\subset\gamma_0 h_1
\subset \gamma_0 h\subset\gamma_0 h_2\subset\gamma_0^2 h_1
\subset \gamma_0^2 h\subset\gamma_0^2 h_2\,.
\end{equation*}
Since $\hat h_1$ and $\hat h_2$ are strongly separated, 
the same is true for $\gamma_0\hat h_1$ and $\gamma_0\hat h_2$.
Hence $\hat h$ and $\gamma_0^2\hat h$ is an \"uber-separated pair by Remark \ref{here}.
\end{proof}
Notice that \"uber-separated pairs are in particular strongly separated and hence they do not exist in the reducible case,
\cite[Proposition~5.1]{Caprace_Sageev}. To take care of his, in the sense we will explain, we will use the following generalization:
\begin{definition}\label{uparallel}
Two parallel halfspaces $h_1$ and $h_2$ are said to be {\em \"uber-parallel} if for every pair of halfspaces $k_1, k_2$ such that $h_i \pitchfork k_i$ for $i =1, 2$, 
then either the halfspaces $k_1$ and $k_2$ are parallel, or they each cross both $h_1$ and $h_2$. 
Two parallel hyperplanes are {\em \"uber-parallel} if their halfspaces are.\end{definition}
Notice that, according to the definition, an \"uber-separated pair is in particular \"uber-parallel.  
If $X$ is a product, then  \"uber-separated pairs from an irreducible factor will not be \"uber-separated in $X$ but will be \"uber-parallel. 
Even when $X$ is irreducible, there may be reducible subcomplexes of $X$. 
In such a reducible subcomplex, such as for example a copy of $\Z^2$ inside $\Z*\Z^2$, 
there may be pairs of halfspaces that are  \"uber-separated in one of the factors of that subcomplex but not \"uber-separated in $X$. 
The notion of \"uber-parallel captures these types of pairs, as is the case for example in the Salvetti complex associated to  $\Z*\Z^2$.
\subsection{The Bridge}\label{bridge}
The concept of {\em bridge} of two parallel hyperplanes was introduced by Behrstock--Charney in \cite{Behrstock_Charney}.
\begin{definition}
Let $h_1\subset h_2$ be a nested pair of halfspaces. 
Consider the set of pairs of points in $h_1\times h_2^*$ minimizing the distance between $h_1$ and $h_2^*$, that is
$$M_{h_1,h_2} =\{(x, y) \in h_1\times h_2^* :  \text{ if }(a,b) \in h_1\times h_2^* \text{ then } d(x,y)\leq d(a,b)\}.$$

It will be convenient to denote by $M_1$ and $M_2$  the projections of $M_{h_1,h_2}$ into $h_1$ and $h_2^*$, respectively.

The \emph{combinatorial bridge} connecting $h_1$ and $h_2^*$ is the union of intervals between such minimal distance pairs: 
$$b(\hat h_1,\hat h_2)= \bigcup_{(x,y)\in M_{h_1,h_2}}{}\mathcal{I}(x,y).$$ 
In the following we will drop the dependence on the hyperplanes whenever no confusion can arise.
\end{definition}

We observe that if $(x_1, y_1), (x_2, y_2) \in M_{h_1,h_2}$ then $d(x_1,y_1) = d(x_2,y_2)$. 
%
%
The following lemma on the distance between a point and a halfspace, permeates several proofs to come.
We denote by $\hat\frakH(u,h)$ the hyperplanes separating $u$ from $h$ and define the distance of 
$u$ from $h$  to be
\bqn
d(u,h):=\min\{d(u,v)|v\in h\}\,.
\eqn
\begin{lemma}\label{lem:distance-point-halfspace}
Let $u\in X$ and let $h$ be a halfspace so that $u\in h^*$. Then $d(u,h)$ equals 
the cardinality $|\hat\frakH(u,h)|$.
\end{lemma}
\begin{proof}
If a hyperplane separates $u$ from $h$, 
it will have to be crossed by any combinatorial geodesic from $u$ to $v$ for any $v\in h$ and 
hence it will contribute to $d(u,v)$. It follows that $|\hat\frakH(u,h)|\leq d(u,h)$.

Conversely take $v\in h$ minimizing the distance to $u$ and 
assume that a combinatorial geodesic from $u$ to $v$ crosses a hyperplane $\hat{k}$ transverse to $\hat{h}$. 
Since $\hat{k}$ and $\hat{h}$ are not comparable there is a (maybe different) combinatorial geodesic that crosses $\hat{k}$ before crossing $\hat{h}$. 
Let $v'$ be the point reached just after crossing $\hat{k}$: then $v'\in h$ because the geodesic has not crossed $\hat{h}$ yet, and $d(u,v')<d(u,v)$.
Since this is impossible by definition of $v$, the geodesic must cross at most all hyperplanes
separating $u$ from $h$.
\end{proof}

The structure of the bridge is obtained in the following

\begin{lemma}\label{Bstructure}
Let ${h}_1\subset {h}_2$ be any pair of nested halfspaces.  
\begin{enumerate}
\item\label{item0} If $\hat h$ separates two points in $M_i$, $i=1$ or $2$, then $\hat h$ crosses $\hat h_i$.
\item\label{item1} Let $(p_1,p_2)\in  M_{h_1,h_2}$ and suppose that $\hat h$ separates $p_1$ and $p_2$.  
Then $\hat h$ is parallel to both $\hat h_1$ and $\hat h_2$. 
\item \label{item2} If a hyperplane $\hat{h}$ separates any two points on the bridge and
$\hat h$ is transverse to either of the $\hat h_i$, with $i\in\{1,2\}$, then $\hat h$ is transverse to both the $\hat{h}_i$.
\item\label{item3} For any $(p_1, p_2)\in M_{h_1,h_2}$, the distance $d(p_1, p_2)$ is exactly  the number of hyperplanes separating $h_1$ from $h_2^*$, 
including $\hat h_1$ and $\hat h_2$.
\item\label{item4} The bridge $b(\hat h_1,\hat h_2)$ is isomorphic to a product and strongly convex. 
More precisely, $b(\hat h_1,\hat h_2)\cong M_1\times\mathcal{I}(p_1,p_2)$ where $M_1$, the projection of $M_{h_1,h_2}$ into $h_1$, 
is strongly convex, and $(p_1,p_2)$ is any pair in $M_{h_1,h_2}$.
\end{enumerate}
\end{lemma}
Before starting the proof we make the general observation that if $p_1\in M_1$
then no hyperplane $\hat h$ can separate $p_1$ from $\hat h_1$.  In fact,
if there such hyperplane, the geodesic joining $p_1$ to the point $p_2\in M_2$
such that $(p_1,p_2)\in M_{h_1,h_2}$ would have to cross this hyperplane 
before crossing $\hat h_1$, contradicting that $(p_1,p_2)$ is a minimizing pair.
The same argument holds of course for $p_2\in M_2$.

We also establish the following easy 
\begin{claim}  Let $p_1\in M_1$ and $p_2\in M_2$ be such that $(p_1,p_2)$ minimizes the distance.
Assume that there exists $\hat h$ such that $\hat h\pitchfork\hat h_1$, $p_1\in h$ and $p_2\in h^*$.
Then either $p_1$ belongs to the cube identified by $\hat h$ and $\hat h_1$ or there exists $\hat h'$
such that $\hat h'\pitchfork\hat h_1$, $p_1\in h'$ and $\hat h'\subset h$.
\end{claim}

\begin{proof}[Proof of Claim]  
If $p_1$ does not belong to the cube determined by $\hat h$ and $\hat h_1$, 
there exists a hyperplane $\hat h'$ separating $p_1$ from $\hat h$.  
If $\hat h'$ were not transverse to $\hat h_1$, then $\hat h'$ would be a hyperplane separating $p_1$ from $\hat h_1$,
which we observed is not possible.
\end{proof}

\begin{proof}[Proof of Lemma~\ref{Bstructure}]
\eqref{item0}  For simplicity let us set $i=1$ and let $p_1,p'_1\in M_1$ be the points separated by $\hat h$.
If $\hat h$ is not transverse to $\hat h_1$, then $\hat h$ must separate, say, $p_1$ from $\hat h_1$
and we observed already that this is not possible.

\eqref{item1} If $\hat h$ were to cross $\hat h_1$, we could assume, by applying repeatedly the claim, 
that $p_1$ belongs to the cube identified by $\hat h$ and $\hat h_1$.  Then by crossing $\hat h$
one would still remain in $\hat h_1$ and reach a point closer to $p_2$, contradicting the minimality of $(p_1,p_2)$.

\eqref{item2} Let $\hat{h}$ be a hyperplane that separates two points on the bridge and $\hat{h}\pitchfork\hat{h}_1$.
Let us assume that $\hat h$ is parallel to $\hat{h}_2$. 
Then, up to replacing $h$ by $h^*$, $M_2\subset h^*$. 
If it were also $M_1 \subset h^*$,  then the interval between any element of $M_1$ and any element in $M_2$ 
would be contained in $h$ and hence $b(\hat h_1,\hat h_2) \subset h$, which contradicts the assumption on $\hat h$. 
Therefore $M_1 \cap h^*\neq\varnothing$. Let $p_1\in M_1\cap\hat h^*$ and let $p_2 \in M_2$ 
such that $(p_1, p_2) \in M_{h_1,h_2}$.   
By construction, $\hat h$ separates $p_1$ and $p_2$ and hence, by \eqref{item1},
cannot be transverse to either $\hat h_1$ or $\hat h_2$, contradicting
the hypothesis.

\eqref{item3} Let $(p_1,p_2)\in M_{h_1,h_2}$. Clearly $d(p_1, p_2)$ is greater than or equal to the number of hyperplanes separating $h_1$ from $h_2$. 
The other inequality is the assertion in  \eqref{item1}.

\begin{definition}\label{def:Bstructure}  Let $b(\hat h_1,\hat h_2)$ be the bridge of the hyperplanes $\hat h_1, \hat h_2$.  
\begin{enumerate} 
\item The hyperplanes crossing both both $\hat{h}_1$ and $\hat{h}_2$ are {\em horizontal hyperplanes} 
and are denoted by $\hat\beta_h$.
\item The hyperplanes separating $\hat{h}_1$ and $\hat{h}_2$ are called {\em vertical hyperplanes} 
and are denoted by $\hat\beta_v$.
\end{enumerate}
\end{definition}

\medskip
\noindent
{\em Continuation of the proof of Lemma~\ref{Bstructure}.}
\eqref{item4} From \eqref{item2} we see that the hyperplanes of the bridge are either 
horizontal or vertical.  
Then any element of $\hat{\beta}_h$ crosses any element of $\hat{\beta}_v$ and vice-versa;
therefore the bridge $b(\hat h_1,\hat h_2)$ is a product $X(\hat\beta_h)\times X(\hat\beta_v)$. 
Furthermore, by part \eqref{item1} we have that $\I(p_1, p_2) \cong X(\hat\beta_v)$ for any $(p_1, p_2) \in M_{h_1,h_2}$. 

To conclude it remains to show that $M_1$ is strongly convex. 
First we notice that, by \eqref{item1} and \eqref{item2}, if $\hat\beta_h$ is not empty,
each element of $\hat\beta_h$ separates elements of $M_1$ and of $M_2$. 
Now take $s_1,t_1\in M_1$ and $u_1\in\mathcal{I}(s_1,t_1)$.  
Let $s_2,t_2\in M_2$ be the other end of the minimizing pairs for $s_1,t_1$. 
Let $u_2\in\mathcal{I}(s_2,t_2)$ be the element obtained by starting at $s_2$ and crossing the hyperplanes separating $s_1$ from $u_1$. 
This is well defined because, by \eqref{item0} and \eqref{item2},
the hyperplanes separating $s_1$ from $t_1$ are all in $\hat\beta_h$. 
Then only the hyperplanes separating ${h}_1$ from ${h}_2^*$ can separate $u_1$ from $u_2$. 
Hence the pair $(u_1,u_2)\in M_{h_1, h_2}$, so that $\mathcal{I}(s_1,t_1)\subset M_1$. 
Hence, $M_1$ is strongly convex and thus is $b(\hat h_1,\hat h_2)=M_1\times\mathcal{I}(p_1,p_2)$.
\end{proof}

Although we will not need it, we observe that, since the bridge is strongly convex, 
if $H:=\{h\in\frakH:\,b(\hat h_1,\hat h_2)\subset h\}$, 
then $\frakH=(\beta_h\sqcup\beta_v)\sqcup(H\sqcup H^*)$ is a lifting decomposition of the hyperplanes associated to the bridge
and hence the bridge is isometrically embedded in $X$ as a product.

In view of Lemma~\ref{Bstructure} the following is well defined.
\begin{definition} Let $h_1\subset h_2$ be a nested pair of halfspaces. The \emph{length $\ell(b(\hat h_1,\hat h_2))$ of the bridge} of $\hat h_1$ and $\hat h_2$ 
is the cardinality of the set $\beta_v$. 
\end{definition}

We will adopt the usual abuse of terminology and refer to {\em horizonal halfspaces} (respectively {\em vertical halfspaces}) 
the halfspaces corresponding to the horizontal (respectively vertical) hyperplanes, and denote them by $\beta_h$ (respectively $\beta_v$).

It is  straightforward to see that if $\hat h_1$ and $\hat h_2$ are strongly separated,
then the corresponding set of horizontal halfspaces is empty, \cite[Lemma~2.2]{Behrstock_Charney}.

The following lemma is probably well-known, but we include it here because we could not find a reference for it.
\begin{lemma}Let $X$ be a CAT(0) cube complex, $x\in X$ and $B\subseteq X$ a strongly convex subset. There is a unique point $p_B(x)\in B$ minimizing the combinatorial distance between $x$ and $B$.\end{lemma}
Note that this lemma is standard in the case of the CAT(0) distance, see for instance \cite{Bridson_Haefliger}.
In the case of a CAT(0) space however, the proof of the existence of an orthogonal projection is a bit more difficult 
than the proof of its uniqueness.
\begin{proof} Since the combinatorial distance takes discrete values, the existence of a point in $B$
minimizing the distance is obvious.
To prove uniqueness, let $y$ and $y'$ be two points in $B$ minimizing the distance between $B$ and $x$.  
Let us show that $\hat\frakH(x,y)\subseteq \hat\frakH(x,y')$, where $\hat\frakH(x,y)$ is the collection of hyperplanes separating $x$ from $y$: 
If a hyperplane $\hat{h}_0\in \hat\frakH(x,y)$ does not separate $x$ from $y'$, 
then it has to belong to $\hat\frakH(y,y')$, and so do the hyperplanes $\hat{h}_1,\dots,\hat{h}_s$ separating $\hat{h}_0$ from $y$. 
Let $p\in \I(x,y)$ obtained by starting at $y$ and crossing the hyperplanes $\hat{h}_s,\dots,\hat{h}_0$. 
Then $p$ also belongs to $\mathcal{I}(y,y')\subseteq B$;  
but $d(x,p)<d(x,y)$ because $p$ is also on a geodesic from $x$ to $y$, 
contradicting that $y$ was distance minimizing. So  $\hat\frakH(x,y)=\hat\frakH(x,y')$. By Lemma \ref{geod} we deduce that $y=y'$.
\end{proof}
We can hence give the following definition.
\begin{definition} For $x\in X$ and  $B$  a strongly convex subset in $X$, denote by $p_B(x)\in B$ the \emph{projection of $x$ on $B$}.\end{definition}
\begin{lemma}\label{MedDecomp}Let $h_1\subset h_2$ be a pair of nested halfspaces, $x_1\in h_1$ and $x_2\in h_2^*$. 
Denote by $b$ the bridge connecting ${h}_1$ and ${h}_2^*$. Then
\begin{enumerate}
\item A horizontal hyperplane of the bridge cannot separate $x_i$ from $p_{b}(x_i)$, for $i=1,2$.
\item The following holds true:
$$d(x_1,x_2)=d(x_1,p_b(x_1))+d(p_b(x_1),p_b(x_2))+d(p_b(x_2),x_2).$$\end{enumerate}\end{lemma}
 \begin{proof} (1) Observe first of all that $p_b(x_1)\in b(\hat h_1,\hat h_2)\cap h_1$.  
 Now let $\hat{h}$ be a horizontal hyperplane of the bridge. If $\hat{h}$ separates $x_1$ from $p_b(x_1)$, 
 say $x_1\in h$ and $p_b(x_1)\in h^*$, then there is a point in $b(\hat h_1,\hat h_2)\cap h_1$ different from $p_b(x_1)$ 
 and at distance from $x_1$ smaller than $d(x_1,p_b(x_1))$, 
 contradicting that $p_b(x_1)$ is the projection of $x_1$ on $b(\hat h_1,\hat h_2)$. 
  
 (2) That $d(x_1,x_2)\leq d(x_1,p_b(x_1))+d(p_b(x_1),p_b(x_2))+d(p_b(x_2),x_2)$ follows from the triangle inequality, 
 so let us show the other inequality which we do by showing that 
 $\hat\frakH(x_1,p_b(x_1))\cup \hat\frakH(p_b(x_1),p_b(x_2))\cup \hat\frakH(p_b(x_2),x_2])\subset \hat\frakH(x_1,x_2)$.
 
A hyperplane $\hat{h}$ separating $p_b(x_1)$ from $p_b(x_2)$ cuts the bridge and 
hence by Lemma~\ref{Bstructure}(3) is either vertical or horizontal. 
If it is vertical, it separates $h_1$ from $h_2^*$ and hence $x_1$ from $x_2$ as well. 
If $\hat{h}$ is horizontal, it cannot separate $x_i$ from $p_i(x_i)$ by part (1) of this lemma. 
Since $\hat{h}$ is separating $p_b(x_1)$ from $p_b(x_2)$, this forces $\hat{h}$ to separate $x_1$ from $x_2$. 
 
By part (1) a hyperplane $\hat{h}$ separating $x_i$ from $p_b(x_i)$, for $i=1$ or $2$, cannot be horizontal; 
by Lemma~\ref{Bstructure} it cannot cross the bridge, so it has to separate $x_1$ from $x_2$.  
\end{proof}

\subsection{Finite Orbits in the CAT(0) Boundary Versus Finite Orbits in the Roller Boundary}
Non-elementarity of the action is defined in terms of the non-existence of a finite orbit 
in the CAT(0) boundary.  We describe in this section to which extent this is equivalent 
to the same property with respect to the Roller boundary.

We start with one direction of the equivalence that is very easy and is here for completeness, 
since it will not be needed in the following. 

\begin{prop}\label{roller-to-visual} Let $Y$ be a finite dimensional CAT(0) cube complex
and let $\Gamma\to\Aut(Y)$ be an action on $Y$.  
If the action is essential and there is a finite orbit in
the Roller boundary, then there is a finite orbit in the CAT(0) boundary.
\end{prop}
\begin{proof} Let $\xi\in\partial Y$ be a point in the finite $\Gamma$-orbit and 
let $\Gamma_\xi$ be its stabilizer, whose action is still essential since is of finite index 
in $\Gamma$.  
Let $h\in\frakH$ be a halfspace containing $\xi$.  If there were no $\Gamma_\xi$-fixed 
point in $\partial_\sphericalangle Y$, we could apply the Flipping Lemma 
(see \S~\ref{subsec:skew-flip}); hence there would exist $\gamma\in\Gamma_\xi$ 
that flips $\*h$, so that $\gamma h\subset\*h$.  
But this would contradict the fact that $\xi=\gamma\xi\in\gamma h$.  
\end{proof}

The following proposition pins down to which extent the elementarily of an action 
implies the existence of a finite orbit in the Roller boundary.

\begin{prop}\label{prop:visual-to-roller} Let $X$ be a finite dimensional CAT(0) cube complex
and let $\Gamma\to\Aut(X)$ be an action on $X$.  
If there is a finite orbit in the CAT(0) boundary, then either
\be
\item there is a finite $\Gamma$-orbit in the Roller boundary, or
\item there exists a subgroup of finite index $\Gamma'<\Gamma$ and a $\Gamma'$-invariant 
subcomplex $X'\hookrightarrow \partial X$ on which the $\Gamma'$-action is non-elementary.
Moreover $X'$ corresponds to a lifting decomposition of halfspaces.
\ee
\end{prop}

The argument will depend on the following lemma, which we assume now,
and whose verification we defer to right after the proof of the proposition.
We start with the following natural construction, that can also be found in \cite[\S~4.1]{Guralnik}.
Note that there is a similar construction in \cite[\S~3]{Caprace_Monod_amen}.

Let $\xi\in \partial_\sphericalangle X$,
let $g:[0,\infty)\to X$ be a geodesic asymptotic to $\xi$ and let us define $T_\xi$ to be 
the following set of halfspaces 
\bq\label{eq:as}
T_\xi:=\big\{h\in\frakH:\,
\text{for every }\epsilon>0\text{ there exists }t_\epsilon\geq0
\text{ such that }N_\epsilon\big(g(t_\epsilon,\infty)\big)\subset h\big\}\,,
\eq
where $N_\epsilon\big(g(t_\epsilon,\infty)\big)$ is the $\epsilon$-neighborhood
of the image of the geodesic ray $g|_{(t_\epsilon,\infty)}$.

\begin{lemma}\label{lem:transverse}  Let $X$ be a CAT(0) cube complex
with a $\Gamma$-action and let $\xi\in\partial_\sphericalangle X$.  
The set $T_\xi$ in \eqref{eq:as} satisfies 
the following properties:
\be
\item it is independent of the geodesic $g$ and $\Gamma_\xi$-invariant,
where $\Gamma_\xi<\Gamma$ is the stabilizer of $\xi$ in $\Gamma$;
\item it is not empty;
\item it satisfies the partial choice and consistency conditions;
\item it contains an infinite descending chain.
\ee
\end{lemma}

\begin{proof}[Proof of Proposition~\ref{prop:visual-to-roller}]
Let $\xi\in \partial_\sphericalangle X$ be one of the points in the finite $\Gamma$-orbit, 
and let $\Gamma_\xi<\Gamma$ be the stabilizer of $\xi$, $[\Gamma:\Gamma_\xi]<\infty$.
It follows from the above Lemma~\ref{lem:transverse} and 
from \S~\ref{subsec:isom-emb} that 
$T_\xi$ induces a lifting decomposition
\bq\label{eq:lifting}
\frakH(X)=\frakH_\xi\sqcup(T_\xi\sqcup \*T_\xi)
\eq 
and there is a $\Gamma_\xi$-equivariant projection $\varrho:\~X\to\~X$
whose image is the isometrically embedded $\Gamma_\xi$-invariant
subcomplex $\~X_\xi:=i\big(\~X(\frakH_\xi)\big)$.  Observe that,
because of Lemma~\ref{lem:transverse}(4) and Lemma~\ref{lem:embedding},
$\dim X_\xi<\dim X$.  Moreover, if $\frakH_\xi=\varnothing$, then
$\~X_\xi=X_\xi$ is a $\Gamma_\xi$-fixed point in $\partial X$.

Proceeding inductively, we can conclude the proof.  
In fact, if the $\Gamma_\xi$-action on $X_\xi$ is non-elementary, we are in case
(2) of the proposition.  If on the other hand there is a finite $\Gamma_\xi$-orbit
in $\partial X_\xi$, using the fact that the composition of two projections is a 
projection, we can repeat the argument.  The finite dimensionality 
of $X$ insures that the process terminates.
\end{proof}

\begin{proof}[Proof of Lemma~\ref{lem:transverse}]  (1) Only for this part of the proof
we denote by $T_\xi(g)$ and by $T_\xi(g')$ the sets defined in \eqref{eq:as}
with respect to two asymptotic geodesics $g$ and $g'$.  
Then $g$ and $g'$ are at bounded distance from each other, that is
there exists an $r>0$ such that $g'\big([0,\infty)\big)\subset N_r\big(g[0,\infty)\big)$.
By the triangle inequality, if $\epsilon>0$, then 
$N_\epsilon\big(g'[t,\infty)\big)\subset N_\epsilon\big(g[t,\infty)\big)$, for all $t\geq0$.
But this implies that if $h\in \frakH$ is such that there exists
$t_{\epsilon+r}$ with $N_\epsilon\big(g'(t_{r+\epsilon},\infty)\big)\subset h$
then also $N_\epsilon\big(g(t_{r+\epsilon},\infty)\big)\subset h$. 
Thus $T_\xi(g)=T_\xi(g')$ and therefore
$T_\xi(g)$ is independent of $g$.

Since $\xi$ is $\Gamma_\xi$-invariant, the geodesics $g$ and $\gamma g$ are
asymptotic and hence $T_\xi(g)=T_\xi(\gamma g)$.  The $\Gamma_\xi$-invariance
of $T_\xi$ then follows at once, since 
$\gamma N_\epsilon\big(g(t,\infty)\big)=N_\epsilon\big(\gamma g(t,\infty)\big)$
for all $t\geq0$.

\medskip
\noindent
(2) We will show that if for every $h\in\mathfrak{H}$ there exists $\epsilon_h>0$ such that 
\begin{eqnarray}\label{parallel}
N_{\e_h}(g(t,\8)) \cap h \neq \varnothing \text{ and } N_{\e_h}(g(t,\8)) \cap h^* \neq \varnothing
\end{eqnarray}
for all $t>0$, then there is an infinite family of pairwise transverse hyperplanes.

We may assume that $h\in\mathfrak{H}$ is not compact, otherwise 
\eqref{parallel} is never verified for any $\e$.
Moreover observe that any geodesic $\gamma$ crosses infinitely many hyperplanes.
Order the cubes according to the order in which they are crossed by $\gamma$.
This gives rise, up to choosing an order of hyperplanes on each of these cubes, to 
an order $\^h_1,\^h_2,\dots$ on the hyperplanes 
according to the order in which they are crossed by $\gamma$.
If $\e_i>0$ is the smallest $\e$ such that  \eqref{parallel} is verified for $\^h_i$, 
then any $\gamma(t)$ is at CAT(0) distance at most $\e_i$ from $\^h_i$.

We claim that for every $\^h_\ell$ there exists $n_\ell\in\N$ such that 
for every $j\geq n_\ell$ the hyperplane $\^h_j$ is transverse to $\^h_\ell$. 
In fact, the CAT(0) distance is quasi-isometric to the combinatorial distance 
and the combinatorial distance between $\gamma(t)$ and 
$\^h_\ell$ is the number of hyperplanes parallel to $\^h_\ell$ 
that separate $\^h_\ell$ from $\gamma(t)$.  On the other hand 
for every $j>n_\ell$, any hyperplane $\^h_j$ parallel to $\^h_\ell$ 
that intersects $\gamma$ will contribute to the distance from $\gamma(t)$
and $\^h_\ell$ for $t$ large enough.  
By the previous observation this is not possible and hence eventually $\^h_j$ must intersect $\^h_\ell$. 

By setting $m_{\ell+1}:=n_{n_\ell}$, for every $d\in\N$ 
the hyperplanes $\^h_1,\^h_{m_1},\^h_{m_2},\dots, \^h_{m_d}$ 
form a family of $d+1$ pairwise transverse hyperplanes.

\medskip
\noindent
(3) is obvious from the construction.

\medskip
\noindent
(4) If $T_\xi$ does not contain an infinite descending chain, then $\~X_\xi\cap X\neq\varnothing$,
where $\~X_\xi$ is the complex associated to the lifting decomposition in \eqref{eq:lifting}.
Because of Lemma~\ref{lem:embedding}, $\xi\in\partial_\sphericalangle X_\xi$.
We can hence apply the construction in the beginning of the proof of 
Proposition~\ref{prop:visual-to-roller}
to the complex $X_\xi$, whose halfspaces are now $\frakH\smallsetminus(T_\xi\sqcup \*T_\xi)$,
thus contradicting (2).
\end{proof}

\subsection{From Products to Irreducible Essential Factors}

The following lemma identifies the important properties that are passed down from a complex 
to the irreducible factors of the essential core.  The content of the lemma is already in
\cite{Caprace_Sageev}, but we recall it here in the form in which we will need it.
 
\begin{lemma}\label{lem:hereditary properties}  Let $X$ be a finite dimensional CAT(0) 
cube complex and let $\Gamma\to\Aut(X)$ be a non-elementary action. 
Then the $\Gamma_0$-action on the irreducible factors of the essential core is
also non-elementary and essential, where $\G_0$ is the finite index subgroup preserving this decomposition.
\end{lemma}

\begin{proof}
Let $Y\subset X$ be the essential core, $Y=Y_1\times\dots\times Y_m$ 
its decomposition into irreducible factors, and let $\G_0$ be the finite index subgroup preserving this decomposition. 
We need to show that the following hold:
\begin{enumerate}
\item The action of $\Gamma_0$ on the $Y_j$, $j=1,\dots,m$ is essential as well.
\item If $\Gamma_0$ has no finite orbit on the visual boundary
  $\partial_\sphericalangle X$, then the same holds for the action on
  $\partial_\sphericalangle Y_j$, $j=1,\dots, m$. 
\end{enumerate}

\noindent
(1) By \cite[Proposition~3.2]{Caprace_Sageev}, 
the $\Gamma_0$-action on $Y$ (resp. on $Y_i$) is essential if and only if 
every hyperplane $\hat h\in\hat\frakH$ (resp. $\hat h_i\in\hat\frakH_j$)
can be skewered by some element in $\Gamma_0$. 
If $\hat h_j\in\hat\frakH(Y_j)$ is a hyperplane in $Y_j$, 
then $\hat h:=Y_1\times\dots Y_{j-1}\times\hat h_j\times Y_{j+1}\times\dots\times Y_m$ 
is a hyperplane in $Y$.
Since the action on $Y$ is essential, there exists $\gamma\in\Gamma_0$
that skewers $\hat h$ and hence it skewers $\hat h_j$.  
Then the $\Gamma_0$-action on $Y_i$ is essential.

\noindent
(2) We prove the contrapositive of the statement.
Let $\Gamma_0<\Gamma$ be the finite subgroup that preserves
each of the factors $Y_j$ and let us assume, by passing if necessary 
to a further subgroup of finite index, that there is a $\Gamma_0$-fixed point in 
$\partial_\sphericalangle Y_j$ for some $1\leq j\leq m$.
Then by Lemma~\ref{lem:boundary of factors} 
there is a if $\Gamma_0$-fixed point in  $\partial_\sphericalangle Y$
and hence a finite $\Gamma$-orbit in $\partial_\sphericalangle Y$. 
Since $Y$ is a convex subset of $X$ and hence
$\partial_\sphericalangle Y\subset \partial_\sphericalangle X$,
there is a finite $\Gamma$-orbit in $\partial_\sphericalangle X$.
\end{proof}

\subsection{Euclidean (Sub)Complexes} 
\begin{definition} Let $X$ be a CAT(0) cube complex. 
We say that $X$ is {\em Euclidean} if the vertex set with the combinatorial metric 
embeds isometrically in $\Re^D$ with the $\ell^1$-metric, for some $D<\infty$.
\end{definition}
In \cite[Theorem~7.2]{Caprace_Sageev}, under some natural conditions on the action
of $\Aut(X)$, the authors relate the existence of an $\Aut(X)$-invariant
{\em Euclidean flat} with the non-existence of a facing triple of halfspaces, in the following sense.
\begin{definition}
Let $n\in\N$. An $n$-tuple of halfspaces is called a \emph{facing $n$-tuple} if they are pairwise disjoint. 
An  $n$-tuple of hyperplanes is called a \emph{facing $n$-tuple} if there is a choice of halfspaces forming a facing $n$-tuple.
\end{definition}
As our setting differs slightly from the one used in \cite{Caprace_Sageev},
we discuss briefly in this section
the notion of Euclidean complexes and subcomplexes.
The following definition is from  \cite{Caprace_Sageev}.
\begin{definition}
A CAT(0) cube complex $X$ is said to be {\em $\Re$-like} 
if there is an $\Aut(X)$-invariant bi-infinite CAT(0) geodesic. 
\end{definition}
\begin{prop}\label{lem:euclidean equivalences}
 Let $Y$ be a CAT(0) cube complex on which $\Aut(Y)$ acts essentially. 
Consider the following statements:
\begin{enumerate} 
\item $Y$ is Euclidean.
\item $Y$ is an interval.
\item  $Y$ is a product of $\R$-like factors.
\end{enumerate}
Then (3)$\Rightarrow$(2)$\Rightarrow$(1).
\end{prop}
\begin{proof}
Observe that conditions (1) and (2) are preserved under taking products. 
Also, the hypothesis of having an essential action is preserved 
by passing to the irreducible factors 
by Lemma~\ref{lem:hereditary properties}. 
Therefore, it is sufficient to consider the case in which $Y$ is irreducible.

(3)$\Rightarrow$(2). Assume that $Y$ is $\R$-like. 
Let $\ell \subset Y$ be the $\Aut(Y)$-invariant CAT(0) geodesic. 
We claim that $\ell$ crosses every hyperplane of $Y$. In fact, otherwise
there would be a halfspace $h_0$ containing $\ell$
and, since $\ell$ is $\Aut(Y)$-invariant, then $\hat h_0$ would not be essential.
 
Let $\ell : \Re \to Y$ be a parametrization of $\ell$. One can check that, 
because of the above claim, the collection of halfspaces 
$$\a:= \{h \in {\mathfrak H} (Y): \text{ there exists } t \in \R \text{ such that } h \supset \ell(t, \8)\}$$
defines a non-terminating ultrafilter.
Then $Y$ is an interval on $\a$ and its opposite ultrafilter $\a^*=\frakH\smallsetminus\alpha$. 

(2)$\Rightarrow$(1) This is Lemma~\ref{embedintervals}.
\end{proof}

We prove next that, under the assumption that there are no fixed points in the visual boundary
and the action is essential, being Euclidean is equivalent 
to the non-existence of facing triples of hyperplanes\footnote{It is possible
that a Euclidean CAT(0) cube complex $Y$ on which 
$\Aut(Y)$ acts essentially and without fixed points in the visual boundary,
is a point (cf. \cite[Theorem~E]{Caprace_Sageev}).}. 
As a byproduct, using \cite[Theorem~7.2]{Caprace_Sageev} we can 
conclude that also (1) implies (3) under the above hypotheses.
We start with the following easy lemma.

\begin{lemma}\label{NtuplesNonEuclidean} If $X$ is a Euclidean CAT(0) cube complex
that isometrically embed into $\R^D$, then any set of pairwise facing halfspaces
has cardinality at most $2D$.
\end{lemma}
\begin{proof}
  Indeed any collection of halfspaces can be arranged in at most 
  $D$ chains.  Hence for each dimension there
  can be at most one pair of facing halfspaces and the assertion
  follows from the fact that the $\ell^1$-metric on $\Re^D$ is the sum
  of the $\ell^1$-metrics on its factors.
\end{proof}
More precisely, we have the following dichotomy that is compatible 
with the terminology in \cite{Caprace_Sageev} but holds also in the case
in which the CAT(0) cube complex does not have a cocompact group
of automorphisms.

\begin{cor}\label{cor:dichotomy} Let $Y$ be a finite dimensional irreducible CAT(0) 
cube complex 
and assume that $\Aut(Y)$ acts essentially and without fixed points on 
$\partial _\sphericalangle Y$. 
Then $Y$ is Euclidean if and only if $\frakH(Y)$ does not contain a facing triple 
of halfspaces.
\end{cor}

\begin{proof}  We first prove that if $Y$ is Euclidean then there are no facing triples
of halfspaces.  
Since the action is essential and there are no fixed points in $\partial_\sphericalangle Y$,
if there is a facing triple of halfspaces we can skewer several times two of the halfspaces
into the third one to obtain a set of pairwise facing halfspaces of arbitrarily large
cardinality.  Then Lemma~\ref{NtuplesNonEuclidean}  implies that $Y$ is not Euclidean.

Conversely, we assume that there are no facing triples of hyperplanes and 
prove that $Y$ must be Euclidean.  Since $Y$ is irreducible, 
let $\{h_n\}$ be a descending sequence of strongly separated halfspaces, 
$h_{n+1}\subset h_n$.  The strategy of the proof consists in showing that 
$\bigcap h_n$ consists of a single point $\alpha\in\partial Y$ and in using
the non-existence of facing triples of hyperplanes to show that $\*\alpha$
is also an ultrafilter.  
Then Remark~\ref{rem:opposite=>Euclidean} will complete the proof.

To show that $\bigcap h_n$ is a single point, let us assume by contradiction that
$\bigcap h_n$ contains at least two distinct points, $u,v\in\bigcap h_n$.  
Let $\hat h$ be a hyperplane that separates them.
Observe that for every $n\in\N$
\bq\label{eq:uv}
\ba
u\in h\cap h_n&\neq\varnothing\text{ and}\\
v\in\*h\cap h_n&\neq\varnothing\,.
\ea
\eq
From this and the fact that the $h_n$ are a descending chain,
one can check that, if there exists $N\in\N$ such that $\hat h\pitchfork h_{N}$,
then $\hat h\pitchfork h_{n}$ for all $n\geq N$, which is impossible since 
the $\{h_n\}$ are pairwise strongly separated.  So $\hat h\|\hat h_n$ for
every $n\in\N$.  

Again from \eqref{eq:uv} it follows that $\hat h\subset h_n$ for all $n\in N$.
But this is also not possible since there exist finitely many hyperplanes
between $\hat h$ and $\hat h_n$.  Hence $\alpha:=\bigcap h_n$ is a
single point.

To see that  $\*\alpha$ is an ultrafilter,
we need only to check the consistency condition, namely that if
$h\in\*\alpha$ and $h\subset k$, then $k\in\*\alpha$.  Observe that 
this is equivalent to verifying that if $\alpha\in\*h$ and $h\subset k$,
then $\alpha\in\*k$.  Suppose that this is not the case,
that is that there exists $h,k\in\frakH(Y)$ such that $h\subset k$
and $\alpha\in\*h\cap k$.

We first claim that 
\bq\label{eq:cont}
\text{there exists }n_0\in\N\text{ such that }h_n\subset k
\text{ for all }n>n_0\,.
\eq

In fact, suppose that there exists $n_0'\in\N$ 
such that $\hat h_{n_0'}\pitchfork\hat k$.
Since the $\{h_n\}$ are pairwise strongly separated,
then $\hat h_n\|\hat k$ for all $n>n_0'$.  Using \eqref{eq:transverse},
the fact that the $\{h_n\}$ are a descending chain
and that $\alpha\in h_n$ for all $n\in \N$, it is easy to
verify that $h_n\subset k$ for all $n>n_0'$.

On the other hand, if $\hat h_n\|\hat k$ for all $n\in\N$, using again that 
the $\{h_n\}$ are a descending chain and that there are
only finitely many hyperplanes between any $\hat h_n$ and $\hat k$,
one can easily verify  that there exists $n_0''\in\N$ such that 
$h_n\subset k$ for all $n>n_0''$.  Hence \eqref{eq:cont}
is verified with $n_0=\max{} \{n_0',n_0''\}$.

Since $h\subset k$ and there are only finitely many
hyperplanes between $\hat h$ and $\hat k$,
there exists $n_1\geq n_0$ such that either $\hat h_{n_1}\pitchfork\hat h$
or $h_{n_1}\subset h$.  But $h_{n_1}$ cannot be contained in $h$
since $\alpha\in h_{n_1}\cap\*h$, hence $\hat h_{n_1}\pitchfork\hat h$.

Again because the hyperplanes $\{\hat h_n\}$ are strongly separated, 
if $n>n_1$ then $\hat h_n\|\hat h$.  This, the fact that $\hat h_{n_1}\pitchfork\hat h$
and that $h_n\subset h_{n_1}$ imply that $h_n\cap h=\varnothing$.

It follows that $\*h_n,\*h$ and $k$ is a facing triple of halfspaces, 
contradicting the hypothesis.  Hence $\*\alpha$ is an ultrafilter
and the proof  is complete.
\end{proof}

We conclude the section with the following corollary that will 
be paramount in the sequel.

\begin{cor}\label{no Euclidean factors} Let $X$ be a finite
dimensional CAT(0) cube complex and $\Gamma\to\Aut(X)$ a non-elementary
action.  Then there are no Euclidean factors in the essential core.
\end{cor}

\begin{proof} Let $Y\subset X$ be the essential core of the $\Gamma$-action
and let $Y_0$ be an irreducible factor of $Y$.  By Lemma~\ref{lem:hereditary properties}
the $\Gamma$-action on $Y_0$ is also essential and non-elementary.
By Corollary~\ref{cor:dichotomy},
$Y_0$ cannot be Euclidean.
\end{proof}


\subsection{Facing Triples of Halfspaces}\label{subsec:nDisjoint}
In this section we show how the hypotheses of
non-elementarity and essentiality of the action are used to construct
suitable facing triple of hyperplanes. 
\begin{definition}A facing $n$-tuple of halfspaces is a \emph{facing \"uber-separated (or parallel) $n$-tuple} if all the pairs of halfspaces are \"uber-separated (or parallel) pairs.\end{definition}
As usual we extend the above definition to hyperplanes in the obvious way.
We will need the following lemma only in the case $n=3$, but the extension to larger $n$ is very easy.
\begin{lemma}\label{lem:sss facing triple}  Let $X$ be a
CAT(0) cube complex with a non elementary action  $\Gamma\to\Aut(X)$ and $n\in\N$. 
Then any essential halfspace $h\in\frakH(X)$ belongs to a facing \"uber-parallel $n$-tuple all of whose
halfspaces can be taken to be in a single $\Gamma$-orbit.
\end{lemma}

\begin{proof}First we assume that $X$ is irreducible and essential. We show the existence of a facing \"uber-separated $n$-tuple. 
According to Corollary~\ref{cor:dichotomy}, since $X$ is non-Euclidean, 
it contains a facing triple of halfspaces, call it $a,b,c$. 
Using Lemma~\ref{lem:ss} we find $\gamma_1,\gamma_2,\gamma_3,\in\Gamma$ such that 
$\gamma_1a\subset a$, $\gamma_2 b\subset b$ and $\gamma_3 c\subset c$ are \"uber-separated pairs. 
Hence the triple $\gamma_1 a,\gamma_2 b$ and $\gamma_3 c$ is facing and \"uber-separated. 
To get a facing $n$-tuple out of a facing $(n-1)$-tuple $h_1,\dots,h_{n-1}$, 
we flip and skewer two elements of the $(n-1)$-tuple into a third one;
for example we flip and skewer $h_1$ and $h-2$ into $h_{n-1}$ via $\gamma_1\,\gamma_2\in\Aut(X)$,
and now the $n$-tuple $h_1,h_2,\dots,\gamma_1h_1,\gamma_2h_2$ will be \"uber-separated.

\smallskip

To get a facing \"uber-separated $n$-tuple in an orbit, take $h$, and any facing \"uber-separated $(n+1)$-tuple of halfspaces. 
Then $h$ crosses at most one element of this facing $(n+1)$-tuple. 
Skewer and flip $h$ into the $n$ other halfspaces to get a facing \"uber-separated $n$-tuple of halfspaces in the orbit of $h$.

\smallskip
 
Let $\G_0<\G$ be a finite index subgroup preserving each irreducible factor 
of the essential core $Y$ of $X$.
Notice that the hypotheses that the action is non elementary and essential are preserved up to passing to $\G_0$. 
One then deduces the general case where $X$ is not necessarily irreducible and essential
by using that any essential halfspace belongs to an irreducible factor of $Y$. 
We find the \"uber-separated $n$-tuple in that irreducible factor of the essential core, 
and use it to produce an \"uber-parallel $n$-tuple on $X$.
\end{proof}

\section{Construction and Boundedness of the Median Class}\label{sec:cocycle}
Let $\Gamma$ be a group and $E$ be a coefficient $\Gamma$-module, 
that is the dual of a separable Banach space on which $\Gamma$ acts by linear isometries.
The bounded cohomology of $\Gamma$ with coefficients in $E$
is the cohomology of the subcomplex of $\Gamma$-invariants in $(\cb(\Gamma^{k+1},E),d)$, 
where 
\bq\label{eq:cb}
\cb(\Gamma^k,E):=\{f:\Gamma^k\to E:\,\sup_{g\in \Gamma^k} \|f(g)\|_E<\infty\}\,,
\eq
is endowed with the $\Gamma$-action
\bqn
(g f)(g_1,\dots,g_k):=g\cdot f(g^{-1}g_1,\dots,g^{-1}g_k)\,,
\eqn
and 
\bqn
\xymatrix@1{d:\cb(\Gamma^k,E)\ar[r]&\cb(\Gamma^{k+1},E)}
\eqn
is the usual homogeneous coboundary operator defined by 
\bqn
df(g_0,\dots,g_k):=\sum_{j=0}^k(-1)^jf(g_0,\dots,g_{j-1},g_{j+1},\dots,g_k)\,.
\eqn

%
%
%

\subsection{The Median Cocycle}\label{cocycle on boundary X}
Let $X$ be an irreducible finite dimensional CAT(0) cube complex. 
Recall that $\~X$ denotes the Roller compactification of $X$,
that is the set of ultrafilters on $\frakH(X)$ (see \S~\ref{sec:prelim}). 
For $n\geq 2$ we denote by $\frakH(X)^n$ the set of $n$-tuples of halfspaces of $X$.

If $1\leq p<\infty$, then $\ell^p(\frakH(X)^n)$ is the dual of a separable Banach space.
In fact, if $1<p<\infty$, then $\ell^p(\frakH(X)^n)$ is the dual of $\ell^q(\frakH(X)^n)$,
where $1/p+1/q=1$.  On the other hand,  $\ell^1(\frakH(X)^n)$ is the dual of the
Banach space $C_0(\frakH(X)^n)$ of functions on $\frakH(X)^n$ that vanish at infinity, 
which is separable since $\frakH(X)^n$ is countable.  
For further use, we set the notation

\bq\label{eq:e_p}
\mathcal E_p:=\begin{cases}\ell^q(\frakH(X)^n)&1<p<\infty\text{ and }1/p+1/q=1\\
C_0(\frakH(X)^n)&p=1\end{cases}
\eq

For each $1\leq p<\infty$ and each integer $n\geq 1$, we define in this section a one-parameter family of cocycles
\bqn
\xymatrix@1{c_{(n,R)}: \~X\times  \~X\times  \~X\ar[r]&\ell^p(\frakH(X)^n)\,,}
\eqn
that, by evaluation on a basepoint in $\~X$ will give a cocycle on 
$\Gamma\times\Gamma\times\Gamma$.
We define the {\em median cocycle} $c_{(n,R)}$ as the coboundary of an $\Aut(X)$-invariant map
$\omega_{(n,R)}$ on $\~X\times\~X$ whose values are not in general $p$-summable 
and we will show that, on the other hand,
$c_{(n,R)}=d\omega_{(n,R)}$ is bounded in the sense of \eqref{eq:cb} if $n\geq2$.
For $n\geq2$, the {\em median class} ${\tt m}_{(n,R)}$ will be defined as the cohomology class of $c_{(n,R)}$
(which is independent of the basepoint).

If $X$ is irreducible with an essential and non-elementary $\Gamma$-action then
the collection of sequences of length $n$ of \"uber-separated pairs at consecutive distance less than $R$ is nonempty for $R$ sufficiently large. 
Indeed, according to Caprace--Sageev \cite{Caprace_Sageev} since $X$ is irreducible and nonelementary, 
it contains a strongly separated pair, and by essentiality we can repeatedly skewer this pair to get an \"uber-separated and nested $n$-tuple for any $n\in\N$. 
In the general case one can always find \"uber-parallel sequences in the essential core of the action, and extend those to the whole space. 
We hence define $[[u,v]]_n$ for $u, v \in\~X$ to be the collection of pairwise \"uber-parallel $n$-tuples $(h_1,\dots,h_n) \in \frakH(X)^n$ 
such that $h_1 \subset \cdots \subset h_n$ and $h_i \in  v\smallsetminus u$ for each $i$.

For $R>0$ we also define
\begin{equation}
[[u,v]]_{(n,R)}= \{(h_1,\dots,h_n) \in [[u,v]]_n: d(h_i,h_{i+1})\leq R\}\,.
\end{equation}
So, $[[u, v]]_{(n,R)}$ is the collection of sequences of length $n$ 
of nested \"uber-parallel halfspaces containing $v$ and not $u$ and at consecutive distance less than or equal to $R$. 
We hope that the notation suggests that these are in some sense subintervals.

For $u,v\in \~X$, let us define 
\bq\label{eq:omega}
\omega_{(n,R)}(u,v):=\1_{[[u,v]]_{(n,R)}} - \1_{[[v,u]]_{(n,R)}}\,.
\eq

We will simply write $c$, $\omega$ and $[[u,v]]$ for $c_{(n,R)}$,
$\omega_{(n,R)}$ and $[[u,v]]_{(n,R)}$ when the context is clear. 

Fixing $u,v\in \~X$,  the function has finitely many values
\bqn
\omega(u,v):\frakH(X)^n\to\{-1,0,1\}
\eqn
and is finitely supported  when $u,v \in X$. 

Notice that $\omega$ is not necessarily bounded when thought of as a function 
with values in $\ell^p(\frakH(X)^n)$ and, in fact, its norm is proportional to the distance between $u$ and $v$.

Let us now consider the $\Aut(X)$-equivariant cocycle 
taking values in the functions on $\frakH(X)^n$,
defined as
\bq\label{cocycle}
\ba
  c(u_1,u_2,u_3)
:=&\(d\omega\)(u_1, u_2, u_3)\\ 
= &\omega(u_2, u_3) - \omega(u_1, u_3) + \omega(u_1, u_2) 
= \omega(u_2, u_3) +\omega(u_3, u_1) + \omega(u_1, u_2)\\ 
= &\1_{[[u_2 ,u_3 ]]} + \1_{[[u_3 ,u_1 ]]} +\1_{[[u_1 ,u_2 ]]} -
\(\1_{[[u_3 ,u_2 ]]} + \1_{[[u_1 ,u_3 ]]} + \1_{[[u_2 ,u_1]]}\)\,. 
\ea
\eq

We will show that, contrary to $\omega$, the cocycle $c$ on $\~X$ actually takes values 
in $\ell^p(\frakH(X)^n)$ (Proposition~\ref{prop:bounded1}) and is bounded
in the sense of \eqref{eq:cb}.
\begin{remark} Let $Y\subset X$ be the essential core of the $\Gamma$-action on $X$ and 
$Y=Y_1\times\dots\times Y_m$ is the 
decomposition of $Y$ into irreducible CAT(0) cube complexes. 
From the decomposition in \eqref{decomp}, we have a corresponding decomposition
\begin{equation}\label{eq:decomposition ellp}
\ell^p(\frakH(X)^n)\cong\ell^p(\frakH(Y_1)^n)\oplus\dots\oplus
     \ell^p(\frakH(Y_m)^n)\oplus \ell^p(\frakH_{nEss}(X)^n)
\end{equation}
given by $f\mapsto \1_{\frakH(Y_1)^n}f+\dots+\1_{\frakH(Y_m)^n}f+\1_{\frakH_{nEss}(X)^n}f$,
where the direct sum is in the $\ell^p$ sense.
The direct summand $\ell^p(\frakH(Y_j)^n)$ are invariant for the action of a finite index
subgroup $\Gamma'<\Gamma$. \end{remark} 

\begin{prop}\label{prop:decomposition} Let $Y$ be an essential CAT(0) cube complex and consider the cocycle defined in \eqref{cocycle} 
\begin{equation*}
c_{(n,R)}:\~Y\times\~Y\times\~Y\to\ell^p(\frakH(Y)^n)\,,
\end{equation*}
where $R$ is chosen to be large enough so that in each irreducible component $Y_i$ of $Y$,
the set of \"uber-separated $n$-tuples at consecutive distance less than or equal to $R$ is not empty.
Then $c_{(n,R)}$ decomposes as
\begin{equation*}
c_{(n,R)}(\xi,\eta,\zeta)=c^1_{(n,R)}(\pi_1(\xi),\pi_1(\eta),\pi_1(\zeta))\oplus\dots
    \oplus c^m_{(n,R)}(\pi_m(\xi),\pi_m(\eta),\pi_m(\zeta))\,
\end{equation*}
where
\begin{equation*}
c^j_{(n,R)}:\~{Y_j}\times\~{Y_j}\times\~{Y_j}\to\ell^p(\frakH(Y_j)^n)
\end{equation*}
is the cocycle on the irreducible factors and $\pi_j:\~Y\to\~Y_j$ is the projection. 
Moreover $c_{(n,R)}(\xi,\eta,\zeta)\neq0$ if and only if $c^j_{(n,R)}(\pi_j(\xi),\pi_j(\eta),\pi_j(\zeta))\neq0$, for some $1\leq j\leq n$.
\end{prop}
\begin{proof} Let $\omega$ and $\omega_j$, for $j=1,\dots,m$,
be defined as in \eqref{eq:omega} respectively on $\~Y$ and $\~Y_j$.
Since $c^j_{(n,R)}=d\omega^j_{(n,R)}$ for $1\leq j\leq k$ and $c_{(n,R)}=d\omega_{(n,R)}$, 
it is enough to verify that 
$$
\omega_{(n,R)}=\omega^1_{(n,R)}+\dots+\omega^k_{(n,R)}\,.
$$
Let $(\xi,\eta)\in\~Y\times\~Y$ and set $\xi_j:=\pi_j(\xi)$ for $1\leq j\leq m$.
Since $\omega^j_{(n,R)}(\xi_j,\eta_j)=\1_{[[\xi_j,\eta_j]]^j_{(n,R)}}-\1_{[[\eta_j,\xi_j]]^j_{(n,R)}}$ and 
$\omega_{(n,R)}(\xi,\eta)=\1_{[[\xi,\eta]]_{(n,R)}}-\1_{[[\eta,\xi]]_{(n,R)}}$, it is enough to see that
$$
[[\xi,\eta]]_{(n,R)}=[[\xi_1,\eta_1]]^1_{(n,R)}\sqcup\dots\sqcup[[\xi_m,\eta_m]]^m_{(n,R)}\,,
$$
where $[[\xi_j,\eta_j]]^j_{(n,R)}\subset\frakH(Y_j)^{n}$.  
But this follows immediately from the structure of
the hyperplanes and halfspaces in a product.  
\end{proof}

Corollary~\ref{cor:dim} will then be a direct consequence of
Proposition~\ref{prop:decomposition},
once Theorem~\ref{thm_intro:main} will be proven.

Another property of the median class of an action is that it behaves nicely with respect
to subcomplexes in the following sense:

\begin{prop}\label{prop:restriction} Let $X$ be a finite dimensional CAT(0) cube complex,
let $\Gamma\to\Aut(\overline{X})$ an action and $\Gamma_0<\Gamma$ a finite index subgroup. 
Let $W\subset\frakH(X)$ be a consistent and $\G_0$-invariant subset, 
so that $\frakH(X)=\frakH_W\sqcup(W\sqcup\*W)$ a lifting decomposition.
Let $X_W\subset\partial X$ be the corresponding subcomplex.
Then the median class of the $\Gamma$-action on $X$ restricts to the median class
of the $\Gamma_0$-action on $\overline{X}_W$.
\end{prop}

\begin{proof}  Since $\frakH_W^n\subset\frakH(X)^n$, there is a map
$j:\ell^p(\frakH(X)^n)\to\ell^p(\frakH_W^n)$ obtained by restriction.
If $c:\Gamma\times\Gamma\times\Gamma\to\ell^p(\frakH(X)^n)$
is the median $\Gamma$-equivariant cocycle on $X$, then
$j\circ c|_{\Gamma_0^3}:\Gamma_0\times\Gamma_0\times\Gamma_0\to\ell^p(\frakH_W^n)$
is the median $\Gamma_0$-equivariant cocycle on $X_W$.
\end{proof}
\subsection{Boundedness of the Median Class}
\begin{prop}\label{cisbounded}\label{prop:bounded1}
Let $X$ be a finite dimensional CAT(0) cube complex and,
for $1\leq p<\infty$,  let $c_{(n,R)}$ be 
the one-parameter family of cocycles defined in \eqref{cocycle}. 
Then
\bq\label{eq:ellp}
\xymatrix@1{c_{(n,R)}:\~X\times\~X\times\~X\ar[r]
&\ell^p(\frakH(X)^n)}
\eq
and 
\begin{equation*}
\sup_{u_1,u_2,u_3\in\~X}\|c_{(n,R)}(u_1,u_2,u_3)\|_p<\infty\,.
\end{equation*} 
More precisely, if $D$ is the dimension of $X$ then for any $u_1, u_2, u_3\in\~X$, 
the support of $c_{(n,R)}(u_1, u_2, u_3)$ has cardinality bounded above by  $$6(2(n-1)R)^{2D+n-2}.$$
\end{prop}
To prove this proposition, we need a few preliminary results. For a set $S\subset X$ define $V_\ell(S)$ to be the $\ell$-neighborhood of $S$, 
i.e. the set of vertices at combinatorial distance less than or equal to $\ell$ from some element of $S$.  

We start with the following key result, where \"uber-parallel is needed:
\begin{lemma}\label{key} Let $h_1\subset h_2$ be an \"uber-parallel pair of halfspaces, $x\in h_1$ and $y\in h_2^*$. 
Let $\ell$ be the length of the corresponding bridge. Then
$$\mathcal{I}(x,y) \subset V_\ell(h_1\cup h_2^*)\,,$$ 
that is, the interval between $x$ and $y$ stays within $\ell$ of $h_1\cup h_2^*$.
\end{lemma}
Before proceeding with the proof we record the following important remark, straightforward from the proof of
Lemma~\ref{lem:distance-point-halfspace} and the concepts used in \S~\ref{bridge}.
\begin{remark}\label{travels}Let $u,v\in X$ and a let $h$ be a halfspace 
so that $u,v\in h^*$ and assume that $\hat\frakH(u,h)\subseteq\hat\frakH(v,h)$. 
Let $u=x_0,\dots,x_n=v$ be a combinatorial geodesic from $u$ to $v$, and let $d_i=d(x_i,h)$. 
Our assumptions on $u$ and $v$ force the sequence $d_i$ to be increasing, 
and  the proof of the above lemma shows that it remains constant as long as the hyperplanes crossed are transverse to $\hat{h}$, 
but increase when they are parallel. In words, crossing a hyperplane parallel to $\hat{h}$ will take the geodesic away from $h$.\end{remark}
\begin{proof}[Proof of Lemma \ref{key}] 
We will show that any geodesic from $x\in h_1$ to $y \in  h_2^*$ stays within $\ell$ of $h_1$ 
then goes to the bridge $b(\hat h_1,\hat h_2)$ and then stays within $\ell$ of $h_2^*$ to reach $y$.  
Since by Lemma~\ref{Bstructure}, $b(\hat h_1,\hat h_2)\subset  V_\ell(h_1\cup h_2^*)$, we will have shown that the geodesic never leaves $V_\ell(h_1\cup h_2^*)$. 

According to Lemma \ref{geod}, a geodesic between $x$ and $y$ corresponds to an enumeration of all the hyperplanes separating $x$ from $y$, 
hence by Lemma~\ref{MedDecomp}(2) it has to cross all hyperplanes separating $x$ from $p_b(x)$, 
where  $p_b(x)\in h_1$ is the projection of $x$ on the bridge $b(\hat h_1,\hat h_2)$, 
all those separating $p_b(x)$ from $p_b(y)$ and all those separating $p_b(y)$ from $y$, not necessarily in this order.  
In fact, when two hyperplanes are parallel the enumeration in the geodesic has to respect the order given by the inclusion of the corresponding halfspaces, 
but when two hyperplanes are transverse the geodesic can cross either one first.

Thus to understand how far away from $h_1\cup h_2^*$ a combinatorial geodesic can possibly go, 
we have to study the possible intersections of elements belonging to the following disjoint sets: 
\begin{equation*} 
\hat\frakH(x,p_b(x))\,,\quad
\hat\frakH(p_b(x),p_b(y))\subset\beta_h\sqcup\beta_v\,,\text{ and }
\hat\frakH(y,p_b(y))\,,
\end{equation*}
where $\beta_h\sqcup\beta_v$ is the decomposition of the halfspaces in the bridge into horizontal and vertical halfspaces according to Lemma~\ref{Bstructure}. 

By Lemmas~\ref{Bstructure} and \ref{MedDecomp}, since $h_1$ and $h_2$ are \"uber-parallel, 
none of the hyperplanes from $\hat\frakH(x,p_b(x))$ can cross a hyperplane from $\hat\frakH(y,p_b(y))$.
Hence a geodesic from $x$ to $y$ must enumerate all the hyperplanes from $\hat\frakH(x,p_b(x))$ 
before enumerating any hyperplane from $\hat\frakH(y,p_b(y))$. 

Now the hyperplanes from $\beta_h$ all cross $\hat{h}_1$ so, according to Remark \ref{travels},
they will not allow the geodesic to travel away from $h_1$. 
The only hyperplanes that can take a geodesic away from $h_1$ are the ones from $\beta_v$ and from $\hat\frakH(y,p_b(y))$. 
There are at most $\ell$ hyperplanes from  hyperplanes from $\beta_v$ and the hyperplanes from $\hat\frakH(y,p_b(y))$ will not matter 
as they will take the geodesic away from $h_1$ when it is already $\ell$-close to $h^*_2$. 
Indeed, since the geodesic has to exhaust all the elements of $\hat\frakH(x,p_b(x))$ before using hyperplane from $\hat\frakH(y,p_b(y))$, 
the same argument for $h_2^*$ shows that it will be $\ell$-close to $h^*_2$.
\end{proof}
The above lemma says that, in case $h_1$ and $h_2$ are \"uber-parallel, in order to go from $h_1$ to $h_2^*$ 
one needs to travel on the bridge. The relevance of the hypothesis of being \"uber-parallel is exemplified in the following:
\begin{example}\label{ex:six quarter planes}
Take six quarter planes glued in a natural way around their boundaries. 
\vskip.05cm
\begin{minipage}{0.5\textwidth}
Let $h_1, h^*_2$ be the halfspaces corresponding to the hyperplanes $\hat h_1$ and $\hat h_2$ in the figure.
In this case $\ell(b(\hat h_1,\hat h_2))=2$, but one can easily find $x\in h_1$ and $y\in h^*_2$ such that 
a geodesic joining $x$ and $y$ is not contained in a $2$-neighborhood of $h_1\cup h^*_2$.
In fact, as $x,y$ move away from the bridge, there are geodesics joining them that are also arbitrarily far away it.
In this case the pair $\hat h_1, \hat h_2$ is strongly separated but not \"uber-separated.
\end{minipage}
\begin{minipage}{0.5\textwidth}
\begin{center}
\begin{tikzpicture}[scale=0.5]
\draw[very thick] (-6,0) -- (6,0);
\draw[very thick] (-5/1.732,-5) -- (5/1.732,5);
\draw[very thick] (-5/1.732,5) -- (5/1.732,-5);
\draw (-2,2*1.732+2) -- (1/1.732,1) -- (1/1.732+5,1);
\draw (-1.5,1.732*1.5+4) -- (2/1.732,2) -- (2/1.732+4,2);
\draw (-1,1.732+6) -- (3/1.732,3) -- (3/1.732+3,3);
\draw (0,8) -- (4/1.732,4) -- (4/1.732+2,4);
\draw (-2,-2*1.732-2) -- (1/1.732,-1) -- (1/1.732+5,-1);
\draw (-1,-1.732-4) --(2/1.732,-2) -- (2/1.732+4,-2);
\draw (0,-6) -- (3/1.732,-3) -- (3/1.732+3,-3);
\draw (1,1.732-8) -- (4/1.732,-4) -- (4/1.732+2,-4);
\draw (-1/1.732-5,1) -- (-1/1.732,1) -- (2,2*1.732+2);
\draw (-2/1.732-4,2) -- (-2/1.732,2) -- (1.5,1.732*1.5+4);
\draw (-3/1.732-3,3) -- (-3/1.732,3) -- (1,1.732+6);
\draw (-4/1.732-2,4) -- (-4/1.732,4) -- (0,8);
\draw (-1/1.732-5,-1) -- (-1/1.732,-1) -- (2,-2*1.732-2);
\draw (-2/1.732-4,-2) -- (-2/1.732,-2) -- (1,-1.732-4);
\draw (-3/1.732-3,-3) -- (-3/1.732,-3) -- (0,-6);
\draw (-4/1.732-2,-4) -- (-4/1.732,-4) -- (-1,1.732-8);
\draw  (4,-1.732*3) -- (1,0) -- (4,1.732*3);
\draw  (4.5,-1.732*2.5) -- (2,0) -- (4.5,1.732*2.5);
\draw  (5,-1.732*2) -- (3,0) -- (5,1.732*2);
\draw  (5.5,-1.732*1.5) -- (4,0) -- (5.5,1.732*1.5);

\draw  (-4,1.732*3) -- (-1,0) -- (-4,-1.732*3);
\draw  (-4.5,1.732*2.5) -- (-2,0) -- (-4.5,-1.732*2.5);
\draw  (-5,1.732*2) -- (-3,0) -- (-5,-1.732*2);
\draw  (-5.5,1.732*1.5) -- (-4,0) -- (-5.5,-1.732*1.5);
\draw[red, thick]  (4,-1.732*3.5) -- (0.5,0) -- (4,1.732*3.5);
\draw[red, thick]  (-4,1.732*3.5) -- (-0.5,0) -- (-4,-1.732*3.5);
\draw[red] (4+.4,1.732*3.5+.2) node {\tiny{$\hat h_2$}}	
	             (-4-.4,1.732*3.5+.2) node {\tiny{$\hat h_1$}};
\filldraw[blue] (-1,0) circle [radius=.15]  
                  (0,0) circle [radius=.15]
                  (1,0) circle [radius=.15];     
\filldraw[blue] (-2.732,3) circle [radius=.1];
\draw[blue] (-2.732-.4,3-.3) node {\tiny{$x$}};     
\filldraw[blue] (2.732,3) circle [radius=.1];
\draw[blue]         (2.732+.5,2.4)  node {\tiny{$y$}}; 
\filldraw[blue] (0,6) circle [radius=.1];
\filldraw[blue] (0,6.3) node {\tiny{$m$}};
\draw[very thick,blue] (-2.732,3) -- (-1.732,3) -- (0,6) -- (1.732,3) -- (2.732,3);
\end{tikzpicture}
\end{center}
\begin{center}
{\sc Figure~1}: $m\not\in V_2(h_1\cup h^*_2)$
\end{center}
\end{minipage}
\end{example}

\begin{cor}\label{MedianClose} Let $x,y\in X$ and $h_1\subset h_2$ be an \"uber-parallel pair of halfspaces 
such that $x\in h_1 \subset h_2$ and $y\in h_2^*$. Take $z\in h_1^*\cap h_2$. 
Then the median $m(x,y,z)\in V_\ell(h_1\cup h_2^*)$, where $\ell$ is the length of the bridge between $h_1$ and $h_2$.
\end{cor} 
\begin{proof}Follows directly from the fact that the median is contained in the interval between $x$ and $y$, 
which in turn is contained in $V_\ell(h_1\cup h_2^*)$.
\end{proof}
\begin{cor}\label{halfspaceNear} Let $x,y\in X$, and  $m\in \mathcal{I}(x,y)$. Let $h_1\in [m,x]$ and $h_2 \in [y,m]$, 
with $h_1\subset h_2$, be an \"uber-parallel pair at distance less than or equal to $R$. Then $B_R(m)\cap (h_1 \cup h_2^*) \neq \varnothing$.
\end{cor} 

\begin{proof} This is just a reformulation of Lemma~\ref{key}.  In fact, if $d(\hat h_1,\hat h_2)\leq R$,
then $\ell(b(\hat h_1,\hat h_2))\leq R$.  Then Lemma~\ref{key} implies that any point in an interval
$\mathcal{I}(x,y)$ with $x\in h_1$and $y\in h^*_2$, is at distance at most $R$ from $h_1$ or $h^*_2$.
\end{proof}

We remark that Example~\ref{ex:six quarter planes} shows that if $\hat h_1$ and $\hat h_2$ are only strongly separated,
the assertion of Corollary~\ref{halfspaceNear} does not hold, as one can see in Figure~1 with $R=2$.

\begin{example}\label{blow} The following ``infinite staircase'' shows an irreducible CAT(0) cube complex 
\begin{minipage}{0.25\textwidth}
with a pair of hyperplanes $\hat h_1$ and $\hat h_2$ that are parallel  and strongly separated but not \"uber-parallel. 
This example is elementary and there are no \"uber-separated pairs. 
Furthermore what captures the pathology of this example is the fact that the median can be arbitrarily far from the bridge $b(\hat h_1,\hat h_2)$. 
The notion of \"uber-separated precisely excludes this pathology. 
\end{minipage}
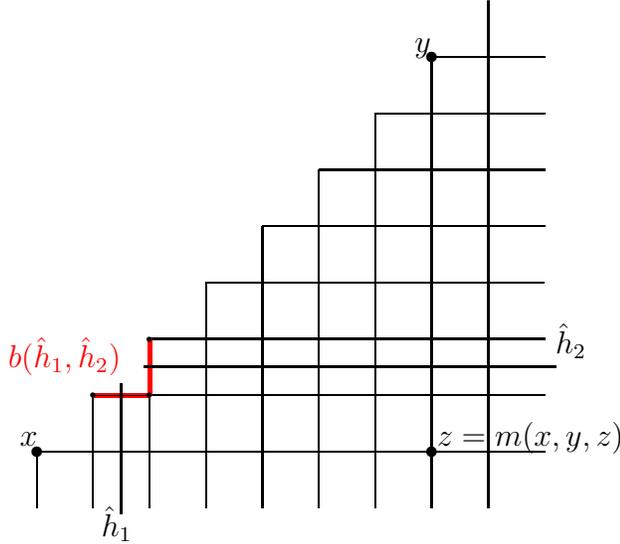
\begin{minipage}{0.5\textwidth}
\hskip1cm
\begin{center}
\setlength{\unitlength}{.75cm}
\begin{picture}(10,10)
\put(1,0){\line(0,1){1}}
\put(2,0){\line(0,1){2}}
\put(3,0){\line(0,1){3}}
\put(4,0){\line(0,1){4}}
\put(5,0){\line(0,1){5}}
\put(6,0){\line(0,1){6}}
\put(7,0){\line(0,1){7}}
\put(8,0){\line(0,1){8}}
\put(9,0){\line(0,1){9}}
\put(1,1){\line(1,0){9}}
\put(2,2){\line(1,0){8}}
\put(3,3){\line(1,0){7}}
\put(4,4){\line(1,0){6}}
\put(5,5){\line(1,0){5}}
\put(6,6){\line(1,0){4}}
\put(7,7){\line(1,0){3}}
\put(8,8){\line(1,0){2}}
\put(1,1){\circle*{0.2}}\put(0.7,1.1){$x$}
\put(8,1){\circle*{0.2}}\put(8.1,1.1){$z=m(x,y,z)$}
\put(8,8){\circle*{0.2}}\put(7.7,8.1){$y$}
{\red
\put(2,2.025){\line(1,0){1}}
\put(2,1.975){\line(1,0){1}}
\put(2.975,2){\line(0,1){1}}
\put(3.025,2){\line(0,1){1}}
}
\put(2,2){\circle*{0.1}}
\put(3,2){\circle*{0.1}}
\put(3,3){\circle*{0.1}}
\put(0.5,2.5){\red $b(\hat h_1,\hat h_2)$ }
\thicklines
\put(2.5,-0.1){\line(0,1){2.3}}
\put(2.9,2.5){\line(1,0){7.3}}
\put(2.15,-.5){$\hat h_1$}
\put(10.2,2.75){$\hat h_2$}
\end{picture}
\end{center}
\end{minipage}
\vskip0.5cm
\end{example}

We also need the following result on the structure of the support of the cocycle.
\begin{lemma}\label{lem:sets} The support of $c(u_1,u_2,u_3)$ is the disjoint union
of the six sets obtained by permuting the indices of 
$[[u_1, u_3]]\smallsetminus\([[u_1, u_2]] \cup [[u_2, u_3]]\)$.
On each of these sets the cocycle is identically equal to $1$ or $-1$.
\end{lemma}
\begin{proof}
Let us first examine the structure of the intersections of the six sets 
appearing in the definition of $c$. 
Observe that for $a,b,i,j \in \{1, 2, 3\}$ and $a\neq b$ and $i\neq j$ we have that 
\bqn
\text{if }a\neq i\text{ and }b\neq j,\text {then }[[u_a, u_b]] \cap [[u_i, u_j]] =\varnothing\,.
\eqn
This is described more clearly by the following diagram:

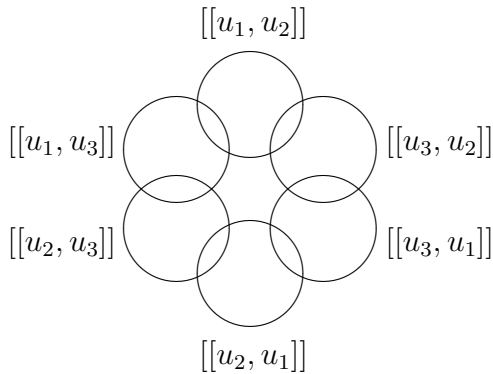
\begin{minipage}{0.7\textwidth}
\begin{center}
\setlength{\unitlength}{0.75cm}
\begin{picture}(5,5)
\put(1.3,0.2){\circle{3}}\put(0.4,-1.4){$[[u_2,u_1]]$}
\put(0,1){\circle{3}}\put(-3,0.6){$[[u_2,u_3]]$}
\put(0,2.4){\circle{3}}\put(-3,2.4){$[[u_1,u_3]]$}
\put(1.3,3.2){\circle{3}}\put(0.4,4.5){$[[u_1,u_2]]$}
\put(2.6,1){\circle{3}}\put(3.7,0.6){$[[u_3,u_1]]$}
\put(2.6,2.4){\circle{3}}\put(3.7,2.4){$[[u_3,u_2]]$}
\end{picture}
\end{center}
\end{minipage}
%
%
\begin{minipage}{0.3\textwidth}
Indeed, consider $s\in [[u_a, u_b]] \cap [[u_i, u_a]]$. 
Then every $h \in s$ must contain $u_b$ and not $u_a$ 
but must also contain $u_a$ and not $u_i$, which shows that the intersection is empty.
Likewise, $[[u_a,u_b]]\cap[[u_b,u_j]]=\varnothing$.
\end{minipage}

\vskip 1cm
However $c$ vanishes on each of the pairwise non-empty intersections. 
Indeed, if $s\in [[u_1, u_2]]\cap [[u_1, u_3]]$, then
$$c({u_1, u_2, u_3})(s) = 0 + 0 + \1_{[[u_1,u_2]]}(s) -(0 +\1_{[[u_1, u_3]]}(s) + 0) = 0\,.$$
The other cases are computed similarly. \end{proof}
We are now ready to prove Proposition~\ref{cisbounded}.
\begin{proof}[Proof of Proposition~\ref{cisbounded}]
According to Lemma~\ref{lem:sets}, an $n$-tuple $h_1\subset h_2\subset\dots\subset h_n$ 
contributing to the cocycle at a triple $u_1,u_2,u_3$ has to be in a set of the type
$$[[u_i,u_j]]\smallsetminus\([[u_i, u_k]] \cup [[u_k, u_j]]\)\,,$$
where $i,j,k\in\{1,2,3\}$ are all different, so that there are six such sets.
In other words, if $(h_1, h_2,\dots, h_n)\in [[u_i,u_j]]\smallsetminus\([[u_i, u_k]] \cup [[u_k, u_j]]\)$,
there exists  $j$, with $1\leq j<n$ such that 
$m(u_1,u_2,u_3)\in h_{j+1}$ but $m(u_1,u_2,u_3)\not\in h_j$.
We hence need to count the number of such $n$-tuples that ``hug'' the median $m(u_1,u_2,u_3)$. 

We start with the case $n=2$ and look at the contribution from the set
$[[u_1, u_2]]_{(2,R)}\smallsetminus\([[u_1, u_3]]_{(2,R)} \cup [[u_3, u_2]]_{(2,R)}\)$.
According to Lemma~\ref{embedintervals}, the interval $\mathcal{I}(u_1,u_2)$ embeds in Euclidean space. 
Let $m=m(u_1,u_2,u_3)\in\mathcal{I}(u_1,u_2)\subseteq{\mathbb{R}}^D$. 
Let  $h_1\subset h_2$, so that $u_2\in h_1 \subset h_2$, $u_1\in h_2^*$ and $m\in h_1^*\cap h_2$. 
According to Corollary~\ref{halfspaceNear},  $B_R(m) \cap (h_1\cup h_2^*)\neq \varnothing$. 
Without loss of generality, we may assume that $B_R(m) \cap h_1\neq \varnothing$. 
Then, there are at most $(2R)^D$ choices for $h_1$; since $h_2$ is at distance $R$ from $h_1$, 
there are at most $(2R)^D$ choices for $h_2$. 

Hence, since there are 6 terms in the definition of the cocycle and 
each is a characteristic function on a set of at most $(2R)^{2D}$ elements, 
we get $6(2R)^{2D}$ for the bound of the cocycle.

For the general case, 
we count in how many ways we can construct a contributing $n$-tuple from 
$[[u_1, u_2]]_{(n,R)}\smallsetminus\([[u_1, u_3]]_{(n,R)} \cup [[u_3, u_2]]_{(n,R)}\)$, 
call it $h_1\subset\dots\subset h_n$. 
According to the case $n=2$, there are at most $(2(n-1)R)^{2D}$ choices for $h_1$ and $h_n$,
since $h_1$ and $h_n$ are at distance less than or equal to $(n-1)R$ from each other.
Therefore, we must count the possible ways of choosing $h_2, \dots, h_{n-1}$. 
To this end, we note that $h_2, \dots, h_{n-1}$ must belong to the set $\beta_v$ of the bridge $b(\hat h_1,\hat h_n)$ 
between $h_1$ and $h_n$. 
The bridge has length at most $(n-1)R$ and hence there are at most $((n-1)R)^{n-2}$ many choices for $h_2, \dots, h_{n-1}$. 
This means that there are at most $(2(n-1)R)^{2D}((n-1)R)^{n-2}$ contributing $n$-tuples 
from $[[u_1, u_2]]_{(n,R)}\smallsetminus\([[u_1, u_3]]_{(n,R)} \cup [[u_3, u_2]]_{(n,R)}\)$.

Hence, since there are 6 terms in the definition of the cocycle and 
each is a characteristic function on a set of at most $2^{2D}((n-1)R)^{2D+n-2}$ elements, 
we get $6\,(2(n-1)R)^{2D+n-2}$ for the bound of the cocycle.
 \end{proof}
\subsection{Towards the Proof of Theorem~\ref{thm_intro:main}}\label{subsec:towards}
We defined at the beginning of this section the bounded cohomology of
$\Gamma$ with coefficients in $\ell^p(\frakH(X)^n)$ as the
cohomology of the complex of the $\Gamma$-equivariant bounded
functions on the Cartesian product $\Gamma^k$ with values in
$\ell^p(\frakH(X)^n)$.  So far, for any $n\geq 2$ we constructed a 1-parameter family of
$\Gamma$-equivariant cocycles
$c_{(n,R)}:\~X\times\~X\times\~X\to\ell^p(\frakH(X)^n)$ and 
we remarked that a choice of a basepoint
will give a cocycle in $\cb(\Gamma^3,\ell^p(\frakH(X)^n))$.  
We still need to show that the cohomology
class represented by this cocycle does not vanish
if the action is not elementary.
In order to do this, we recall from \cite{Burger_Monod_GAFA,
Monod_book} that, if $(B,\vartheta)$ is a strong $\Gamma$-boundary, 
there is an isometric isomorphism 
\bq\label{eq:bm}
\hb^2(\Gamma,\ell^p(\frakH(X)^n))\cong
\mathcal{Z}\linftyaw(B^3,\ell^p(\frakH(X)^n))^\Gamma\,,
\eq
where the space on the right hand side is the space of $\linfty$
alternating $\Gamma$-equivariant cocycles on $B\times B\times B$, with
the measurability intended with respect to the weak-$*$ topology on
$\ell^p(\frakH(X)^n)$, $1\leq p<\infty$ as a dual of $\mathcal E_p$ 
(see \eqref{eq:e_p}).

We recall from the introduction that a strong $\Gamma$-boundary $(B,\vartheta)$ 
is a Lebesgue  space endowed with a measure class preserving $\Gamma$-action
that is in addition 
\be
\item\label{1} amenable, and
\item\label{2} ``doubly ergodic with coefficients'', namely:
\ee

\begin{definition}\label{defi:doublyrgodic}  Let $\Gamma$ be a group and 
$(B,\vartheta)$ a Lebesque space endowed with a measure class preserving $\Gamma$-action.  
The action of $\Gamma$ on $B$ is {\em doubly ergodic with (Hilbert) coefficients} if 
any weak-$*$ measurable
$\Gamma$-equivariant map $B\times B\to \mathcal E$ into the dual $\mathcal E$ of a separable
Banach (Hilbert) space on which $\Gamma$ acts by isometries is essentially constant.
\end{definition}

One of the advantages of the realization \eqref{eq:bm} is that, 
because of \eqref{2} with $\mathcal E=\ell^p(\frakH(X)^n)$ for $1\leq p<\infty$, 
in degree two there are no coboundaries:
hence showing that a cohomology class does not vanish amounts simply to showing that the
corresponding cocycle is non-zero.  The disadvantage is that realizing the pullback via
$\rho:\Gamma\to\Aut(X)$ of a bounded cohomology class defined on the boundary is possible 
under the condition that there exists a $\Gamma$-equivariant measurable boundary map $\varphi:B\to\~X$ and
that the bounded cohomology class can be represented 
by a Borel measurable alternating cocycle, \cite{Burger_Iozzi_app}.  

It is immediate to verify that the cocycle $c$ defined in this section is
alternating in $(u_1, u_2, u_3)$,
that is to say that if $\sigma$ is a permutation of $\{u_1, u_2, u_3\}$
then 
\begin{equation}\label{c is alternating}
c(\s(u_1, u_2, u_3)) = \mathrm{sign}(\sigma) c(u_1, u_2, u_3)\,.
\end{equation}
The Borel measurability of $c$ is proven in Lemma~\ref{cocycleBorelMeas}.

Furthermore, a strong $\Gamma$-boundary with properties \eqref{1} and \eqref{2} exists for any locally compact and compactly
generated group according to \cite{Burger_Monod_GAFA}, and for
arbitrary locally compact groups with respect to a spread out non-degenerate symmetric measure 
according to \cite{Kaimanovich};
the existence of the boundary map will take up the next section.

%
\section{The Boundary Map}\label{sec:bdrymap}
This section is devoted to the proof of the following theorem,
with an eye to the implementation of the isomorphism in \eqref{eq:bm}.
\begin{theorem}\label{thm:bdrymap}
  Let $\G \to \Aut(Y)$ be a group action on an irreducible finite
  dimensional CAT(0) cube complex $Y$.  Assume the action is essential
  and non-elementary.
  If $B$ is a strong $\Gamma$-boundary, there
  exists a $\G$-equivariant measurable map $\f : B \to \partial Y$
  taking values into the non-terminating ultrafilters in $\partial Y$.
\end{theorem}
%

To realize the isomorphism in \eqref{eq:bm} in our generality,
we will in fact need the following
stronger statement which guarantees the existence of some kind of
boundary map when the action is not assumed to be essential 
and the complex is not necessarily irreducible.


\begin{cor}\label{cor:bdrymap}
  Let $\G \to \Aut(X)$ be a group acting on a finite dimensional
  CAT(0) cube complex $X$.  Assume that there is no finite orbit in
  the visual boundary $\partial_\sphericalangle X$
  and denote by $Y$ the essential core of $X$.  Then there exists
  a $\Gamma$-equivariant measurable map $\f : B \to \partial Y\subseteq\partial X$.
\end{cor}

\begin{proof}
  Since the action of $\Gamma$ has no finite orbit in
  $\partial_\sphericalangle X$, in particular it has no fixed point. 
  Therefore, the essential core $Y$ is not empty,
  \cite[Proposition~3.5]{Caprace_Sageev}, and $\Gamma$ also has no
  finite orbit in $\partial_\sphericalangle Y$ as well.  If
  $Y=Y_1\times\dots\times Y_m$ is the decomposition of $Y$ into a
  product of irreducible subcomplexes, by Lemma~\ref{lem:hereditary
    properties}, $\Gamma$ also has no finite orbit in
  $\partial_\sphericalangle Y_i$, for $i=1,\dots,m$
  and moreover the action on each $Y_i$ is essential.

  If $j=1,\dots,q$, let $\varphi_j:B\to\partial Y_j$
  be the $\Gamma$-equivariant measurable boundary map whose existence
  is proven in Theorem~\ref{thm:bdrymap}.  Since
  $\prod_{j=1}^q\partial Y_j\subseteq\partial Y\subseteq\partial X$, the map
  $\varphi:B\to\partial Y$ defined by
  $\varphi(b):=(\varphi_1(b),\dots,\varphi_q(b))$ has the desired
  properties.
\end{proof}

The idea of the proof of Theorem~\ref{thm:bdrymap} is as follows. Since
$\~X$ is a continuous compact metric $G$-space, 
the space $\Pn(\~X)$ of probability measures on $\~X$ endowed
with the weak-$\ast$ topology
is a subset of the (unit ball in the) dual of the continuous functions on $\~X$.  
By amenability of the $\Gamma$-action on $B$, there exists a
$\Gamma$-equivariant measurable map $\psi:B\to\Pn(\~X)$ into the
probability measures on $\~X$ (see \cite[Proposition 4.3.9]{Zimmer}). 
Each probability measure $\mu$ on $\~X$ divides the set of
halfspaces into ``balanced'' (that is halfspaces such that $\mu(h)=\mu(h^*)$)
and ``unbalanced'' ones.  If all halfspaces are unbalanced this
defines an ultrafilter, hence the map $\psi:B\to\Pn(\~X)$ gives a
$\G$-equivariant map $\psi:B\to\~X$. Since the measure $\vartheta$ on $B$ is ergodic 
so is the push-forward measure on $\~X$. Hence up to measure $0$ the image of
$\psi$ is either in $X$ or in $\partial X$. If it is in $X$ then it is
essentially constant so we get a $\G$-fixed point, hence it had to
land in $\partial X$.  The whole work in the proof will be to exclude the presence
of balanced halfspaces using non-elementary actions assumptions as
well as essentiality.

\subsection{General Preliminary Lemmas Using Ergodicity}\label{subsec:prelim-lemmas}
The following lemma can be thought of as a weaker version of the
statement that a strong $\Gamma$-boundary for a lattice is a 
strong $\Gamma$-boundary
for its ambient group and vice versa.

\begin{lemma}\label{lemma:finite-index-ergodicity} Let $\G$ be a group 
acting on a measure space $(M,\vartheta)$.   
If $\G$ acts ergodically on $(M\times M, \vartheta\times \vartheta)$, 
then every finite index subgroup  $\G_0\leq\G$ acts ergodically on $(M,\vartheta)$. 
\end{lemma}

\begin{proof}
We prove the contrapositive of the statement.   
Let $\G_0 \leq \G$ be a finite index subgroup that does not act ergodically. 

Let $(M_0,\vartheta_0)$ be the Mackey's point realization of the measure algebra generated by the $\Gamma_0$-invariant sets.  
In other words, $M_0$ is a standard measure space equipped with a measurable map $p:M\to M_0$
such that $p_\ast(\vartheta)=\vartheta_0$.  Since, by passing to a finite index subgroup if necessary (that will still act non ergodically), 
$\Gamma_0$ can be taken to be normal in $\Gamma$, this measure algebra is $\Gamma$-invariant and 
hence it defines a $\Gamma$-action on $M_0$ with respect to which the map $p:M\to M_0$ is $\Gamma$-equivariant.  
Hence there is an ergodic action of the finite group $\Gamma/\Gamma_0$ on $M_0$, 
which is therefore an atomic space, but cannot consist of one point (otherwise the $\Gamma_0$-action would be ergodic).  

Now take any point $m_0\in M_0$ and define $A:=p^{-1}(m_0)\subset M$.  By construction $A$ is neither null nor conull and 
$\Gamma_0$-invariant.  Consider the subset 
  $\Cup{[\g] \in \G/\G_0}{} \g A \times \g A \subset M\times M$ (which is well
  defined by the $\G_0$-invariance of $M_0$). This set is
  $\G$-invariant and not null.  Furthermore, it is not conull. Indeed, let $A^c = M\smallsetminus A$ denote the complement. 
  We claim that $A \times A^c\subset \big(\Cup{[\g] \in \G/\G_0}{} \g A \times \g A\big)^c$. 
  Indeed, if there is a $\g'\in \G$ such that $A\times A^c \cap (\g'A\times \g'A)$ has positive measure,
  then $\vartheta(A\cap\gamma'A)>0$ and $\vartheta(A^c\cap\gamma'A)>0$,
  while, by construction, 
$\vartheta(A\cap\gamma A)=\vartheta(A)$ or $\vartheta(A\cap\gamma A)=0$
for all $\gamma\in\Gamma$.
    
 Therefore, the $\Gamma$-invariant set   $\Cup{[\g] \in \G/\G_0}{} \g A \times \g A$ is neither null nor conull and 
 hence the diagonal action of $\G$ on $M\times M$ is not ergodic. 
\end{proof}

\begin{lemma}\label{BtimesBtoC}  Let $C$ be a countable set with a $\G$ action and 
$(B,\vartheta)$ a Lebesgue space with a measure class preserving $\Gamma$-action
that is in addition doubly ergodic with Hilbert coefficients.
  Let $\B:=B$  or to $\B:=B\times B$.
  If $\psi:\B\to C$ is a $\G$-equivariant measurable map, then $\psi$
  is essentially constant.
\end{lemma}

\begin{proof} We prove the assertion for $\B:=B\times B$.  
The assertion for $\B:=B$ follows then from the first one 
applied to the precomposition with the projection $\pi_1:B\times B\to B$ on the first component.

As the action of $\G$ on $B \times B$ is ergodic, 
so is the push-forward measure $\psi_*( \b \times \b)$ 
and hence the image of $\psi$ is supported on an orbit. 
We now assume that the $\G$-action on $C$ is transitive.

If $C$ is finite then there is a finite index subgroup $\G_0$ which acts trivially on $C$. 
But as the $\G_0$ action on $B$ is still ergodic, 
by Lemma~\ref{lemma:finite-index-ergodicity}  
we conclude that the action of $\G_0$ on $C$ is still transitive and hence $C$ is a single point.

Next, assume that $C$ is infinite. This means that the corresponding
generalized Bernoulli action of $\G$ on $2^C$ is ergodic (indeed, it
is weakly mixing) and measure preserving with respect to the standard
Bernoulli measure $\lambda$ on $2^C$ generated by taking 
0 and 1 with equal mass. By the double ergodicity with coefficients
of $B$,  \cite{Burger_Monod_GAFA} (see also \cite[Lemma 2.2]{Bader_Furman_Shaker})
we conclude that the
diagonal $\G$-action on $B\times B \times 2^C$ is ergodic. 
Let $(x,y)\in B\times B$ and $S \subset C$. It is clear that the following
evaluation function is essentially constant as it is invariant under
the diagonal $\G$-action

$$(x,y, S) \mapsto \1_S(\psi(x,y)) \in \{0,1\}\,.$$

By Fubini's Theorem, there is a point $(x_0, y_0) \in B\times B$ so
that for $\lambda$-almost every $\1_S \in 2^C$ the value of $\1_S(\psi(x_0,y_0))$ is identically 0, or 1. 
This gives a contradiction. Indeed, for
any $c\in C$ we know that 
$$
\lambda\(\{\1_S \in 2^C : \1_S(c)) = 0\)\}
= \lambda\(\{\1_S \in 2^C : \1_S(c) = 1\}\)= 1/2\,,
$$ 
in particular for $c_0 := \psi(x_0, y_0)$.
\end{proof}

We apply the previous lemma to the countable set $2_f^{\frakH(X)}$ consisting of 
finite subsets of $\frakH(X)$.

\begin{cor}\label{count->not-ess} Let $\Pn$ be equal to either $\Pn(\~X)$ or 
$\Pn(\~X)\times\Pn(\~X)$.  
If there exists a $\Gamma$-equivariant 
measurable map $\Pn\to 2_f^{\frakH(X)}$, then the $\Gamma$-action on $X$ is not essential.
\end{cor}

\begin{proof} By hypothesis there is a finite $\G$-invariant subset of
  $\frakH(X)$ and in particular, there is a finite $\G$-orbit $\G\cdot h$.
  Then, the corresponding CAT(0) cube complex $X(\G\cdot h)$ is
  finite and by \cite[Proposition~3.2, (i)$\Rightarrow$
  (iii)]{Caprace_Sageev}, the action is inessential.
\end{proof}

\subsection{Heavy and Balanced Halfspaces, and Properties of Their 
Associated Complexes}\label{subsec:heavy-balanced}
Let $\Pn(\~X)$ denote the space of probability measures on $\~X$. If $\mu\in\Pn(\~X)$ define

\begin{eqnarray*}\label{eq:h_mu}
H_\mu:=\{h\in\frakH(X):\mu(h)=\mu(\*h)\}\\
H_\mu^+:=\{h\in\frakH(X):\mu(h)>1/2\}\\
H_\mu^-:=\{h\in\frakH(X):\mu(h)<1/2\}\\
H_\mu^\pm:=\{h\in\frakH(X):\mu(h)\neq1/2\}_.
\end{eqnarray*}

We refer to $H_\mu$ as to the {\em balanced} halfspaces and to $H_\mu^+$ to the 
{\em heavy} halfspaces.
The terms {\em unbalanced} and {\em light} halfspaces are also self-explanatory.

We record a few easy consequences of the definition. 

\begin{lemma}\label{Hmufacts} Let $\mu,\nu\in\Pn(\~X)$ be any two measures.
\begin{enumerate}
\item\label{item:involution} The family $H_\mu$ is closed under the involution 
$h\mapsto\*h$ and the involution is a bijection 
between $H_\mu^+$ and $H_\mu^-$.
\item\label{item:partition} There is the following partition of halfspaces: 
$\frakH(X) = H_\mu \sqcup H_\mu^\pm$, where $H_\mu^\pm=H_\mu^+\sqcup H_\mu^-$.
\item\label{item:hnutightlynested} If $h,k$ belong to $H_\mu$ (resp. $H_\mu^+$ 
or $H_\mu^-$),
then either $h\pitchfork k$ or all halfspaces between $h,k$ are in $H_\mu$
(resp. $H_\mu^+$ or $H_\mu^-$).
\item\label{item:nofacingtriplesinhmu} There are no facing triple of halfspaces in $H_\mu$.  
If $X$ is not Euclidean if follows that $H_\mu^+\neq\varnothing$.
\item\label{item:hmuspread} If $X$ is not Euclidean, $H_\mu$ and $H_\nu$ are not empty and $H_\mu\cap H_\nu=\varnothing$, 
then $H_\mu\cap H_\nu^\epsilon\neq\varnothing$ for $\epsilon\in\{+,-\}$.
\item\label{item:measurezero} If $h,k \in H_\mu$ are two parallel halfspaces with $h\subset k$ 
then $\mu(\*h\cap k) =0$.
\item\label{item:equivariance}  The assignments $\mu \mapsto H_\mu$ and 
$\mu \mapsto H_\mu^\epsilon$, 
for $\epsilon\in\{+,-\}$, are $\Aut(X)$-equivariant 
for the natural actions on $\Pn(\~X)$ and $2^{\frakH(X)}$. 
\end{enumerate}
 \end{lemma}
 
 \begin{proof}  Assertions \eqref{item:involution}, \eqref{item:partition} and 
\eqref{item:hnutightlynested} are obvious.  
 
To see \eqref{item:nofacingtriplesinhmu}, 
assume that $h_1,h_2,h_3$ were a facing triples of halfspaces in $H_\mu$, 
so that $\*h_2\subset h_1$,
$\*h_3\subset h_1$ and $\hat h_2\|\hat h_3$.  This would imply that
 $1/2=\mu(h_1)\geq \mu(\* h_2)+\mu(\*h_3)=1$, a contradiction.  
Since $X$ is not {{Euclidean}}, and hence there are facing
 triple of halfspaces, then $H_\mu^\pm\neq\varnothing$ and also $H_\mu^+\neq\varnothing$.
 
Assertion \eqref{item:hmuspread} follows from the fact that if $H_\mu\cap H_\nu=\varnothing$, 
then  $H_\mu\subset H_\mu^\pm$.  But then, since $H_\mu$ is invariant 
under the involution $h\mapsto\*h$, 
both $H_\mu\cap H_\nu^+$ and $H_\mu\cap H_\nu^-$ must be non-empty.

Assertion \eqref{item:measurezero} is immediate since $\mu(k)=\mu(\*h)+\mu(h\cap k)$ and 
$\*h,k \in H_\mu$
and \eqref{item:equivariance} is immediate from the definitions.
  \end{proof}
 
 It follows from Lemma~\ref{lem:embedding} with $W:=H_\mu^+$ and $\frakH_W:=H_\mu$
 that there is an isometric embedding $\~X(H_\mu)\hookrightarrow\~X$ and,
 to simplify the notation, we denote by $\~X_\mu$ its image in $\~X$ 
 ($\~X_{H_\mu^+}$ in the notation of Lemma~\ref{lem:embedding}).

 We remark again that if $H_\mu=\varnothing$, 
 then $\~X_\mu$ is a single vertex in $\~X$.  Notice moreover that
 the $H_\mu^+$ are not $\Gamma$-invariant and the subcomplex
 $\~X_\mu\subset\~X$ is not $\Gamma$-invariant.

\begin{lemma}\label{HmuIsEuclidean}
The complex $\~X(H_\mu)$ is an interval. 
\end{lemma}




\begin{proof}
Let us consider the projection 
  $p:\~X\to\~X(H_\mu)$ and let $\a_0 \in  \supp(p_\ast\mu)$.  
%
  Let $\a_0^*$ be the ``opposite'' of $\a_0$ (in $H_\mu$). Observe that $\a_0^*$ is an
  ultrafilter on $H_\mu$: indeed, the only nontrivial condition we
  must check is that if $h \in \a_0^*$ and $h\subset k$ then $k \in  \a_0^*$. 
If instead $k\notin \a_0^*$, then $k \in \a_0$ which means
  that $\*h \cap k$ is an open neighborhood of $\a_0$, contradicting
  that $\a_0$ is in the support of $\mu$ with
  Lemma~\ref{Hmufacts}\eqref{item:measurezero}.  
  By construction, $H_\mu = [\a_0, \a_0^*]\cup [\a_0^*, \a_0]$, where 
  the intervals are taken in $\~X(H_\mu)$.
\end{proof}

\begin{definition}
Let $\frakH'$ be a subset of $\frakH(X)$. An element $h \in \frakH'$ is called:
 \begin{itemize}
\item  {\em minimal in $\frakH'$} if for every $k\in \frakH'$ either $k \pitchfork h$, $h \subset k$,  
or $h \subset \*k$;
\item  {\em maximal in $\frakH'$} if for every $k\in \frakH'$ either $k \pitchfork h$, $k\subset h$, 
or $\*k \subset h$, 
that is to say, $h$ is maximal if $\*h$ is minimal;
\item {\em terminal in $\frakH'$} if it is either maximal or minimal.
\end{itemize}
\end{definition}
\begin{remark}\label{FinitelyManyTerminalInH-mu}
  The number of terminal elements is bounded above by $2d$ not just
  for $H_\mu$ but for any union of pairwise
  incomparable chains in $H_\mu$.
\end{remark}
\subsection{Proof of Theorem~\ref{thm:bdrymap}}\label{subsec:proof-bdrymap}
\begin{proof}[Proof of Theorem~\ref{thm:bdrymap}]
  Since the $\Gamma$-action on $(B,\vartheta)$ is amenable, there
  exists a $\Gamma$-equivariant measurable map $\psi:B\to\Pn(\~X)$
  into the probability measures on $\~X$.  We consider $\Pn(\~X)$
  endowed with the push-forward of the quasi-invariant, doubly-ergodic
  measure $\vartheta$ on $B$, so that $\Gamma$ acts ergodically on
  $\Pn(\~X)\times\Pn(\~X)$.  We will show that under the hypotheses of the
  theorem, we may associate to every $\mu$ in the image of $\psi$
  a point in $\partial X$ and the composition will be the required
  $\Gamma$-equivariant measurable boundary map $\varphi:B\to\partial X$.

  The map $C_1:\Pn(\~X)\to\N\cup\{\8\}$ defined by 
  $\mu\mapsto|H_\mu|$ is measurable (Corollary~\ref{MeasN}\eqref{N1}) and
  $\Gamma$-equivariant, hence by ergodicity it is essentially
  constant.

\underline{\tt I. $H_\mu=\varnothing$ for almost all $\mu$}

If the essential value of $C_1$ is $0$, then for almost every
$\mu\in\Pn(\~X)$, $H_\mu=\varnothing$.  
This means that, up to measure 0, the image of $\psi$ lies in the set $
\mathcal{E} := \{\mu \in \Pn(\~ X) : H_\mu = \varnothing\}$. 
Thus we have a well defined composition
$\f:B\to\mathcal{E}\to\~X$ defined by
$x\mapsto\psi(x)\mapsto\~X_{\psi(x)}$,
whose image is the single point $\~X_{\psi(x)}\in\~X$.
(Lemma~\ref{lem:embedding}). 
Measurability is guaranteed by Lemma~\ref{muassign},
and Lemma~\ref{2hto2X}.  The equivariance under $\Gamma$ follows from
Lemma~\ref{Hmufacts}\eqref{item:equivariance}.
Proposition~\ref{prop:values in non-terminating} will show that, in
fact, $\varphi$ takes values into the non-terminating ultrafilters of
$X$.

The rest of the proof will consist in showing that all other cases cannot occur.

\underline{\tt II. $0<|H_\mu|<\8$ for almost all $\mu$}

If the essential value of $C_1$ were to be finite, then
Corollary~\ref{count->not-ess} with $\Pn=\Pn(\~X)$ would imply that
the action is not essential.

\underline{\tt III. $|H_\mu|=\8$ for almost all $\mu$}

To deal with this case we consider the $\Gamma$-equivariant and
measurable function $C_2:\Pn(\~X)\times \Pn(\~X)\to \N\cup\{\infty\}$, defined by
$(\mu,\nu)\mapsto|H_\mu\cap H_\nu|$  (Corollary~\ref{MeasN}(\eqref{N2}).
Again by ergodicity of the $\Gamma$-action on $\Pn(\~X)\times\Pn(\~X)$, 
the function $C_2$ is essentially constant.

\underline{\tt III.a $0<|H_\mu\cap H_\nu|<\8$ for almost all $\mu,\nu$}

If the essential value of $C_2$ were finite and non-zero,  then
Corollary~\ref{count->not-ess} with $\Pn=\Pn(\~X)\times\Pn(\~X)$ would
again imply that the action is not essential.

\underline{\tt III.b $|H_\mu\cap H_\nu|=0$ for almost all $\mu,\nu$}

Now suppose that the essential value of $C_2$ is $0$, so that for
almost every $\mu,\nu\in\Pn(\~X)$, $H_\mu\cap H_\nu=\varnothing$.  Let
us consider the measurable (Corollary~\ref{MeasN}\eqref{T}) $\Gamma$-equivariant function
$T:\Pn(\~X)\times \Pn(\~X)\to\N\cup\{\infty\}$, defined by
$$
T(\mu,\nu):=\big|\tau((H_\mu \cap H_\nu^+ )\cup (H_\nu \cap H_\mu^+))\big|\,,
$$  
where \bq\label{eq:terminal} 
\tau:2^{\frakH(X)}\to2^{\frakH(X)} 
\eq 
is the map that assigns to a subset of halfspaces its terminal elements.
By double ergodicity $T$ is essentially constant. Using the fact that
both $H_\mu$ and $H_\nu$ are Euclidean, any subset of them must have
finitely many terminal elements and therefore this essential value
must be finite (see Remark~\ref{FinitelyManyTerminalInH-mu}).  Once
more, essentiality of the action, along with
Corollary~\ref{count->not-ess} assures us that the essential value is
0.

This leaves us with the case in which the
essential value is zero, that is $H_\mu\cap H_\nu^+$ has no terminal
elements for almost every $(\mu, \nu)$.  In this case the following
proposition (whose proof we postpone to \S~\ref{subsec:more-lemmas}) allows
us to conclude that this case cannot happen.

\begin{prop}\label{III.C}  Suppose that for almost every 
$\mu,\nu\in\Pn(\~X)$, $|H_\mu|=|H_\nu|=\8$, $H_\mu\cap H_\nu=\varnothing$ 
and $H_\mu\cap H_\nu^+$
  has no minimal elements.  Then $X$ contains cubes of arbitrarily
  large dimension.
\end{prop}

\underline{\tt III.c $|H_\mu\cap H_\nu|=\8$ for almost all $\mu,\nu$}

Finally, let us suppose that the essential value of $C_2$ is $\8$,
namely $|H_\mu\cap H_\nu|=\8$ for almost every
$(\mu,\nu)\in\Pn(\~X)\times\Pn(\~X)$.

If $H_\mu=H_\nu$ for almost every $\mu,\nu\in\Pn(\~X)$, then applying
Fubini, there is a $\mu_0\in\Pn(\~X)$ such that for every $\nu$ in a
co-null $\G$-invariant subset we have that $H_{\mu_0} =H_{\nu}$.
Hence, $H_{\g_\ast\nu}= \g H_{\nu}= H_{\nu}$.  Since the action is
essential without fixed points on the visual boundary, we may flip any
$h\in H_\nu^+$. This means that $H_\nu^- \cap H_{\g_\ast\nu}^+\neq
\varnothing$ and so $H_\nu^+$ is not $\G$-invariant.  As a result the
corresponding embedded subcomplexes $\overline{X}_{\g_\ast\nu}$ are
not invariant.  We will see in Proposition~\ref{prop:product} that
this implies that $X$ is a product, which is a contradiction.

We are therefore left in the case in which $H_\mu \cap H_\nu$ is
infinite but $H_\mu \neq H_\nu$, for almost all $\mu,\nu\in\Pn(\~X)$.
 
We now consider whether or not $H_\mu$ has strongly separated halfspaces.
Observe that the set 
\bqn
\mathcal{S} = \{(h_1, h_2)\in\frakH(X)\times\frakH(X) : 
h_1, h_2 \text{ are strongly separated}\}\eqn
is $\G$-invariant. Therefore, 
the map $\mu \to |(H_\mu\times H_\mu)\cap \mathcal S|$ is
measurable (Corollary~\ref{MeasN}\eqref{Nnu}) and $\G$-invariant, 
and hence essentially constant.
 
If $H_\mu$ contains pairs of strongly separated halfspaces then $H_\mu^+$
satisfies the Descending Chain Condition (Lemma~\ref{lem:SSimpliesDCC}).
This implies that the action is again inessential by extracting the
finitely many terminal elements of the set 
$(H_\mu^+\cap(H_\nu\smallsetminus H_\mu)) \cup (H_\nu^+\cap(H_\mu\smallsetminus H_\nu))$
(Corollary~\ref{MeasN}\eqref{N3}), 
and we proceed as before to conclude that the action is inessential. 
 
If on the other hand $H_\mu$ does not contain pairs of strongly separated halfspaces,
then by  Corollary~\ref{no terminal elements intersection} the action is inessential. 
\end{proof}

\subsection{Further Properties and Proofs}\label{subsec:more-lemmas}

\begin{prop}\label{prop:values in non-terminating} Let $X$ be a 
finite dimensional CAT(0) cube complex, $\Gamma\to\Aut(X)$ an essential action on $X$,
$(B,\nu)$ a doubly ergodic $\Gamma$-space with quasi-invariant measure $\nu$
and $\varphi:B\to\~X$ a measurable $\Gamma$-equivariant map.
Then $\varphi$ takes values in the non-terminal ultrafilters of $X$. 
\end{prop}

We start with few easy observations.  
Recall that, if $\alpha$ and $\beta$ are two ultrafilters,
\bqn
\frakH(\alpha,\beta):=[\alpha,\beta]\cup[\beta,\alpha]=[\alpha,\beta]\cup[\alpha,\beta]^\ast\,.
\eqn
Then it is easy to check that 
\bq\label{eq:finite}
\tau(\frakH(\alpha,\beta))=\tau([\alpha,\beta])\cup\tau([\alpha,\beta]^\ast)
\eq
and hence $|\tau(\frakH(\alpha,\beta))|$ is finite.

\begin{lemma}\label{lem:tau}  Let $\alpha$ and $\beta$ be two ultrafilters and $h\in\tau(\alpha)$.
Then $\beta\notin h$ if and only if $h\in\tau(\frakH(\alpha,\beta))$.
\end{lemma}

\begin{proof} If $\beta\in h$, then $h$ does not separate $\alpha$ and $\beta$,
so that $h\notin \frakH(\alpha,\beta)$ and, even more so, $h\notin\tau(\frakH(\alpha,\beta))$.
The converse is equally easy and will not be needed.
\end{proof}

\begin{proof}[Proof of Proposition~\ref{prop:values in non-terminating}]
We may assume that $X$ is irreducible.
The general case will follow from this case as in the proof of Corollary~\ref{cor:bdrymap},
since the set of non-terminating ultrafilters in a product is the cartesian product of the
sets of non-terminating ultrafilters of each factor.

The composition of $\varphi$ with the map $\tau$ defined in \eqref{eq:terminal}
that assigns to a set of halfspaces its terminal element, gives a $\Gamma$-equivariant
measurable map $B \to 2^{\frakH(X)}$ defined by $x \mapsto \tau(\phi(x))$.
The function $C_4:B\to\N\cap\{\infty\}$ defined by $x\mapsto|\tau(\phi(x))|$
is hence essentially constant.  

Therefore, we want to show that $|C_4(x)|=0$ for almost every $x$,  
that is that the set $\tau(\phi(x))$ is empty, thus showing that $\varphi(x)$ is
non-terminating.

To this purpose let us consider the map $\theta:B\times B\to 2^{\frakH(X)}$
that to a pair $(x,y)\in B\times B$ associates the set of terminal elements
in $\frakH(\varphi(x),\varphi(y))$.
Again by ergodicity the function $C_5:B\times B\to\N\cup\{\infty\}$,
defined by $C_5(x,y):=|\tau(\frakH(\varphi(x),\varphi(y)))|$ is essentially constant
and, by \eqref{eq:finite}, 
$0\leq |\tau(\frakH(\varphi(x),\varphi(y)))|<\infty$.

By Corollary~\ref{count->not-ess} with $\Pn=B$, we deduce that 
for almost every $x,y\in B$, $\tau(\frakH(\varphi(x),\varphi(y)))=\varnothing$.
We show now that this is incompatible with $|\tau(\varphi(x))|>0$
for almost every $x\in B$, thus proving the proposition.

Let $x_0\in B$ be such that $|\tau(\varphi(x_0))|>0$ and let $B_0\subset B$
be a set of full measure such that $\tau(\frakH(\varphi(x_0),\varphi(y)))=\varnothing$
for all $y\in B_0$.  Then by Lemma~\ref{lem:tau}, if $h\in\tau(\varphi(x_0))$,
we must have that $\varphi(y)\in h$ for all $y\in B_0$. But $B_0$ contains
a $\Gamma$-orbit and hence this contradicts the fact that the action 
is essential.
%
%
%
%
\end{proof}

\subsubsection{Proof of Proposition~\ref{III.C} (in step III.b)}

We will find arbitrarily a large family of pairwise intersecting halfspaces.
  To this purpose, choose a sequence $\{\mu_i\}_{i\in\N}$ of pairwise generic measures
  satisfying the hypotheses of Proposition~\ref{III.C}. For each
  $i$, choose an infinite descending chain $h_n^i \in H_{\mu_0}^+ \cap H_{\mu_i}$.  

Consider the following property of an ordered pair $(\mu_i,\mu_j)$ of measures:
\begin{itemize}
\item[$(*)$]
There exists $C(i,j)\in\N$ such that for every $n\geq C(i,j)$
there is an $M_n\geq C(i,j)$ such that if $m>M_n\geq C(i,j)$, 
then $\hat h_n^i \pitchfork \hat h_m^j$.
\end{itemize}

\begin{lemma}\label{lem:single dychotomy}
Up to switching $i$ and $j$, any pair of measures $\mu_i$ and $\mu_j$, 
satisfies $(*)$.
\end{lemma}

We postpone the proof of this lemma and show how to conclude the proof of Proposition~\ref{III.C}.

Let us consider a graph $\Gn:=\Gn(V,E)$, where $V:=\{\mu_i\}$ and
where two measures $\mu_i$ and $\mu_j$ are connected by an edge $e\in E$
with source $\mu_i$ and target $\mu_j$ if the ordered pair $(\mu_i,\mu_j)$
satisfies  $(*)$.
By Lemma~\ref{lem:sutg}, given $D\in\N$, there exist (relabelled) measures 
$\mu_1,\dots,\mu_D\in\{\mu_i\}_{n\in\N}$ such that for $1\leq i<j\leq D$, 
each ordered pair $(\mu_i,\mu_j)$ satisfies $(*)$.

By choosing 
\bqn
C:=\max{} \{C(i,j):\,1\leq i<j\leq D\}
\eqn
and 
\bqn
M:=\max{} \{M_C(i,j):\,1\leq i<j\leq D\}\,.
\eqn
we obtain that for all $n,m\geq C$ and $1\leq i, j\leq D$, 
the corresponding hyperplanes are transverse, $\hat h_n^i\pitchfork \hat h_m^j$.
This concludes the proof of Proposition~\ref{III.C}.

\begin{proof}[Proof of Lemma~\ref{lem:single dychotomy}]
  Fix two measure that we denote for ease of notation, $\mu$ and $\nu$. 
  Let $h_n\in H_{\mu_0}^+ \cap H_{\mu}$ and $k_m\in H_{\mu_0}^+ \cap H_\nu$
  be the corresponding infinite descending sequences.
  Since all the halfspaces in  question belong to $H_{\mu_0}^+$, 
  for each pair $n,m$ we have the following decomposition
$$
\N\times\N=N_1\sqcup N_2\sqcup N_3\sqcup N_4\,,
$$
where
$$
\begin{aligned}
N_1&=\{(n,m):\,h_n\pitchfork k_m\}\\
N_2&=\{(n,m):\,\*{h_n}\subset k_m\}\\
N_3&=\{(n,m):\,h_n\subset k_m\}\\
N_4&=\{(n,m):\,h_n\supset k_m\}
\end{aligned}
$$

We claim that if we allow ourselves to throw away a finite number of pairs $(n,m)$
if necessary, then the decomposition of $\N\times\N$ takes in fact a simpler shape.
Namely:

\begin{claim}  There exists a constant $C\in\N$ depending on $\mu$ and $\nu$, 
such that 
\bqn
\N_C
:=(\N\times\N)\cap([C,\infty)\times[C,\infty))=N_1\sqcup N_j\,,
\eqn
where $j=2,3$ or $4$.
\end{claim}

In fact, let us suppose that $N_2\neq\varnothing$ and $N_3\neq\varnothing$
and let us take $(n_3,m_3)\in N_3$ and $(n,m)\in N_2$.  Set $m':=\min{}\{m,m_3\}$,
such that 
\bqn
\*h_n\subset k_{m'}\qquad\text{ and }\qquad h_{n_3}\subset k_{m'}\,.
\eqn
If $n\geq n_3$, then $h_n\subset h_{n_3}\subset k_{m'}$, which is impossible
since also $\*h_n\subset k_{m'}$.  Hence there is no pair $(n,m)\in N_2$
such that $n\geq\min{}\{n_3:\,(n_3,m_3)\in N_3\}=:A_3$.  It follows that 
\bq\label{23}
\{(n,m)\in N_2:\,n\geq A_3\}\cap N_3=\varnothing\,.
\eq

Now let us suppose that $N_3\neq\varnothing$ and $N_4\neq\varnothing$
and let us take $(n_3,m_3)\in N_3$ and $(n,m)\in N_4$.  
If $n\geq n_3$, then
\bqn
k_n\subset h_n\subset h_{n_3}\subset k_{m_3}\,,
\eqn
which is impossible by Lemma~\ref{Hmufacts} part (\ref{item:hnutightlynested}).
Hence, analogously to the previous case, we have that 
\bq\label{34}
\{(n,m)\in N_2:\,n\geq A_3\}\cap N_4=\varnothing\,.
\eq

Finally, let us suppose that $N_2\neq\varnothing$ and $N_4\neq\varnothing$
and let us take $(n,m)\in N_2$ and $(n_4,m_4)\in N_4$.
Set $n':=\min{}\{n,n_4\}$, so that 
\bqn
h_{n'}\supset k_m\qquad\text{ and }\qquad h_{n'}\supset k_{m_4}\,.
\eqn
If $m\geq m_4$ then $h_{n'}\supset k_{m_4}\subset k_m$, which is impossible
since also $h_{n'}\supset \*k_m$.  Hence there is no pair $(n,m)\in N_2$
such that $m<\min{}\{m_4:\,(n_4,m_4)\in N_4\}=:B_4$.  It follows that 
\bq\label{24}
\{(n,m)\in N_2:\,m\geq B_4\}\cap N_4=\varnothing\,.
\eq
By setting $C:=\max{}\{A_3, B_4\}$ we have proven the claim.

Let us suppose now that for $n_0,m_0\geq C$, the pair $(n_0,m_0)\in N_1\sqcup N_3$
and, in fact, that $(n_0,m_0)\in N_3$ (otherwise there is nothing to prove).
Choose $m_0=m_0(n_0)$ to be the largest integer such that $(n_0,m_0)\in N_3$.
Then, because $\N_C=N_1\sqcup N_3$, for every $m\geq m_0(n_0)+1$,
the hyperplanes $\hat h_{n_0}$ and $\hat k_m$ are transverse.   Hence the assertion 
of the lemma is proven in the case in which $\N_C=N_1\sqcup N_3$.

Remark that the same identical argument shows the assertion if $\N_C=N_1\sqcup N_2$,
since we only used that there is a minimal element $k_{m_0}$ in the sequence $k_m$
that contains the hyperplane $\hat h_{n_0}$.

The argument if $\N_C=N_1\sqcup N_4$ is analogous, but with the role of $h_n$ and $k_m$
reversed, as now there is a minimal element $h_{n_0}$ in the sequence $h_n$
that contains the hyperplane $\hat k_{m_0}$.  Namely, let $(n_0,m_0)\in N_4$ be such a pair.
Then for all $n\geq n_0(m_0)+1$, the hyperplanes $\hat h_n$ and $\hat k_{m_0}$ are transverse.
\end{proof}

\begin{remark}  The last assertion in the proof relative to the case $\N_C=N_1\sqcup N_4$
holds also for the case $\N_C=N_1\sqcup N_2$, but the symmetry of this case is not useful.
\end{remark}
\subsubsection{Proofs needed in step III.c}\label{subsubsec:step III.c}


\begin{lemma}\label{lem:AY} Let $X$ be a CAT(0) cube complex and 
$A$ be any cubical subset of $X$ 
(that is, $A$ is a union of cubes, not necessarily connected). 
If $Y$ is the smallest strongly convex subcomplex of $X$ containing $A$, 
then 
\bqn
\hat\frakH(Y)=\hat\frakH(A)\sqcup\{\hat h\in\hat\frakH(Y):\,\hat h\text{ separates }
A\text{ in at least two non-trivial subsets}\}\,.
\eqn
\end{lemma}
\begin{proof} We only need to check that every 
$\hat h\in\hat\frakH(Y)\smallsetminus\hat\frakH(A)$
separates $A$ in non-trivial subsets. 
Take $\hat h=(h,h^\ast)\in \hat\frakH(Y)$ and assume by contradiction that $A\subseteq h$. 
Then any geodesic between two points of $A$ is also contained in $h$ 
(otherwise this geodesic would cross $h$ twice. Hence $Y\subseteq h$, 
contradicting that $h\in \hat\frakH(Y)$.
\end{proof}
%

\begin{prop}\label{prop:product} Let $X$ be a CAT(0) cube complex and $\Gamma\to\Aut(X)$ 
an essential action.
Let $\frakH'\subset\frakH(X)$ be a $\Gamma$-invariant subset of halfspaces and
$X_\alpha\subset X$ a $\Gamma$-invariant family of subcomplexes 
such that $\hat\frakH(X_\alpha)=\hat\frakH'$.
Let $Y$ be the smallest strongly convex subcomplex containing $A:=\cup_\alpha X_\alpha$.
Then $Y=X$ and $\overline{X}=\overline{X(\hat\frakH')}\times Z$.
\end{prop}

%
\begin{proof} Since $Y$ is $\Gamma$-invariant and the action is essential, then $Y=X$.

Because of Lemma~\ref{lem:AY}, the hyperplanes of $Y$ are of two types: 
either they are in $\hat\frakH'=\hat\frakH(A)$ and they separate one (equivalently, any) 
of the $X_\alpha$ 
or they separate a $X_\alpha$ from a $X_{\alpha'}$. 
Any hyperplane $\hat h$ of this second type will cross any hyperplane $\hat k\in\hat\frakH'$. 
Indeed, if $\hat h=(h,h^\ast)$ and $\hat k=(k,k^\ast)$ 
it is easy to see that the four intersections in \eqref{eq:transverse} are non-empty. 
Hence $X$ is a product.
\end{proof}

\begin{lemma}\label{lem:SSimpliesDCC}
If $|H_\mu|=\8$ and $H_\mu$ contains strongly separated halfspaces, then
$H_\mu^+$ satisfies the Descending Chain Condition. 
\end{lemma}

\begin{proof}
  Let $h,k \in \frakH(X)$ be a pair of strongly separated halfspaces in
  $H_\mu$ with $h\subset k$.  There is the following decomposition
\begin{equation}\label{eq:decomposition}
H_\mu^+ =P(h)\cup P(k)\,,
\end{equation}
where $P(h)$ and $P(k)$ are the $\mu$-heavy halfspaces that are
parallel respectively to $h$ and $k$.  Notice that, while $P(h)$ and
$P(k)$ are not necessarily disjoint, their union is the whole of
$H_\mu^+$ since $h$ and $k$ are strongly separated.
 
Let $h_n \in H_\mu^+$ be a descending chain, i.e. $h_{n+1} \subset h_n$.  
We must show that the chain terminates.  By passing to a
subsequence, we may assume that $h_n$ belong to the same set for all
$n\in\N$ and it is hence enough to consider for example the case
$h_n\in P(h)$ for all $n\in\N$.
 
Since $h_n\in H_\mu^+$ and $h\in H_\mu$, we cannot have that
$h_n\subset h$ or $h\subset\*h_n$.  Let us suppose that $h\subset h_n$.  
Since between $h$ and $h_n$ there are only finitely many
halfspaces, and since no $\mu$-heavy halfspace can be contained in a
balanced one, the chain must terminate.  Likewise the same argument
applied to $\*h\subset h_n$ shows that the chain must terminate.
\end{proof}

\begin{lemma}\label{lem:StronglySeparatedTrap}
  For every measure $\mu$ either $\hat H_\mu$ contains a pair of
  strongly separated hyperplanes or there exists a pair $h\in H_\mu^-$, $k\in H_\mu^+$ 
of halfspaces, such that the hyperplanes
  $\hat h$ and $\hat k$ are strongly separated and 
for every $x\in  H_\mu$, $\hat x\subset\*h\cap k$.
\end{lemma}

\begin{proof} Suppose that $H_\mu$ does not contain strongly separated
  halfspaces.  We first show that for every $x\in H_\mu$, there exist
  $k_0(x), k_3(x)\in H_\mu^\pm$ such that $\hat k_0(x)$ and $\hat  k_3(x)$ 
are strongly separated and $\hat x\subset \*k_0(x)\cap  k_3(x)$.  
For ease of notation we drop the dependence on $x$.

In fact, since $X$ is irreducible, given $x\in H_\mu$, there exist
halfspaces $k_1, k_2$ such that $\hat k_1$ and $\hat k_2$ are
strongly separated hyperplanes and $k_1\subset x\subset k_2$.  Then at
least one between the $k_1$ and $k_2$ must be in $H_\mu^\pm$, but
perhaps not both of them.  Then double skewer $k_2$ into $k_1$ and
$\*k_1$ into $\*k_2$ to obtain
$$
\gamma k_2\subset k_1\subset x\subset k_2\subset \gamma^{-1}k_1\,,
$$
where the pairs $\gamma k_2, k_2$ and $k_1,\gamma^{-1}k_1$ are
strongly separated.  Since all hyperplanes corresponding to pairs of
halfspaces in the sequence $k_0\subset k_1\subset k_2\subset k_3$ are
strongly separated, there can be at most one halfspace that belongs
to $H_\mu$.  By measure considerations, this halfspace can only be
either $k_1$ or $k_2$, so that $k_0,k_3\in H_\mu^\pm$, and the
assertion is proven.  In particular $\gamma k_1\in H_\mu^-$ and $\gamma^{-1}k_2\in H_\mu^+$.

Double skewer once again to get $h\in H_\mu^-$ and $k\in H_\mu^+$ with
$\hat h$, $\hat k$ strongly separated, such that
$$
h\subset k_0\subset x\subset k_3\subset k\,.
$$

We show now that, given any $y\in H_\mu$, we have $\hat y\subset\*h\cap k$.  
In fact, we cannot have $k\subset y$ or $k\subset\*y$,
since $y,\*y\in H_\mu$ and $k\in H_\mu^+$.  Analogously, we cannot
have that $\hat y\pitchfork \hat k$, because otherwise $\hat y$ could
not intersect $\hat k_3$ and hence it would have to contain it, which
is impossible again by measure considerations.  Hence $y\subset k$.
An analogous argument shows that $h\subset y$, thus completing the
proof.
\end{proof}

The above argument can be extended to show the following:

\begin{lemma}\label{lem:StrongSeparatedTrapPair}
  Let $\mu_i \in \Pn(\~X)$ be measures such that $\hat H_{\mu_i}$ does
  not contain strongly separated hyperplanes for all $i$ and 
  $H_{\mu_i} \cap  H_{\mu_j}\neq \varnothing$ for all $i,j$.  Then there exists a pair
  of halfspaces $h\subset k$ such that $\hat h, \hat k$ are strongly
  separated and, for every $x\in H_{\mu_j}$, $\hat x\subset \*h\cap k$.
\end{lemma}

\begin{proof}
Fix $\mu_0$ and apply Lemma~\ref{lem:StronglySeparatedTrap} 
to find halfspaces $h_2\subset h_3$ such that 
$\hat h_2, \hat h_3$ are strongly separated and 
\begin{equation}\label{eq:containment}
\hat x\subset \*h_2\cap h_3\,.
\end{equation}
Use the double skewering lemma several times to find a chain $h_0
\subset \cdots \subset h_5$ of halfspaces with corresponding pairwise
strongly separated hyperplanes.  We will use that
\eqref{eq:containment} holds in particular for every $x_j \in H_{\mu_0}\cap H_{\mu_i}$ 
to show that $\hat y\subset \*{h_0}\cap h_5$
for every $y \in H_{\mu_j}$ and every $j$.
 
Consider in fact $y\in H_{\mu_i}$. Observe that $\hat y$ can be
transverse to at most one $\hat h_i$, $0\leq i\leq 5$, since these are
pairwise strongly separated.  If it is transverse to any $\hat h_i$
for $1\leq i\leq 4$, we are done, since then $\hat y\subset\*h_0\cap h_5$.  
Suppose instead that $\hat y$ is transverse to $\hat h_0$.
Then $\hat h_1$ and $h_2$ are nested in between $\hat y$ and $\hat x_j$, 
which is impossible by Lemma~\ref{Hmufacts} part(\ref{item:hnutightlynested}) and because $\hat H_{\mu_j}$ does not contain strongly separated
hyperplanes.  A similar argument shows that $\hat y$ cannot be
transverse to $\hat h_5$.

If instead $\hat y$ is parallel to all $\hat h_i$, for $0\leq i\leq 5$, 
then we have to check that $\hat y\subset h_0$ and $\hat y\subset\*h_5$ cannot happen.  
If fact, if $\hat y\subset h_0$, as
before this would force $\hat h_1,\hat h_2$ to be in $\hat H_{\mu_j}$,
which is not possible because they are a strongly separated pair.  The
case in which $\hat y\subset\*h_5$ can be excluded analogously.
\end{proof}

\begin{cor} \label{no terminal elements intersection} Assume that for
  almost every $\mu\in\Pn(\~X)$, there are no strongly separated pairs in
  $H_\mu$.  If $H_\mu \cap H_\nu\neq \varnothing$ for almost every
  pair $(\mu,\nu)$ then the $\G$-action is non-essential.
\end{cor}

\begin{proof}
  Fix a generic measure $\mu_0$ with a generic $\G$-invariant set $B_0$
  such that for every $\nu\in B_0$ we have that $H_{\mu_0} \cap H_\nu  \neq \varnothing$.

  Lemma~~\ref{lem:StrongSeparatedTrapPair} implies the existence of a
  pair of halfspaces $h\subset k$ such that $\*h\cap k$ contains all
  the hyperplanes in $\hat H_{\mu}$ for $\nu \in B_0$, in particular
  those in $\gamma\hat H_{\mu_0}=\hat H_{\gamma_\ast\mu_0}$ for all
  $\gamma\in\Gamma$.  This shows that the two halfspaces $h$ and $k$
  are not $\G$-flippable, which contradicts either that the action is
  essential or that it is without fixed points on the CAT(0) boundary
  \cite[Theorem~4.1]{Caprace_Sageev}.
\end{proof}




\section{Proof of Theorem~\ref{thm_intro:main} and Theorem~\ref{thm:converse}}\label{sec:proofs}
Let $X$ be a finite dimensional CAT(0) cube complex,
$\Gamma\to\Aut(X)$ a non-elementary action and 
let $Y\subset X$ be the essential core of $X$.
Let $(B, \vartheta)$ be any strong $\Gamma$-boundary. 
In order to prove our main result, 
we constructed in \S~\ref{sec:bdrymap} a measurable $\Gamma$-equivariant boundary map 
$\f: B\to \partial X$ to the Roller boundary $\partial X$.
The precomposition of the median cocycle $c$ with
$\f: B \to \partial X$ yields a $\G$-equivariant cocycle
defined on $B^3$, which we will show is non-zero on a set of positive measure
(Proposition~\ref{Roller to Poisson} and Lemma~\ref{lem:nonvanishing1}). 
According to \eqref{eq:bm}, this ensures the
existence of a non-trivial cohomology class on $\G$.
Then \cite{Burger_Iozzi_app} ensures that the median class 
of the $\Gamma$-action $\rho^*({\tt m}_{(n,R)})\in\hb^2(\Gamma,\ell^p(\frakH(X)^n))$, 
$n\geq2$, corresponds to the cohomology class $c\circ\f^3$ on $B^3$ 
and hence does not vanish.

\subsection{Passing from a Cocycle on $\partial X$ to a Cocycle on $B$}\label{non-zero} 
We give a condition on a $\G$-equivariant cocycle $d: (\partial X)^3 \to E$ to guarantee 
that $d\circ \f^3 : B^3 \to E$ is non-zero on a set of positive measure,
where $\varphi:B\to\partial X$ is a measurable $\Gamma$-equivariant map.

Let $K$ be compact metrizable $\Gamma$-space.  
A measure $\lambda\in\Pn(K)$ is quasi-invariant
if $\lambda$ and $\gamma_\ast\lambda$ have the same null sets, 
for all $\gamma\in \Gamma$. 

If $h\in\frakH(X)$ is a halfspace, we set
\bqn
\~h:=\{x\in\~X:x\in h\}
\eqn
and
\bqn
\partial h:=\~h\cap\partial X\,.
\eqn

\begin{prop}\label{Roller to Poisson}
Let $\G$ be a group with a non-elementary and essential action $\Gamma\to\Aut(Y)$
on a finite dimensional CAT(0) cube complex $Y$. 
If $(B, \vartheta)$ is a strong $\Gamma$-boundary, 
let $\f : B \to \partial Y$ be a $\G$-equivariant measurable map.
Let $d: (\partial Y)^3 \to E$ be an everywhere defined alternating
bounded $\G$-equivariant Borel cocycle with values in a coefficient 
$\Gamma$-module $E$. 
If there exist halfspaces $h_i \in \frakH(Y)$ 
such that $d(\xi_1, \xi_2, \xi_3)\neq 0$ 
for every $(\xi_1, \xi_2, \xi_3) \in \partial h_1\times \partial h_2 \times \partial h_3$ 
then $d\circ \f^3$ is a non-trivial element of $\mathcal{Z}\linftyaw(B^3,E)^\Gamma$.
\end{prop}

The proof is almost an immediate consequence of the following:

\begin{lemma}\label{lem:positive measure} Let $Y$ be a finite dimensional CAT(0) cube complex 
  and $\Gamma\to\Aut(Y)$ a non-elementary essential action. If
  $\lambda\in\Pn(\partial Y)$ is any quasi-invariant probability measure
  then $\lambda(\partial h)>0$ for any halfspace $h\in \frakH(Y)$.
\end{lemma}

\begin{proof}  If $\lambda(\partial h)=0$ then $\lambda(\partial\*h)=1$. 
By the Flipping Lemma ~\cite[Theorem~4.1]{Caprace_Sageev}, 
there exists $\gamma\in\Gamma$ such that $\*h\subset\gamma h$.  
But this is a contradiction
because ${\partial\*h}\subset\partial(\gamma h)=\gamma\partial h$, 
while, by quasi-invariance, $\lambda(\gamma\partial h)=0$.
\end{proof}

\begin{proof}[Proof of Proposition~\ref{Roller to Poisson}]
  If there exist halfspaces $h_i \in \frakH(Y)$ such that
  $d(\xi_1, \xi_2, \xi_3)\neq 0$ for every 
$(\xi_1, \xi_2, \xi_3)  \in \partial h_1\times \partial h_2 \times \partial h_3$ then
  $d\circ \f^3(x_1, x_2, x_3)\neq 0$ for almost every 
$(x_1, x_2, x_3)\in 
\f^{-1}( \partial h_1)\times \f^{-1}(\partial h_2) \times\f^{-1}(\partial h_3)=:S\subset B^3$.

By Lemma~\ref{lem:positive measure} applied to the quasi-invariant probability measure
$\varphi_*\vartheta\in\Pn(\partial Y)$, the set $S$ has positive $\vartheta^3$-measure
and hence $d\circ\varphi^3$ it is a non-trivial element of
  $\mathcal{Z}\linftyaw(B^3,E)^\Gamma$.
\end{proof}
 \subsection{Proof of Theorems~\ref{thm_intro:main} and~\ref{thm:converse}}\label{ProofMainThm} 
We start with a lemma ensuring that the cocycle is non-vanishing on a set of positive measure.
\begin{lemma}\label{lem:nonvanishing1}
Let $X$ be a finite dimensional CAT(0) cube complex, with a non-elementary action $\Gamma\to\Aut(X)$. 
Then for every essential $h\in\frakH(X)$ and for every $n\geq 2$, 
there is a positive measure set $A_{(h,n)}\subset\partial X^3$ and $R_h>0$ so that 
for every $R>R_h$, the restriction $c_{(n,R)}|_{A_{(h,n)}}$ does not vanish.
\end{lemma}
\begin{proof} Let $\gamma$ and $\gamma'\in\G$ be 
such that the triple $h,\gamma h, \gamma' h$ is an \"uber-parallel  facing triple in an orbit as in Lemma~\ref{lem:sss facing triple}. 
By the Flipping Lemma there exists $\mu\in\Gamma$ be such $\mu h^*\subset h$ and let
$\eta\in\Gamma$ be an element that skewers $\mu h$ into $\gamma h$, so as to obtain
\bqn
\mu h\supset h^*\supset\gamma h\supset\eta\mu h\,.
\eqn
The pair $\mu h $ and $\mu\eta h$ is \"uber-parallel and hence so is any consecutive pair in the sequence
\bqn
\mu h\supset\eta\mu h\supset\eta^2\mu h\supset\dots\supset\eta^{n-1}\mu h\,.
\eqn
Because of Lemma~\ref{lem:positive measure}, the set
\begin{equation*}
A_{(h,n)}:=\partial(\mu h^*) \times\partial(\eta^{(n-1)}\mu h)\times\partial(\gamma' h)\
\end{equation*}
has positive measure.  Since for every $(\xi_1,\xi_2,\xi_3)\in A_{(h,n)}$, the set
$[[\xi_3,\xi_2]]_{(n,R)}\smallsetminus([[\xi_3,\xi_1]]_{(n,R)}\cup[[\xi_1,\xi_2]]_{(n,R)})$ is not empty,
provided $R$ is larger that the translation length $R_h$ of $\eta$, 
Lemma~\ref{lem:sets} insures that $c_{(n,R)}|_{A_{(h,n)}}$ does not vanish.
\end{proof}

%
\begin{proof}[Proof of Theorem~\ref{thm_intro:main}] Let $Y$ be the essential core of the $\G$-action on $X$.
According to \cite{Burger_Iozzi_app}, 
the class of $c_{(n,R)}\circ\f^3$ is the isometric image 
of the median class ${\tt m}_{(n,R)}$ under the isomorphism \eqref{eq:bm},
where $\varphi:B\to\partial Y$ is the boundary map constructed in 
Theorem~\ref{thm:bdrymap}. Let $R_\G=\min\{R_h\,|\,h\hbox{ essential}\}$ and where the $R_h$'s are as defined in Lemma~\ref{lem:nonvanishing1}. Then, 
Proposition~\ref{Roller to Poisson} and Lemma~\ref{lem:nonvanishing1}
ensure that $c_{(n,R)}\circ\f^3$ is non-trivial if $n\geq2$ and $R\geq R_{\G}$.
\end{proof}
\begin{proof}[Proof of Theorem~\ref{thm:converse}] 
If the $\Gamma$-action is elementary, by definition there exists 
a finite orbit in $X\cup\partial_\sphericalangle X$.  
If the finite $\Gamma$-orbit is in $X$, then there is a subgroup 
of finite index $\Gamma_0<\Gamma$ that fixes a point $x\in X$.
Hence the median class of the $\Gamma_0$-action on $X$ vanishes.
Since the map $\hb^2(\Gamma,\ell^p(\frakH(X)^n))\to\hb^2(\Gamma_0,\ell^p(\frakH(X)^n))$
is injective \cite{Monod_book}, the median class of the $\Gamma$-action vanishes.

If on the other hand there exists a finite $\Gamma$-orbit in $\partial_\sphericalangle X$,
then we can apply Proposition~\ref{prop:visual-to-roller} and deduce that
either there is a finite orbit in $\partial X$ -- in which case we conclude as in the first part
and the median class of the $\Gamma$-action vanishes -- or there exists a
finite index subgroup $\Gamma'<\Gamma$ and 
a $\Gamma'$-invariant subcomplex $X'\subset\partial X$
in which the $\Gamma'$-action is non-elementary.  
By Theorem~\ref{thm_intro:main} the median class of the
$\Gamma'$-action on $X'$ does not vanish, and,
since $X'$ corresponds to a lifting decomposition of halfspaces, 
by Propositon~\ref{prop:restriction} 
it is the restriction of the median class of the $\Gamma$-action on $X$. 
\end{proof}



\section{Applications}\label{sec:appl}
\subsection{Rigidity of Actions}
\begin{proof}[Proof of Theorem~\ref{thm:appl_intro}]  
We will need our action to satisfy the property that
the barycenter of every face has trivial stabilizer.  This is a
natural generalization to CAT(0) cube complexes of the notion of no
edge inversions in the context of actions on trees.  For this reason,
we start with an arbitrary action and then pass to its
cubical subdivision.

Let $Y'$ be the cubical subdivision of $Y$. 
Observe that $\Aut(Y) \hookrightarrow \Aut(Y')$ and 
the image acts with the property that the barycenter of every face in $Y'$
has trivial stabilizer.  Moreover $\Gamma$ acts essentially on $Y'$
and $Y'$ is irreducible.

The proof follows very closely the strategy of the proof
in \cite{Shalom}.  Namely, if we denote by $e_i$ the identity in $G_i$, 
we aim to show that there is an $i\in\{1,\dots,\ell\}$ for which the set
\begin{equation*}
\begin{aligned}
Y_i:=\{x\in Y':\,&\text{ if }\gamma_m\in\Gamma\text{ such that }\pri(\gamma_m)\to e_i,\\
&\text{ then there exists }N>0\text{ such that }\gamma_mx=x\text{ for all }m\geq N\}
\end{aligned}
\end{equation*}
is not empty, where $\pri: G\to G_i$ is the $i$-th projection.  It is
easy to see that the set $Y_i$ is $\Gamma$-invariant.  Indeed, if
$\gamma_m\in\Gamma$ is a sequence such that $\pri(\gamma_m)\to e_i$,
then for every $\gamma\in \Gamma$ we have that
$\pri(\gamma^{-1}\gamma_m\gamma)\to e_i$.  Moreover $Y_i$ is convex
with respect to the CAT(0) metric:  in fact, let $x_1,x_2\in Y_i$ and let
$\gamma_m\in\Gamma$ a sequence such that $\pri(\gamma_m)\to e_i$.
Then, by definition of $Y_i$ there exists $N$ sufficiently large such
that $\gamma_mx_j=x_j$ for all $m\geq N$ and $j=1,2$. Since $\G$ acts
by isometries, if $m\geq N$ then $\g_m$ will also fix pointwise the
unique CAT(0) geodesic between $x_1$ and $x_2$.

We claim now that, if $Y_i$ is not empty, then it is in fact a subcomplex of $Y'$.  
To see this, let us write $Y'$ as the disjoint union of {\em $k$-dimensional faces}, 
where a $k$-dimensional face is the interior of a
$k$-dimensional cube if $1\leq k\leq\dim(Y)$ and is the boundary of a
$1$-dimensional cube if $k=0$.  Let $F_k$ be a $k$-dimensional face
which has non-empty intersection with $Y_i$. Then 
$F_k\subset Y_i$.  Since we are acting on the cubical subdivision
$Y'$ we have also that if $\gamma_m$ eventually fixes a face $F_k$,
then it fixes all lower dimensional faces that are contained in its
closure $\overline{F_k}$, thus showing that $\overline{F_k}\subset
Y_i$.  Thus $Y_i$ is a CAT(0) cube subcomplex of $Y'$.

We are then left to show that there exists $i\in\{1,\dots,\ell\}$ such
that $Y_i\neq\varnothing$.

Let $i\in\{1,\dots,\ell\}$ and let
$\H_i\subset\ell^2(\frakH(Y)^n)$ be the (possibly trivial)
subspace on which the isometric action of $\Gamma$ extends
continuously to $G$ via the projection $\pri:G\twoheadrightarrow G_i$.
By \cite[Theorem 16]{Burger_Monod_GAFA}, $\hb^2(\Gamma,\ell^2(\frakH(Y)^n)=\bigoplus_{i=1}^\ell\hb^2(G_i,\H_i)$, 
so that by Theorem~\ref{thm_intro:main} there exists an  $i\in\{1,\dots,\ell\}$ (not necessarily unique) such that $\H_i\neq \{0\}$.

  The space $Y_i$ will be constructed from this data as follows.

  Define on $\frakH(Y)^n$ an equivalence relation, namely if
  $s,s'\in\frakH(Y)^n$, we say that $s\sim s'$ if $f(s)=f(s')$ for
  all $f\in \H_i$.  Since these functions are square summable, all
  of the equivalence classes are finite, with the possible exception of the
  class where all functions in $\H_i$ vanish.  Moreover $\Gamma$
  permutes all the finite equivalence classes and leaves invariant the
  only infinite one (if it exists). Therefore, the complement $\frakH(Y)^n_0$
  of the infinite class in $\frakH(Y)^n$ is $\Gamma$-invariant.

\begin{claim}\label{claim:stab} Let $[s]\in\frakH(Y)^n_0/\sim$ be 
a finite equivalence class and 
  $\stab_\Gamma([s])$ its stabilizer.  If $\gamma_m\in\Gamma$ is a
  sequence such that $\pri(\gamma_m)\to e_i$, then there exists $N>0$
  such that for all $m\geq N$, $\gamma_m\in\stab_\Gamma([s])$.
\end{claim}

We assume the claim for the moment and show
that $Y_i$ is not empty.  Fix $[s]\in \frakH(Y)^n_0$ and let $\hat
h$ be a hyperplane corresponding to one of the halfspaces appearing
in an element of $[s]$.  By Lemma~\ref{lem:ss}, there exists
$\gamma\in\Gamma$ such that $\hat h$ and $\gamma\hat h$ are strongly
separated.  By \cite[Lemma~2.2]{Behrstock_Charney}, 
the bridge $b(\hat h,\gamma\hat h)$ consists of a single
geodesic (of finite length).  Observe that, since $\frakH(X)^n_0$
is invariant, the class $[\gamma s]$ is finite.  
Hence there are finitely many halfspaces in the set
$\{h:\,h\in s',\text{ for } s'\in[s']\}$, both if $s'=s$ and if $s'=\gamma s$.
It follows that if we define $L:= \stab_\Gamma([s])\cap \stab_\Gamma([\gamma s])$,
then the $L$-orbit of $b(\hat h,\gamma\hat h)$ is finite, therefore bounded,
and its circumcenter is an $L$-fixed point.

If $\gamma_m\in\Gamma$ is a sequence such that
$\pri(\gamma_m)\to e_i$, Claim~\ref{claim:stab} implies that,
for $m$ large enough, the sequence $\gamma_m$ is in $L$
and hence fix the circumcenter, thus showing that $Y_i\neq\varnothing$.
Then Proposition~4.3 in \cite{Shalom} shows that the action of $\Gamma$ 
on $Y_i$ extends to $G$ by factoring through $G_i$.

Since the action of $\Gamma$ on $Y'$ is essential, 
then $Y_i=Y'$.  Observe however that, since $\Aut(Y)$ is closed in $\Aut(Y')$
in the topology of the pointwise convergence, 
then the extension of the action to $G$ is in $\Aut(Y)$.
\end{proof}

\begin{proof}[Proof of Claim~\ref{claim:stab}]  
Let $s\in\frakH(Y)^n_0$ and let $f\in\H_i$ so that $f(s)\neq0$.
Since $\lim_{m\to\infty}\|\gamma_mf-f\|_2=0$, 
then $\lim_{m\to\infty}f(\gamma_ms)=f(s)$.
Because $f$ is square summable, 
it takes finitely many values in a $|f(s)|/2$-neighborhood of $f(s)$, 
so that we conclude that there exists $N(f,s)$ such that 
$f(\gamma_ms)=f(s)$ for all $m\geq N(f,s)$.
In particular $\{\gamma_ms:n\geq1\}$ is finite.  

If $\gamma_{m_k}s\not\sim s$ for some subsequence $m_k$, then, 
by passing to a further subsequence, we may assume that $s_0:=\gamma_{m_k}s\not\sim s$.
But then there is $g\in\H_i$ such that $g(s)\neq0$ and $g(s_0)\neq g(s)$,
which, together with 
\bqn
g(s_0)=
\lim_kg(\gamma_{m_k}s)=g(s)\,,
\eqn
is a contradiction.
\end{proof}

The proof of the above theorem does rely on the assumption that $Y$ is irreducible 
and essential.  In general, we can pass to the essential core $Y$
of the $\Gamma$-action on $X$
and to its cubical subdivision $Y'$.
Let $\Gamma'<\Gamma$ be the finite index subgroup that acts on
each of the irreducible factors in $Y'$ and let $G'_i:=\~{\pri(\Gamma')}$.
By applying Theorem~\ref{thm:appl_intro}
to each of the irreducible factors we obtain that the action of $\Gamma'$
on $Y$ extends continuously to an action of $G'$,
where $G'=G'_1\times\dots\times G'_\ell$, by factoring via one of the factors.\qed

We have hence proven the following:
%
%

\begin{cor}\label{thm:appl}
Let $X$ be a finite dimensional CAT(0) cube complex
and $\Gamma$ be an irreducible lattice in the 
product of locally compact groups $G_1\times\dots\times G_\ell=:G$.
Let $\Gamma\to\Aut(X)$  be a non-elementary action on $X$.
Then the action of $\Gamma$ on the essential core of $X$
virtually extends to a continuous action of an open finite index subgroup in $G$, 
by factoring via one of the factors.
\end{cor}

\subsection{The Class $\mathcal{C}_{reg}$}

We now prove Corollary~\ref{cor:creg_intro} 
concerning the class of groups $\mathcal{C}_{reg}$.  
The idea of the proof  is as follows.  If the action
is proper, then in particular the vertex stabilizers are finite.  We
can then find, for each $n>1$ an $n$-tuple  $s\in \frakH(X)^n$
with finite stabilizers  such that 
$\ell^p(\G\cdot s) \hookrightarrow \ell^p(\Gamma)$. 
We then prove that $s$ can be chosen 
in such a way that this map does not vanish on the image of the cocycle.

\begin{proof}[Proof of Corollary~\ref{cor:creg_intro}]
Let $h\subset k$ be two strongly separated halfspaces in $\frakH$.  
Since the action of $\G$ is also by
CAT(0)-isometries, the stabilizer of $\{h,k\}$ must also stabilize
their CAT(0) bridge. Since the action is proper, the stabilizer of the
CAT(0) bridge $b(\hat h,\hat k)$ is finite.  It follows that if $s= (h_1, \dots, h_n)$ 
is a \"uber-separated sequence of halfspaces of consecutive distance less than or equal to $R$ then $\stab_\G(s)$, 
the stabilizer of $s$ in $\Gamma$ is finite. 

Fix $s\in \frakH(Y)^n$, and consider the   $\Gamma$-equivariant map
\bq\label{eq:sigma}
\sigma_s:\ell^p(\Gamma\cdot s)\to\ell^p(\Gamma)
\eq
defined by $\sigma_sf(\gamma):=f(\gamma s)$.  
Since $\|\sigma_sf\|_p=|\stab_\G(s)|\,\|f\|_p$, 
the map is injective.


Now, observe that $\ell^p(\frakH(Y)^n) = \underset{s\in S}{\oplus} \ell^p(\G\cdot s)$ where $S$ is a choice of $\G$-orbit representatives. 

Since the cocycle $c: \~Y\times \~Y\times \~Y\to \ell^p(\frakH(Y)^n)$ is $\G$-equivariant, for every $s$, 
we may post-compose $c$ with $\sigma_s$ and precompose with the boundary map $\f: B \to \partial Y$ and obtain:

$$c_s = \sigma_s \circ c\circ\f^3: B\times B\times B\to \ell^p(\G\cdot s).$$

We now choose $s$ so as to guarantee that $c_s$ is non vanishing on a set of positive measure:

As in the proof of Lemma~\ref{lem:nonvanishing1}, we find fix $h\in \frakH(Y)$ and find $\g, \g'\in \G$ so that $h^*, \g h,\g'h$ form a  facing \"uber-parallel triple. 
Let $s = (h, \g h, \dots, \g^{n-1}h)$. Then, $c_s$ restricted to the set 

$$A_{(h,n)} = \partial h\times \partial(\g^{(n-1)}h^*) \times \partial(\g' h^*)$$

is nonzero, as is demonstrated in the proof of Lemma~\ref{lem:nonvanishing1}.

\end{proof}





\appendix

\section{Some More Proofs}
\subsection{The Measurability of Certain Key Maps}\label{subsec:measurable maps}
The notation used in this section refers to \S~\ref{sec:bdrymap}.

\begin{lemma}\label{muassign}
  Let $I\subseteq [0,1]$ be a subinterval, possibly open, half open,
  or closed. Let $H^I_\mu = \{h \in \frakH(X): \mu(h) \in I\}$ The map $\Pn(\~X)\to 2^{\frakH(X)}$,
  defined by $\mu\mapsto H^I_\mu$ is measurable with
  respect to the weak-$*$ topology on $\Pn(\~X)$.  
  
  As a consequence,  the map $N :\Pn(\~X) \to \N\cup \{\8\}$  
defined by $N(\mu) = |H^I_\mu|$ is measurable.
\end{lemma}

\begin{proof}
  Recall the definition of cylinder sets: Let $F_1, F_2 \in 2^{\frakH(X)}$ be
  two finite sets. The cylinder set associated to them is $$C(F_1,
  F_2) = \{H \in 2^{\h}: F_1 \subseteq H \text{ and } F_2 \subseteq
  H^c\}.$$

  Cylinder sets form a basis for the topology on $2^{\frakH(X)}$. Therefore it
  is sufficient to show that $$K(F_1,F_2) = \{\mu: H^I_\mu \in C(F_1,  F_2)\}$$ is measurable.

  To this end, observe that $h$ is open and closed as a subset of $\~X$ so that
  its characteristic function $\1_h$ is continuous. Therefore, the
  set $E_I(h) = \{\mu: \mu(h) \in I\}$, for $h\in \frakH(X)$ is weak-$*$
  open, half open, or closed according to $I$ and therefore is
  measurable. Then the following finishes the proof 
$$ K(F_1,F_2) = \Cap{h\in F_1}{} E_I(h) \cap \Cap{h\in F_2}{} (E_I(h))^c\,.
$$
\end{proof}

\begin{cor}\label{MeasN} The following maps are measurable:
\begin{enumerate}
\item\label{N1} $C_1:\Pn(\~X)\to\N\cup\{\infty\}$, defined by $C_1(\mu):=|H_\mu| ;$\\
\item\label{N2} $C_2:\Pn(\~X)\times\Pn(\~X)\to\N\cup\{\infty\}$, 
defined by $C_2(\mu,\nu):=|H_\mu\cap H_\nu| ;$\\
\item\label{T} $T:\Pn(\~X)\times\Pn(\~X)\to\N\cup\{\infty\}$, 
defined by $$T(\mu,\nu):=\big|\tau((H_\mu\cap H_\nu^+)\cup(H_\nu\cap H_\mu^+))\big| ;$$
\\
\item\label{Nnu} $N_\nu:\Pn(\~X)\to\N\cup\{\infty\}$, 
defined by $N_{\nu}(\mu):=|(H_\mu\times H_\nu)\cap\mathcal{S}|$, 
where 
$$\mathcal{S}:=\{(h_1, h_2)\in\frakH(X)\times\frakH(X) : h_1, h_2 \text{ are strongly separated}\}.
$$\\
\item\label{N3} $C_3:\Pn(\~X)\times\Pn(\~X)\to\N\cup\{\infty\}$, 
defined by 
$$C_3(\mu,\nu):=\big|[H_\mu^+\cap(H_\nu\smallsetminus H_\mu)] \cup [H_\nu^+\cap(H_\mu\smallsetminus H_\nu)]\big|.
$$\\
\end{enumerate}
\end{cor}

\begin{proof}
  The proofs of statements (1), (2), and (5) are consequences of Lemma
 ~\ref{muassign} and the observation that the product and composition
  of $\G$-equivariant measurable maps is again $\G$-equivariant and
  measurable. Statement (4) follows by considering the fact that
  $\mathcal{S}$ is a $\G$-invariant set. 
Statement (3) follows from the following result. 
\end{proof}

\begin{lemma}\label{2hto2X}
 The map $p:2^{\frakH(X)} \to 2^X$ defined by $\1_S \mapsto \1_{\Cap{h\in S}{} h}$ is measurable.
\end{lemma}

\begin{proof}
  Choose an enumeration $\frakH(X)= \{h_n: n \in \N\}$. Recall that the
  standard projection $\pi_N : 2^{\frakH(X)} \to 2^{\{h_1, \dots, h_N\}}$ is
  continuous. Next define $p_N : 2^{\{h_1, \dots, h_N\}}\to 2^X$ as
  $\1_S \mapsto \1_{(\Cap{h\in S}{} h)}$ which is also continuous as
  $2^{\{h_1, \dots, h_N\}}$ is endowed with the discrete topology.

  Observe that $p(\1_S) = \sup\{p_N\circ \pi_N(\1_S) :n\in \N\}$ and
  is hence measurable as the supremum of continuous functions.
\end{proof}

Recall that in \eqref{eq:e_p} we set  the notation 
$\mathcal E_p:=\ell^q(\frakH(X)^{n})$ if $1/p+1/q=1$ and $1<p<\infty$,
and $\mathcal E_1$ to be the Banach space of functions on $\frakH(X)^{n}$
that vanish at infinity.

\begin{lemma}\label{cocycleBorelMeas}
for all $1\leq p<\infty$, the cocycle $c: \~X^3 \to \ell^p(\mathfrak{H}^{n})$ is a Borel map,
where $\~X^3\subset 2^{\frakH(X)}$ has the induced product topology 
and $\ell^p(\frakH(X)^{n})$ has the weak-$*$ topology as the dual of $\mathcal E_p$. 
\end{lemma}

\begin{proof}
  Choose an enumeration of $\mathfrak{H}$ and let 
$\mathfrak{H}_N :=  \{ h_1, \dots, h_N\}$.  Let us define the finite space 
\bqn
  \mathfrak{H}^{n}_N:= \{s\in \mathfrak{H}^{n}: \,s\subset
  \mathfrak{H}_N\}\,, 
\eqn
 and, for any subsets $E, F \subset\mathfrak{H}$, the set 
\bqn (E\smallsetminus
  F)^{n}_N:=\{s\in\mathfrak{H}^{n}_N:\,s\subset E\text{ and
  }s\not\subset F\}\,.
  \eqn
  The map $c^+_N: 2^\mathfrak{H}\times 2^\mathfrak{H} \times
  2^\mathfrak{H} \to C_0(\mathfrak{H}^{n})$ defined as
$$c^+_N(F_1,F_2,F_3) := \1_{(F_3\smallsetminus F_2)^{n}_N}
                     + \1_{(F_1\smallsetminus F_3)^{n}_N}
                     + \1_{(F_2\smallsetminus F_1)^{n}_N}$$
factors through the canonical projection $2^\mathfrak{H} \to 2^{\mathfrak{H}_N}$ on triples 
and hence is continuous. 
Then the map $c_N(F_1,F_2,F_3) := c_N^+(F_1,F_2,F_3) - c_N^+(F_1,F_3,F_2)$ is also continuous and 
in particular is continuous 
when restricted to the subset $\~X^3 \subset \(2^\mathfrak{H}\)^3$. 
(Here we use the identification of a vertex $v\in X$ with
the principal ultrafilter containing $v$.)

For any $f \in \mathcal E_p$, the function on $\~X^3$ defined by 
$$(x,y,z)\mapsto\<c_N(x,y,z), f\>$$
is continuous.  Its pointwise limit 
$$(x,y,z)\mapsto\Lim{N\to \8} \<c_N(x,y,z), f\>$$
is measurable and, in fact, 
$$\<c(x,y,z), f\> = \Lim{N\to \8} \<c_N(x,y,z), f\>\,,$$
thus showing that the cocycle $c$ restricted to $\~X^3$ is Borel. 
\end{proof}

\subsection{A Lemma in Graph Theory}\label{subsec:graph theory}
Let $\Gn(V,E)$ a complete directed finite graph\footnote{In graph theory a graph with 
these properties is called {\em turnament}.} with vertices $V$ and edges $E$.
We denote by $s,t:E\to V$ respectively the {\em source} and the {\em target} 
of an edge.
We allow the possibility that there are two edges between two vertices, one
in each direction.
Given a vertex $v\in V$, we denote by $o(v)$ (respectively $i(v)$) 
the number of outgoing (respectively incoming) edges at $v$.  Since the graph 
is complete, 
\bq\label{eq:complete graph}
o(v)+i(v)\geq|V|-1\,,
\eq
for every $v\in V$.  

The next lemma shows that if the graph is complete, there is at least 
one vertex that has ``many" outgoing edges.

\begin{lemma}\label{lem:graph}  If $\Gn:=\Gn(V,E)$ is a complete directed 
finite graph and $|V|=D$,
then there exists $v\in V$ such that $o(v)\geq\frac{D-1}{2}$.
\end{lemma}
\begin{proof} From \eqref{eq:complete graph} we have that
\bqn
\sum_{v\in V}o(v)+i(v)\geq D(D-1)\,.
\eqn
We have also that
\bqn
\sum_{v\in V}o(v)=\sum_{v\in V}i(v)\,,
\eqn
so that 
\bqn
\sum_{v\in V}o(v)\geq\frac{D(D-1)}{2}\,.
\eqn
Since $|V|=D$, the assertion follows readily.
\end{proof}

\begin{definition}
We say that a complete directed finite graph $\Gn(V,E)$ with $|V|=D$
is {\em strictly upper triangular}\footnote{Or {\em transitive turnament}.} 
if there exists a numbering
$v_1,\dots,v_{D}$ of its vertices, such that for all $j=1,\dots,D$,
\bqn
\ba
o(v_j)=&D-j\\
i(v_j)=&j-1
\ea
\eqn
\end{definition}
The terminology is inspired from the fact that 
the corresponding $D\times  D$ adjacency matrix $M$ with coefficients
\bqn
M_{ij}:=\begin{cases}
1&\qquad\text{if there exists }e\in E\text{ with }s(e)=v_i\text{ and }t(e)=v_j\\
0&\qquad\text{otherwise}\,,
\end{cases}
\eqn
is strictly upper triangular, namely $v_1$ is connected by an outgoing vertex to 
all of the remaining $v_2,\dots, v_{D}$, $v_2$ is connected to $v_3,\dots, v_{D}$ 
and so on.

\begin{example} A strictly upper triangular graph with $d=2$ corresponds to the matrix
$M=\begin{pmatrix}0&1&1\\0&0&1\\0&0&0\end{pmatrix}$
and is of the form
\vskip-2cm
\hskip7cm
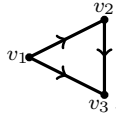
{\begin{tikzpicture}[very thick,decoration={
    markings,
    mark=at position 0.5 with {\arrow{>}}},
scale=.5]
\begin{scope}
    \draw[postaction={decorate}] (0,0) -- (2,1);
    \draw[postaction={decorate}] (0,0) -- (2,-1);
    \draw[postaction={decorate}] (2,1) -- (2,-1);
\draw (-0.3,0) node {\tiny{$v_1$}};
\draw (2,1.3) node {\tiny{$v_2$}};
\draw (2,-1.3) node {\tiny{$v_3$}\,.};
  \filldraw[ultra thick] 
  	(0,0) circle [radius=.05] 
        (2,1) circle [radius=.05] 
       (2,-1) circle [radius=.05];
       \end{scope}
\end{tikzpicture}}
\end{example}

\begin{lemma}\label{lem:sutg} 
Let $\Gn=\Gn(V,E)$ be a complete directed graph (not necessarily finite) and $D\in\N$.
If $|V|\geq5^{D}$, there exist $D$ vertices $v_1,\dots,v_{D}$
such that the induced complete directed subgraph on $v_1,\dots, v_{D}$
is strictly upper triangular.
\end{lemma}

\begin{proof} The idea is to construct the strictly upper triangular
  graph inductively.  By Lemma~\ref{lem:graph} there exists $v_1\in V$
  with $o(v_1)\geq\frac{|V|-1}{2}\geq\frac{|V|}{5}$ outgoing edges.
  Denote by \bqn O(v_1):=\{e\in E:\,s(e)=v_1\} \eqn the set of
  outgoing edges (so that $|O(v_1)|=o(v_1)$).  We consider now the
  induced complete directed subgraph $\Gn(v_1)$ on $v_1$ and on the
  vertices at the end of the edges in $O(v_1)$, namely
  $\Gn_1:=\Gn(V(v_1), E(v_1))\subset\Gn$, where \bqn
  V(v_1):=\{v_1\}\sqcup\{t(e):\,e\in O(v_1)\} \qquad\text{ and }\qquad
  O(v_1)\subsetneq E(v_1)\subseteq E\,, \eqn and the edges $E(v_1)$
  are exactly the edges in $E$ needed to make complete the graph on
  the vertices $V(v_1)$.  Remark that, by construction,
  $|V(v_1)|=o(v_1)+1\geq\frac{|V|}{5}\geq 5^{D-1}$ and the induced
  complete directed subgraph on any ordered pair of vertices $v_1,v$,
  with $v\in V(v_1)$, is (trivially) a strictly upper triangular
  graph.

  We could now proceed to formulate a rigorous proof by induction, but
  we prefer showing how to move to the next step, as we believe that
  the very simple idea of the proof will be more transparent.

  We repeat now exactly the same construction as before, applied to
  the graph $\Gn_1(V(v_1),E(v_1))$ instead of $\Gn(V,E)$.  Namely ,
  let $v_2\in V(v_1)$ be the vertex, whose existence is asserted by
  Lemma~\ref{lem:graph}, such that if \bqn O(v_2):=\{e\in
  E(v_1):\,s(e)=v_2\}\,, \eqn then
  $o(v_2)=|O(v_2)|\geq\frac{|V(v_1)|-1}{2}\geq\frac{|V(v_1)|}{5}$.  By
  construction there is an outgoing edge from $v_1$ to $v_2$ and from
  $v_1$ to any other vertex in $\Gn_1$.  From the graph $\Gn_1$ we
  retain now only those vertices $w\in V(v_1)$ that are at the end of
  an outgoing edge from $v_2$, so that $v_1,v_2,w$ is strictly upper
  triangular, and eliminate all of the other vertices.  Namely, let
  $\Gn_2$ be the induced complete directed graph on $v_2$ and on the
  vertices that are the targets of the $o(v_2)$ edges in $E(v_1)$
  outgoing from $v_2$.  That is $\Gn_2:=\Gn(V(v_2),
  E(v_2))\subset\Gn_1$, where \bqn V(v_2):=\{v_2\}\sqcup\{t(e):\,e\in
  O(v_2)\} \qquad\text{ and }\qquad O(v_2)\subsetneq E(v_2)\subseteq
  E(v_1)\subseteq E\,.  \eqn Now we have a complete directed graph on
  $|V(v_2)|\geq o(v_2)+1\geq\frac{|V(v_1)|}{5}\geq 5^{D-2}$ vertices
  from which we can continue choosing vertices $v_3,\dots, v_{D}$ such
  that at each step we increase by one our strictly upper triangular
  graph and we reduce by a factor of 5 the cardinality of the vertex
  set.
\end{proof}

\vfill\eject

\section{Boundary stabilisers in  {\upshape CAT(0)}  cube complexes\\
{\small by Pierre-Emmanuel Caprace\footnote{UCLouvain, 1348 Louvain-la-Neuve, Belgium. P-E.C. is an F.R.S.-FNRS Research Associate, supported in part by the European Research Council (grant \#278469)}}
}
\label{sec:AppB}


Let $X$ be a (not necessarily proper) {\upshape CAT(0)} cube complex.
A group $\Gamma \leq \Aut(X)$   is called \textbf{locally $X$-elliptic} if every finitely generated subgroup of $\Gamma$ fixes a point of $X$. The goal of this appendix is to establish the following.

\begin{theorem}\label{thm:RollerStab}
Let $X$ be a finite-dimensional  {\upshape CAT(0)}  cube complex and $\alpha$ be a point in the Roller compactification.

Then the stabiliser $\Stab_{\Aut(X)}(\alpha)$ has a  locally $X$-elliptic normal subgroup $N$ such that the quotient $\Stab_{\Aut(X)}(\alpha)/N$ is finitely generated and virtually abelian of rank $\leq \dim(X)$. 
\end{theorem}

The following consequence of  Theorem~\ref{thm:RollerStab} is immediate.

\begin{cor}\label{cor:Rigidity}
Let $\Gamma$ be a finitely generated group satisfying the following two conditions:
\begin{enumerate}
\item Every finite index subgroup of $\Gamma$ has finite abelianization. 

\item For every $\Gamma$-action on a finite-dimensional  {\upshape CAT(0)}  cube complex, there is a finite $\Gamma$-orbit in the Roller compactification.
\end{enumerate}

Then every $\Gamma$-action on a finite-dimensional  {\upshape CAT(0)}  cube complex has a fixed point.
\end{cor}

\begin{remark}
As pointed out to me by Talia Fernos, the converse statement to Corollary B.2 holds as well, namely: A finitely generated group $\Gamma$ all of whose actions on finite-dimensional CAT(0) cube complexes have fixed points, automatically satisfies (1) and (2). Indeed, Property (2) is straightforward, while the existence of a finite index subgroup with infinite abelianization can be used to produce an unbounded action on the standard cubulation of the Euclidean $n$-space, for $n$ large enough.
\end{remark}

The proof of Theorem~\ref{thm:RollerStab} uses a relation between the Roller boundary  and the \textbf{simplicial boundary}  of $X$. The latter boundary was constructed by Mark Hagen in \cite{Hagen}, under the hypothesis that $X$ is finite-dimensional. Before reviewing its construction, we start with an abstract tool that will be used to produce the abelian quotient appearing in Theorem~\ref{thm:RollerStab}. 

Following Yves de Cornulier  \cite{CorComm}, we say that  two subsets $M, N$ of a set $X$ are \textbf{commensurate} if their symmetric difference $M \triangle N$ is finite. We say that $M$ is \textbf{commensurated} by the action of a group $G$ acting on $X$ if for all $g \in G$, the sets $M$ and $gM$ are commensurate. 

\begin{prop}[Proposition~4.H.1 in \cite{CorComm}]\label{prop:ModularChar}
Let $G$ be a group, $X$ be a  discrete $G$-set, and $M \subset X$ be a {commensurated subset}.

Then  the map
$$\tr_M \colon G \to \ZZ : g \mapsto \# (M \setminus g^{-1} M) - \# (g^{-1} M \setminus   M)$$ 
is a homormorphism, called the \textbf{transfer character}. Moreover, for any $N \subset X$ commensurate to $M$ and any $g \in G$, we have $\tr_M(g) = \tr_N(g)$. \qed
\end{prop}

We now turn back to our geometric setting: in the rest of this note we let $X$ be a {\upshape  {\upshape CAT(0)} } cube complex. Its set of halfspaces (resp. hyperplanes) is denoted by $\mathfrak H(X)$ (resp. $\mathfrak W(X)$). The map $\mathfrak H(X) \to \mathfrak W(X) : h \mapsto \hat h$ associates to each halfspace its boundary hyperplane.  In order to define the simplicial boundary, we recall some of Mark Hagen's terminology from \cite{Hagen}. A set of hyperplanes $\mathcal U \subset \mathfrak W(X)$ is called:
\begin{itemize}
\item \textbf{inseparable} if each hyperplane separating two elements of $\mathcal U$ belongs to $\mathcal U$. 

\item the \textbf{inseparable closure} of a set of hyperplanes $\mathcal V$ if $\mathcal U$ is the smallest inseparable set containing $\mathcal V$. 
 
\item \textbf{unidirectional} if for each $\hat h \in \mathcal U$, at least one halfspace bounded by $\hat h$ contains only finitely many elements of $\mathcal U$. 

\item a \textbf{facing triple} if $\mathcal U$ consists of the three boundary hyperplanes of three pairwise disjoint halfspaces.

\item a \textbf{UBS} if $\mathcal U$ is infinite, inseparable, unidirectional and contains no facing triple (UBS stands for \emph{unidirectional boundary set}).  

\item a \textbf{minimal UBS} if every UBS $\mathcal U'$ contained in $\mathcal U$ is commensurate to $\mathcal U$. 

\item \textbf{almost transverse} to a set of hyperplanes $\mathcal V$ if each $\hat h \in \mathcal U$ crosses all but finitely many hyperplanes in $\mathcal V$, and each $\hat k \in \mathcal V$ crosses all but finitely many hyperplanes in $\mathcal U$.
\end{itemize}

The  following result is due to Mark Hagen.

\begin{prop}[Theorem 3.10 in \cite{Hagen}]\label{prop:UBS}
Assume that $X$ is finite-dimensional. 

Given a UBS $\mathcal V$, there exists a UBS $\mathcal V'$ commensurate to $\mathcal V$ such that $\mathcal V'$ is partitioned into a finite union of minimal UBS, say $\mathcal U_1 \cup \dots \cup \mathcal U_k$, where $k \leq \dim(X)$, such that for $ i \neq j \in \{1, \dots,  k\}$ the set $ \mathcal U_i$ is almost transverse to  $\mathcal U_j$. 

Furthermore, if $\mathcal V''$ is a UBS which is commensurate to $\mathcal V$ and is partitioned into a finite union of minimal UBS, say $\mathcal U'_1 \cup \dots \cup \mathcal U'_{k'}$, which are pairwise almost transverse, then $k = k'$ and, up to reordering, the set $\mathcal U'_i $ is commensurate to $\mathcal U_i$ for all $i$.\qed
\end{prop}

Following Mark Hagen  \cite{Hagen}, the \textbf{simplicial boundary} of $X$, denoted by $\partial_\triangle X$, is defined (when $X$ is finite-dimensional) as the abstract simplicial complex whose underlying poset is the set of commensuration classes of UBS, with the natural order relation induced by inclusion. Its vertices thus correspond to the commensuration classes of minimal UBS, and two vertices are adjacent if they are represented by two UBS which are almost transverse to one another.   The set of simplices of $\partial_\triangle X$ is denoted by $\mathfrak S X$. 

The following observation, which is implicit in \cite{Hagen} (see  the proof of Lemma~3.7 in loc. cit.), characterizes the minimal UBS. 

\begin{lemma}\label{lem:InsepClosure}
Assume that $X$ is finite-dimensional. 
Let $h_0 \supsetneq h_1 \supsetneq \dots$ be an infinite chain of halfspaces. Then the inseparable closure of $\{\hat h_i \; | \; i \geq 0\}$ is a minimal UBS. Moreover, every minimal UBS contains a (necessarily cofinite) UBS arising in this way. 
\end{lemma}

\begin{proof}
Let   $\mathcal V = \{\hat h_i \; | \; i \geq 0\}$ and $\mathcal U$ be the inseparable closure of $\mathcal V$. It is clear that $\mathcal U$ is infinite and inseparable. Observe moreover that each hyperplane in $\mathcal U$ separates $\hat h_0$ from $\hat h_i$ for some sufficiently large $i$. This implies that  $\mathcal U$  is unidirectional and does not contain any facing triple. Thus $\mathcal U$ is a UBS. 
Proposition~\ref{prop:UBS} implies that $\mathcal U$ must be minimal. 

That every minimal UBS arises in this way follows since  any UBS contains an infinite set of pairwise disjoint hyperplanes by the finite-dimensionality of $X$. 
\end{proof}

We now briefly recall that definition of the Roller compactification, following Martin Roller \cite{Roller}. 
A section $\alpha \colon \mathfrak W(X) \to \mathfrak H(X)$ of the map $h \mapsto \hat{h}$ is called an \textbf{ultrafilter} 
if for every finite set $\mathcal F \subset \mathfrak W(X)$, the intersection $\bigcap_{\hat h \in \mathcal F} \alpha(\hat h)$ is non-empty. 
An ultrafilter is \textbf{principal} if $\bigcap_{\hat h \in \mathfrak W(X)} \alpha(\hat h)$ is non-empty, 
in which case that intersection contains a unique vertex. 
Conversely,  every vertex $v$ of $X$ gives rise to a unique ultrafilter, also denoted by $v$, 
which maps each hyperplane $\hat h$ to the unique halfspace bounded by $\hat h$ and containing $v$. 
We identify henceforth each vertex with the corresponding principal ultrafilter. 
The collection of all ultrafilters is denoted by $\overline X$. 
The subset of non-principal ultrafilters is  denoted by $\partial X$. 
With respect to the topology of pointwise convergence, the set $\overline X = X^{(0)} \cup \partial X$  is compact; 
it is a compactification of the (discrete) set of principal ultrafilters $X^{(0)}$. 
The set  $\overline X$ (resp. $ \partial X$) is called the \textbf{Roller compactification} (resp. \textbf{Roller boundary}) of $X$. 
The following observation provides a link between the Roller boundary and the simplicial boundary.

\begin{lemma}\label{lem:RollerSimplicial}
Let $\alpha \in \partial X$. The following hold for all $x, y \in X^{(0)}$:
\begin{enumerate}
\item The set $\mathcal U(x, \alpha) = \{ \hat  h \in \mathfrak W(X) \; | \; x(\hat h) \neq \alpha(\hat h)\}$ is a UBS. 
\item The sets $\mathcal U(x, \alpha)$ and  $\mathcal U(y, \alpha)$ are commensurate.
\end{enumerate}
In particular, the map  $\Sigma \colon  \partial X \to \mathfrak S X$, 
associating to $\alpha$ the commensuration class of the UBS $\mathcal U(x, \alpha)$, where $x$ is a fixed vertex, is well-defined and $\Aut(X)$-equivariant. 
\end{lemma}

\begin{proof}
(1) The set $\mathcal U(x, \alpha)$ is infinite since otherwise $\alpha$ would be principal, because $x$ is so. That $\mathcal U(x, \alpha)$ is inseparable is clear. That $\mathcal U(x, \alpha)$ is unidirectional follows from the fact that $x$ is principal. Finally, given any facing triple of hyperplanes, the maps $x$ and $\alpha$ must coincide on at least one of them. Hence $\mathcal U(x, \alpha)$ is a UBS. 

\medskip \noindent
(2) Any hyperplane in the symmetric difference $\mathcal U(x, \alpha) \triangle \mathcal U(y, \alpha)$ separates $x$ from $y$. Therefore $\mathcal U(x, \alpha) \triangle \mathcal U(y, \alpha)$  is finite. 
\end{proof}

The final ingredient needed for the proof of Theorem~\ref{thm:RollerStab} is the following result, due to Michah Sageev. 

\begin{prop}\label{prop:Sageev}
Assume that $X$ is finite-dimensional. Let $\Gamma \leq \Aut(X)$ be a finitely generated group acting without any fixed point on $X$. Then there exists $\gamma \in \Gamma$ and $h \in \mathfrak H(X)$ such that $\gamma h \subsetneq h$. In particular a (possibly infinitely generated) group $\Lambda$ is locally $X$-elliptic if and only if every element of $\Lambda$ has a fixed point in $X$.
\end{prop}

\begin{proof}
This follows from the proof of Theorem 5.1 in \cite{Sageev_95}. 
\end{proof}

\begin{proof}[Proof of Theorem~\ref{thm:RollerStab}]
If $\alpha \in X^{(0)}$, the desired conclusion is clear. We suppose henceforth that $\alpha$ belongs to $\partial X$. Lemma~\ref{lem:RollerSimplicial} then provides a $k$-simplex  $\sigma = \Sigma(\alpha)$ fixed by $\Gamma = 	\Stab_{\Aut(X)}(\alpha)$. Note that $k+1 \leq \dim(X)$ by Proposition~\ref{prop:UBS}. We denote the vertices of $\sigma$ by $v_0, \dots, v_k$. For each $j$, the stabiliser $\Stab_\Gamma(v_j)$ is of finite index in $\Gamma$, and commensurates any UBS representing $v_j$. We denote by $\chi_j$ the transfer character associated to this commensurating action by means of Proposition~\ref{prop:ModularChar}. Thus the sum $\bigoplus_{j=0}^k \chi_j$ is a homormorphism to $\ZZ^{k+1}$ which is defined on a finite index subgroup of $\Gamma$. We denote its kernel by $\Gamma^0$, and claim that $\Gamma^0  $ is locally $X$-elliptic. The desired conclusion follows from that claim. 

By Proposition~\ref{prop:Sageev}, it suffices to show that every element of  $\Gamma^0$ has a fixed point. Suppose for a contradiction that an element $g \in  \Gamma^0$ has none. Applying Proposition~\ref{prop:Sageev} to the cyclic group generated by $g$, we find   a halfspace $h$ and a positive integer $n$ such that $g^n h \subsetneq h$. Set $\hat h_i =  g^{ni} \hat h$ for all $i \in \ZZ$. Since $g$ fixes $\alpha$, the collection $\{\alpha(\hat h_i) \; | \; i \in \ZZ\}$ is $\langle g^n \rangle$-invariant, and must therefore be a chain. Upon replacing $g$ by $g^{-1}$, we may assume that $\alpha(\hat h_0) \supsetneq \alpha(\hat h_1)$. Let $x \in \alpha(\hat h_0) \setminus \alpha(\hat h_1)$ be a vertex. In particular $x(\hat h_i) \neq \alpha(\hat h_i)$ for all $i >0$. It follows that the UBS $\mathcal U(x, \alpha)$, which represents $\sigma = \Sigma(\alpha)$, contains  $\hat h_i$ for all $i>0$. We now apply Proposition~\ref{prop:UBS} to $\mathcal U(x, \alpha)$. This yields a finite set of pairwise almost transverse minimal UBS $\mathcal U_0, \dots, \mathcal U_k$ contained in $\mathcal U(x, \alpha)$, each representing a vertex $v_j$ of $\sigma$, and such that the union $\bigcup_{j=0}^k \mathcal U_j$ is cofinite in $\mathcal U(x, \alpha)$. Since the hyperplanes $\hat h_i$ have pairwise empty intersection, we deduce moreover from Proposition~\ref{prop:UBS} that there is some $j \in \{0, \dots, k\}$ such that $\hat h_i \in \mathcal U_j$ for all  $i$ larger than some fixed $I$. By Lemma~\ref{lem:InsepClosure}, we may assume that $\mathcal U_j$ is the inseparable closure of $\{\hat h_i \; | \; i > I\}$.  Since $g^n(\hat h_i) = \hat h_{i+1}$ for all $i$, we infer that $g^n$ maps $\mathcal U_j$ properly inside itself, thereby implying that $0 \neq \chi_j(g^n) = n \chi_j(g)$.  This contradicts the fact that $g \in \Gamma^0 \leq \Ker(\chi_j)$. 
\end{proof}

\subsection*{Acknowledgement} 
P-E.C. enthusiastically thanks Yves de Cornulier for pleasant discussions and useful comments on this appendix.


\newcommand{\etalchar}[1]{$^{#1}$}
\def\cprime{$'$}

\end{document}